%% file: main.tex
\documentclass[reqno,centertags,12pt]{amsart}

\usepackage[normalem]{ulem}

\usepackage{amsmath,amsthm}
\usepackage{amssymb,amsbsy}
\usepackage{bbm}
\usepackage{tikz}
\usepackage{enumerate}
\usepackage{mleftright}
\usepackage{xcolor}
\usepackage{color}
\usepackage{hyperref}
\usepackage{cancel}
\usepackage{amsfonts}
\usepackage{graphicx}
\usepackage{bookmark}
\usepackage{longtable}
\usepackage{etoolbox}

\long\def\metanote#1#2{{\color{#1}\
		\ifmmode\hbox\fi{\sffamily\mdseries\upshape [#2]}\ }}

\patchcmd{\subsubsection}{\itshape}{\bfseries}{}{}

-\marginparwidth \oddsidemargin -4mm \evensidemargin \oddsidemargin

\makeatletter
\newcommand\xleftrightarrow[2][]{%
	\ext@arrow 9999{\longleftrightarrowfill@}{#1}{#2}}
\newcommand\longleftrightarrowfill@{%
	\arrowfill@\leftarrow\relbar\rightarrow}
\makeatother

\makeatletter
\newcommand{\xRightarrow}[2][]{\ext@arrow 0359\Rightarrowfill@{#1}{#2}}
\makeatother

\newcounter{smalllist}

\allowdisplaybreaks
\numberwithin{equation}{section}

\newcommand{\al}{\alpha}
\newcommand{\be}{\beta}
\newcommand{\ga}{\gamma}
\newcommand{\Ga}{\Gamma}
\newcommand{\de}{\delta}
\newcommand{\De}{\Delta}
\newcommand{\ep}{\epsilon}
\newcommand{\ta}{\theta}
\newcommand{\Ta}{\Theta}
\newcommand{\ka}{\kappa}
\newcommand{\la}{\lambda}
\newcommand{\La}{\Lambda}
\newcommand{\si}{\sigma}
\newcommand{\Si}{\Sigma}
\newcommand{\om}{\omega}
\newcommand{\Om}{\Omega}

\newcommand{\rb}[1]{\left(#1\right)}
\newcommand{\sqb}[1]{\left[#1\right]}
\newcommand{\cb}[1]{\left\{#1\right\}}
\newcommand{\abs}[1]{\left|#1\right|}
\newcommand{\norm}[1]{\left\|#1\right\|}
\newcommand{\inn}[1]{\left\langle #1 \right\rangle}

\newcommand{\scb}[1]{\{#1\}}

\newcommand{\mrb}[1]{\mleft(#1\mright)}
\newcommand{\msqb}[1]{\mleft[#1\mright]}
\newcommand{\mcb}[1]{\mleft\{#1\mright\}}

\newcommand{\tr}{\mathrm{tr}}

\newcommand{\mfk}{\mathfrak}
\newcommand{\mcl}{\mathcal}
\newcommand{\mbe}{\mathbb{E}}

\newcommand{\mbo}{\mathbbm{1}} 
\newcommand{\mcA}{\mathcal{A}}
\newcommand{\mcB}{\mathcal{B}}
\newcommand{\mcC}{\mathcal{C}}
\newcommand{\mcD}{\mathcal{D}}
\newcommand{\mcF}{\mathcal{F}}
\newcommand{\mcG}{\mathcal{G}}

\newcommand{\mcM}{\mathcal{M}}
\newcommand{\mcP}{\mathcal{P}}
\newcommand{\mcS}{\mathcal{S}}
\newcommand{\mcW}{\mathcal{W}}

\newcommand{\mcc}{\mathcal{C}}
\newcommand{\mcg}{\mathcal{G}}
\newcommand{\mcp}{\mathcal{P}}
\newcommand{\mA}{\mathcal{A}}
\newcommand{\mB}{\mathcal{B}}

\newcommand{\mbz}{\mathbb Z}
\newcommand{\mbn}{\mathbb N}

\newcommand{\mbr}{\mathbb R}

\newcommand{\f}{\frac}
\newcommand{\drv}[2]{\frac{d #1}{d #2}}
\newcommand{\drvev}[2]{\left.\frac{d #1}{d #2}\right|}

\newcommand{\bs}{\boldsymbol}
\newcommand{\Bx}{\boldsymbol{x}}
\newcommand{\By}{\boldsymbol{y}}
\newcommand{\Bc}{\boldsymbol{c}}
\newcommand{\BX}{\boldsymbol{X}}
\newcommand{\BY}{\boldsymbol{Y}}

\newcommand{\Bmu}{\boldsymbol{\mu}}
\newcommand{\Bnu}{\boldsymbol{\nu}}
\newcommand{\Brho}{\boldsymbol{\rho}}
\newcommand{\Bbr}{\boldsymbol{\bar{\rho}}}
\newcommand{\BMM}{\boldsymbol{\mathcal{M}}}
\newcommand{\BHA}{\boldsymbol{\hat{\mathcal{A}}}}

\newcommand{\br}{\bar{\rho}}
\newcommand{\bX}{\bar{X}}

\newcommand{\Phitn}{\Phi_t^n}
\newcommand{\Psitn}{\Psi_t^n}

\newcommand{\lafn}{\la_n^F}
\newcommand{\lagn}{\la_n^G}

\newcommand{\mbs}{\mathbb S}
\newcommand{\RPC}{\mathrm{RPC}}
\newcommand{\Exp}{\mathrm{Exp}_p}

\newcommand{\Expp}{\mathrm{Exp}_p}
\newcommand{\CEXP}{{\mathrm{Exp}_p}}
\newcommand{\CTPC}{C([0,T];\mcP_p(\Pi))}
\newcommand{\oN}{\otimes N}
\newcommand{\GPZP}{\mathcal{G}_p^0(\Pi)}
\newcommand{\PPP}{\mathcal{P}_p(\Pi)}
\newcommand{\MFG}{\mathcal{A}:\mathcal{P}_p(\Pi)\to\mathcal{G}_p^0(\Pi)}

\newcommand{\GCZP}{\mathcal{G}_p^0(\Pi)}

\newcommand{\Expt}{\mathrm{Exp}_2}

\newcommand{\GCZR}{\mathcal{G}_2^0(\mathbb{R}^d)}
\newcommand{\PCR}{\mathcal{P}_2(\mathbb{R}^d)}
\newcommand{\SPSD}{\mathcal{S}_d^{\ge 0}}

\newcommand{\WLA}{\mathcal{W}_{\Lambda}}
\newcommand{\BWD}{\mathcal{W}_{\mathcal{S}}}

\newcommand{\GLVR}{\mathcal{G}_2^{\Lambda}(\mathbb{R}^d)}

\newcommand{\LVM}{\Lambda_2(\mathbb{R}^d)}
\newcommand{\PCRN}{\mathcal{P}_2((\mathbb{R}^d)^N)}

\newcommand{\refA}{(\hyperref[hypoA]{A})}
\newcommand{\refH}{(\hyperref[hyp:hilbert-features]{H})}
\newcommand{\refLa}{(\hyperref[hypoLa]{$\La$})}
\newcommand{\refAp}{(\hyperref[hypoA']{A${}^\prime$})}
\newtheorem{theorem}{Theorem}[section]
\newtheorem{proposition}[theorem]{Proposition}
\newtheorem{lemma}[theorem]{Lemma}
\newtheorem{corollary}[theorem]{Corollary}

\theoremstyle{definition}
\newtheorem{definition}[theorem]{Definition}
\newtheorem{notation}[theorem]{Notation}

\newtheorem{example}[theorem]{Example}

\newtheorem{remark}[theorem]{Remark}

\newtheorem*{hypothesisA}{Hypothesis (A)}
\newtheorem*{hypothesisH}{Hypothesis (H)}
\newtheorem*{hypothesisLa}{Hypothesis ($\La$)}
\newtheorem*{hypothesisAp}{Hypothesis (A${}^\prime$)}
\begin{document}
\title[]{Abstract Formulation of Mean-Field Models and Propagation of Chaos}

\author{Tau Shean Lim}
\address{\noindent Department of Mathematics \\ Xiamen University Malaysia}
\email{taushean.lim@xmu.edu.my}

\author{Chao Dun Teoh}
\address{\noindent Department of Mathematics \\ The Pennsylvania State University}
\email{cd.teoh@psu.edu}

%=============================================================================
\input{Abstract}
\maketitle
\tableofcontents
%=============================================================================
\input{Section1_v2}

\subsection{List of symbols and notations}\,
\begingroup
    \footnotesize
    \input{List_of_Symbols}
\endgroup

\input{Section2}

\input{Section3}
\input{Section4}
\input{Section5}
\input{Section6}
\input{SectionAppen}
\section*{Funding}
This work was supported by the Xiamen University Malaysia Research Fund (XMUMRF) under Grant No.~XMUMRF/2022-C9/IMAT/0021.

\bibliography{refs}
\bibliographystyle{abbrv}

\end{document}

%% file: Abstract.tex
\begin{abstract}
    %Mean-field $N$-particle systems involve particles whose evolution is governed by their mean field, an averaged effect that captures the collective interactions among all particles. A classical example is the McKean-Vlasov diffusion, which is described by a system of weakly coupled stochastic differential equations (SDEs). A defining characteristic of McKean-Vlasov diffusion is the \emph{propagation of chaos}---a phenomenon where the dynamics of individual particles become asymptotically independent as the number of particles $N\to\infty$, due to the diminishing influence of individual interactions.

    We formulate a unified abstract framework for mean-field particle systems and their macroscopic limits from the perspective of differential equations and functional analysis. Combining semigroup theory with optimal transport, we study measure-valued evolutions directly through their generators, allowing different classes of mean-field models to be treated within a common structure. Specifically, the framework involves (1) formulating mean-field models abstractly using semigroups and generators; (2) establishing well-posedness of the corresponding nonlinear mean-field evolution equations in a class of solutions stable under approximation; and (3) proving quantitative propagation of chaos for the associated $N$-particle systems. The central tool is the \emph{transport derivative}, which measures the infinitesimal variation of transport costs between Markov evolutions. We establish transport-stability criteria expressed through local generator inequalities and show how these criteria underpin both nonlinear well-posedness and particle approximation. When the generator's dependence on the law is controlled by averaged Hilbert-valued features, we obtain an $O(N^{-1})$ sampling contribution to the quadratic Wasserstein estimate through a variance argument. As an application, we obtain quantitative propagation of chaos for L\'evy-type mean-field systems with state- and law-dependent drift, diffusion, and jump kernels under intrinsic transport conditions. The admissible jump kernels extend beyond those generated by multiplicative L\'evy noise, demonstrating the scope of the abstract framework.

    \vspace{6ex}
    
    \noindent {\bf Keywords}: Mean-field models; propagation of chaos; transport derivative; optimal transport; Markov semigroups; L\'evy-type dynamics.
\end{abstract}

%% file: Section1_v2.tex
\section{Introduction}

An \emph{$N$-particle system} is typically modelled as a Markov process $\BX_t=(X_t^1,\ldots,X_t^N)$ on $\Pi^N$, where $\Pi$ is the single-particle state space. Informally, a \emph{mean-field $N$-particle system} is one in which each particle evolves under an averaged influence of the population, commonly represented by the empirical measure
\begin{align*}
    \mu(\BX_t)=\frac{1}{N}\sum_{i=1}^N\de_{X_t^i}.
\end{align*}
Such systems are often permutation-invariant: their law is unchanged by relabelling the particles.

As the number of particles tends to infinity, suitably chaotic initial data are expected to evolve into a system whose finite collections of particles remain asymptotically independent, with a common law governed by a nonlinear mean-field equation. This property is known as \emph{propagation of chaos}. The theory originated in the work of Kac \cite{kac1956foundations}; classical and modern accounts can be found in \cite{sznitman1991topics,hauray2014kac,Chaintron_2022_1,Chaintron_2022_2}. More broadly, surveys of mean-field limits and related large-particle asymptotics include \cite{jabin2014review,golse2016dynamics,jabin2017mean}.

%=============================================================================
\subsection{McKean--Vlasov diffusion and the coupling method}\label{sec:intro-mv-coupling}

A canonical example of a mean-field system is the \emph{McKean--Vlasov diffusion}. Let $\mcP(\mbr^d)$ be the space of Borel probability measures on $\mbr^d$ and $\mcM_d(\mbr)$ the space of $d\times d$ real matrices. Given measurable coefficient fields 
\begin{align*}
    b:\mbr^d\times\mcP(\mbr^d)\to\mbr^d,
    \qquad
    \si:\mbr^d\times\mcP(\mbr^d)\to\mcM_d(\mbr),
\end{align*} 
consider the $(\mbr^d)^N$-valued process $\scb{\bs{X}_t}_{t\ge0}=\cb{(X_t^1,X_t^2,\cdots,X_t^N)}_{t\ge0}$ defined by the weakly coupled system of stochastic differential equations
\begin{align}\label{eq:mckean-vlasov-diffusion}
    dX_t^i
    =
    b(X_t^i,\mu(\bs X_t))\,dt
    +
    \si(X_t^i,\mu(\bs X_t))\,dB_t^i,
    \qquad 1\le i\le N,
\end{align}
where $\scb{\scb{B_t^i}_{t\ge0}}_{1\le i\le N}$ are independent standard Brownian motions.  Physically, each particle interacts with an average field generated by all particles in the system, each contributing a weight of order $1/N$.

The corresponding limiting dynamics are described by the nonlinear McKean--Vlasov process $\scb{\bX_t}_{t\ge0}$ satisfying
\begin{align}\label{eq:mv-mf-sde}
    d\bX_t
    =
    b(\bX_t,\br_t)\,dt
    +
    \si(\bX_t,\br_t)\,dB_t,
    \qquad
    \br_t=\mathrm{law}(\bX_t).
\end{align}
The process is nonlinear in the sense that its coefficients depend on its own time-dependent law. Propagation of chaos asserts, in particular, that finite collections of particles in \eqref{eq:mckean-vlasov-diffusion} become asymptotically independent and are governed by the limiting McKean--Vlasov dynamics \eqref{eq:mv-mf-sde}. We refer to $\scb{\bX_t}_{t\ge0}$ as the \emph{limiting process} and to $\scb{\br_t}_{t\ge0}$ as the corresponding \emph{mean-field law}.

A widely used method for proving quantitative propagation of chaos for McKean--Vlasov diffusions is based on \emph{synchronous coupling}. The method, originating in the classical work of McKean \cite{McKean1969} and presented systematically by Sznitman \cite{sznitman1991topics}, was first formulated for constant diffusion coefficients and subsequently extended to more general coefficients; see, for example, \cite[Proposition 2.3]{jourdain1998propagation}. One constructs copies $\{\bar X^i\}_{i=1}^N$ of the limiting process \eqref{eq:mv-mf-sde} that are independent of one another, with each $\bar X^i$ driven by the same Brownian motion $B^i$ and initialized at the same value as $X^i$. The resulting family $\scb{\bar \BX_t^N}_{t\ge0}
=\{\{\bar X_t^i\}_{t\ge0}\}_{1\le i\le N}$ provides a pathwise coupling between the interacting particle system and an independent system of limiting particles.

For instance, suppose that the drift has the average-interaction form $b(x,\mu)=\int B(x,y)\mu(dy)$, with $B$ globally Lipschitz, and that the diffusion coefficient $\si$ is law-independent and globally Lipschitz. With identical i.i.d. initial data of finite second moment for the synchronously coupled systems, It\^o calculus and Gr\"onwall's lemma yield a bound of the form
\begin{align}\label{eq1.15:temp}
    \mathbb E\left[\sup_{t\in[0,T]}
    |X_t^i-\bar X_t^i|^2\right]
    \le
    \frac{C(b,\si,T)}{N}.
\end{align}
This is a typical \emph{pathwise propagation of chaos} estimate: it controls the distance between coupled trajectories on a common probability space. By taking expectations with respect to the coupling, such estimates yield corresponding bounds for the Wasserstein distance between the time-marginal laws.

%=============================================================================
\subsection{Other mean-field models and methods}

McKean--Vlasov diffusions form only one class of mean-field systems. Other examples include chemical reaction models \cite{albi2019leader,peter1986law,lim2020quan}, mean-field jump processes
\cite{DAWSON1991293,leonard1995large}, L\'{e}vy-type mean-field systems
\cite{fournier2013pathwise,frikha2021well,Cavallazzi2023,Kolokoltsov2009},
and Nanbu-type particle systems for the Boltzmann equation
\cite{fournier2016nanbu,graham1997stochastic,meleard1996asymptotic}.
These models may involve jumps, collision mechanisms, or more general interaction structures for which the synchronous coupling described above is no longer natural or readily available. Their diversity motivates methods that compare the evolving probability laws without requiring a prescribed coupling of sample paths.

Several analytical approaches to propagation of chaos operate directly at the level of probability laws. In particular, the relative entropy methods developed by Jabin--Wang \cite{jabin2016mean} and Lacker \cite{lacker2018} compare particle distributions with tensorized mean-field laws and have yielded powerful quantitative estimates; see also \cite{Huang2025,jabin2017mean,Lacker2023,lim2020quan}. Their application depends on suitable control of the evolution of relative entropy and the associated interaction terms. These approaches illustrate how quantitative propagation of chaos can be studied through analytical estimates on distributions. In the present work, we pursue a transport-based approach formulated through Markov generators and semigroups.

%=============================================================================
\subsection{An abstract generator and semigroup framework}

In the present work, we formulate mean-field systems directly at the level of generators. Formally, a \emph{mean-field generator} on a state space $\Pi$ is a map
\begin{align*}
    \mathcal{A}:\mathcal{P}(\Pi)\to\mathcal{G}(\Pi),
\end{align*} 
where $\mathcal{P}(\Pi)$ and $\mathcal{G}(\Pi)$ denote, respectively, the space of probability measures and the space of Feller generators on $\Pi$. For a given mean-field measure $\mu\in\mathcal{P}(\Pi)$, the generator $\mathcal{A}(\mu)$ describes the infinitesimal evolution of a single particle under the influence of a surrounding population distributed according to $\mu$. The McKean--Vlasov diffusion is a particular instance of this framework, and further examples are discussed in Section \ref{chap6}. Similar generator-based viewpoints have appeared previously, for instance, in \cite{Chaintron_2022_1}.

A broader semigroup framework for propagation of chaos was developed in \cite{mischler2013kac,mischler2013kac2,mischler2015new}, building on ideas originating in \cite{grunbaum1971propagation}. In that approach, the $N$-particle process induces an \emph{empirical semigroup} $\{\widehat{T}_{N,t}\}_{t\ge0}$ acting on $C_b(\widehat{\mathcal{P}}_N(\Pi))$, where $\widehat{\mathcal{P}}_N(\Pi)$ is the space of empirical measures of $N$ particles. Propagation of chaos is derived from suitable consistency and stability properties of this semigroup.

%=============================================================================

\subsection{A law-level approach to quantitative propagation of chaos}

The classical approach described in Section \ref{sec:intro-mv-coupling} relies on a \emph{pathwise coupling}: coupled trajectories are controlled on a common probability space, yielding estimates on their time-marginal laws. Our approach instead works through a \emph{timewise coupling}, understood as an optimal transport comparison between the probability laws at each time, without requiring the couplings at different times to arise from a common coupling of trajectories. We compare the joint law of the interacting particle system with the tensorized mean-field law in the Wasserstein framework, uniformly over a finite time interval. This formulates propagation of chaos as a deterministic transport-stability problem between two measure-valued evolutions. By deriving the relevant stability estimates at the level of generators, we accommodate an abstract setting without requiring a prescribed coupling of sample paths.

Our objectives are to:
\begin{enumerate}
	\item formulate mean-field models and their associated nonlinear evolution equations at the level of generators;
	\item establish well-posedness of these equations in a class of solutions stable under approximation;
	\item prove quantitative propagation of chaos through an estimate of the form
	\begin{align}\label{eq:pointwise}
		\sup_{t\in[0,T]}
		\bs{\mathcal{W}}_p^p(\Brho_t^N,\br_t^{\otimes N})
		\le
		\bs{\mathcal{W}}_p^p(\Brho_0^N,\br_0^{\otimes N})e^{KT}
		+
		\ep(N,T),
	\end{align}
\end{enumerate}
Here, $\Brho_t^N$ is the joint law of the interacting particle system, $\br_t$ is the mean-field law, and $\bs{\mcW}_p$ is the Wasserstein-$p$ distance associated with the normalized product metric on $\Pi^N$. The constant $K\ge0$ is independent of $N$, and the sampling error $\ep(N,T)$ tends to zero as $N\to\infty$ under the stated moment and empirical-measure convergence assumptions.

The estimate \eqref{eq:pointwise} separates the propagated initial discrepancy from the empirical sampling error. It is a \emph{timewise} estimate: it controls the discrepancy between time-marginal laws uniformly in $t$, rather than the discrepancy between entire coupled trajectories. Although this provides less information than a pathwise estimate, it is sufficient for the propagation of chaos considered here and permits a common treatment of models specified through their generators. See Remark \ref{rem:pointvspath} for further discussion.

%=============================================================================
\subsection{The role of the transport derivative}

The main conceptual novelty of this work is the systematic use of the \emph{transport derivative} $\om_p$. Informally, this quantity measures the infinitesimal variation of the $p$-th Wasserstein transport cost between two Markov flows; its precise definition is given in Section \ref{chap3}. Although closely related quantities have appeared in specific settings \cite{alfonsi2018,KangLim2025OptimalMarkovianCouplings}, one of the purposes of the present work is to identify the transport derivative as a common analytical tool for abstract mean-field systems.

The motivation for this quantity can be understood by comparing the abstract framework with the classical coupling argument. For McKean--Vlasov diffusions, Lipschitz estimates on the coefficients $b$ and $\si$ yield stability estimates for coupled trajectories and hence for their laws. In a generator-level framework, we formulate the relevant stability directly in terms of transport costs between Markov flows. A suitable bound on the transport derivative therefore plays the role of coefficient Lipschitz assumptions in the SDE setting. This viewpoint also offers an advantage when an SDE representation is available: it measures regularity in the geometry of the underlying transport problem. In the quadratic setting, diffusion matrices are compared through the \emph{Bures--Wasserstein} distance, rather than through the \emph{Frobenius norm} of their chosen diffusion coefficients, and jump kernels through the \emph{L\'evy--Wasserstein distance}. These intrinsic comparisons can yield less restrictive stability conditions; see Section \ref{chap6}.

%=============================================================================
\subsection{Main contributions and application}

The central contribution of this work is a unified generator-level framework that connects infinitesimal transport bounds with mean-field evolution and quantitative propagation of chaos. The transport derivative and its dual provide criteria for controlling the Wasserstein stability of Markov flows through local inequalities involving their generators. These criteria form the common analytical basis for constructing the nonlinear mean-field evolution and estimating the discrepancy between the interacting particle law and the tensorized mean-field law.

More precisely, the main contributions are:
\begin{itemize}
	\item A transport-stability theory for Markov flows, establishing equivalent formulations in terms of semigroup estimates, transport-derivative bounds, and dual generator inequalities. The theory accommodates general continuous transport costs with suitable moment growth; the mean-field results specialize to Wasserstein-$p$.
	\item A construction of nonhomogeneous linear and nonlinear mean-field evolutions by piecewise constant approximation, together with well-posedness in the corresponding class of stable solutions.
	\item A quantitative Wasserstein estimate for the associated mean-field $N$-particle systems, separating the propagated initial discrepancy from an empirical sampling error and yielding propagation of chaos.
	\item An application to L\'evy-type mean-field dynamics with state- and law-dependent drift, diffusion, and jump kernels, under conditions expressed through intrinsic transport distances between covariance matrices and L\'evy measures.
\end{itemize}

Under an additional structural condition, the framework also yields a sharper quantitative estimate. When the dependence of the generator on the law is controlled by averaged Hilbert-valued features, a variance identity replaces the empirical Wasserstein error and gives an $O(N^{-1})$ sampling contribution to the quadratic transport cost. This includes nonlinear dependence on averaged features and applies, in particular, to L\'evy-type models constructed from averaged drift and noise amplitudes; see Sections \ref{sec:hilbert-refinement} and \ref{sec6:averaged-amplitudes}.

The L\'evy application demonstrates a concrete benefit of the abstract formulation. Cavallazzi \cite{Cavallazzi2023} studies McKean--Vlasov equations driven by a fixed L\'evy process through a matrix-valued multiplicative coefficient, with quantitative propagation of chaos established in the additive-noise case. Our framework allows the jump kernel to vary with both the state and the current law without requiring it to be the linear pushforward of a fixed L\'evy measure. It therefore accommodates a broader class of jump mechanisms through a common transport-stability criterion, without requiring a prescribed coupling of sample paths. A detailed comparison of the assumptions and conclusions is given in Section \ref{chap6}.

%=============================================================================
\subsection{Organization of the paper}

The paper is organized as follows. Section \ref{chap2} introduces the necessary preliminaries on Feller semigroups, optimal transport, and Wasserstein spaces. Section \ref{chap3} develops the transport derivative and its dual for general continuous costs with suitable moment growth, and establishes the criteria connecting local generator inequalities with transport stability of Markov flows.

Section \ref{chap4} uses these criteria to construct nonhomogeneous linear and nonlinear mean-field evolutions by piecewise constant approximation and establish well-posedness in the class of stable solutions. Section \ref{chap5} formulates the associated $N$-particle systems and derives the quantitative Wasserstein estimate yielding propagation of chaos. Finally, Section \ref{chap6} specializes the framework to L\'evy-type mean-field dynamics, expresses the stability conditions through intrinsic transport distances between their coefficients, and obtains the corresponding well-posedness and propagation-of-chaos results.

%% file: List_of_Symbols.tex
%\section*{List of Symbols / Abbreviations}
%\addcontentsline{toc}{section}{List of Symbols / Abbreviations}

\newcommand{\sym}[2]{#1 &:& #2 \\[1ex]} 

%%WARNING YOU NEED LONGTABLE PACKAGE
%\footnotesize
\begin{longtable}{lll}
	\sym{$\oplus$}{Separable sum of functions, $(f\oplus g)(x,y)=f(x)+g(y)$}
	\sym{$\otimes$}{Tensor product (of functions, measures, operators, semigroups)}
	\sym{$\inn{\cdot,\cdot}$}{Natural pairing of $\mcP(\Pi)\times C_b(\Pi)$}
	\sym{$\|\cdot\|_\infty$}{Supremum norm on $C_b(\Pi)$}
	%\sym{$\|\cdot\|_{\mcl F}$}{Frobenius norm on matrices}
	\sym{$\xrightarrow[]{\mcW_p}$}{$\mcW_p$-metric convergence of measures}
    % Convergence on $\Pi$ in $c$-semimetric, or convergence of measures in $\mcC_c$-semimetric
    %\sym{$\mbo_B$}{Function that is $1$ on $B$, $0$ elsewhere}
	%%A
	\sym{$(b,a,\Ta)$}{L\'{e}vy triplet}
	\sym{$\mcA,\mcB$}{Feller generators on $\Pi$}
	\sym{$\{\mcA(\mu)\}_{\mu}$}{Mean-field generator}
    %\sym{$\{\mcA(x,\mu)\}_{x,\mu}$}{L\'evy-type Mean-field generator}
	%\sym{$\mcA^\nabla$}{Drift operator on $\mbr^d$}
	%\sym{$\mcA^\De$}{Diffusion operator on $\mbr^d$}
	%\sym{$\mcA^J$}{L\'{e}vy jump operator on $\mbr^d$}
	\sym{$\bs\mcA$}{Feller generator on $\Pi^N$}
	\sym{$\BHA=\BHA_N$}{$N$-particle generator (on $\Pi^N$) associated to a mean-field generator $\mcA$}
	\sym{$\aleph_N(\br)$}{$\int_{\Pi^N} \mcW_p^p(\mu(\By_1'),\br)\,d\br^{\otimes N}(\By)$}
	%%%C
	\sym{$c$}{Cost function on $\Pi^2$}
	\sym{$c_p$}{The $p$-th cost function $c_p = p^{-1}d^p$}
	\sym{$\Bc_p$}{Normalized $p$-th product cost on $\Pi^N$}
	\sym{$\mcC_c$}{$c$-optimal cost}
	%\sym{$\mcC_{\Bc}$}{$\Bc$-optimal cost}
	%\sym{$\mcC_p,p\ge 1$}{$p$-optimal cost}
	\sym{$C(\Pi)$}{Space of continuous functions on $\Pi$}
	\sym{$C_0(\Pi)$}{Space of continuous functions vanishing at infinity on $\Pi$}
	\sym{$C_0^\sim(\Pi)$}{Space of continuous functions $\phi:\Pi\to\mbr$ having a constant limit at infinity}
	\sym{$C_b(\Pi)$}{Space of bounded continuous functions}
	\sym{$C_p(\Pi)$}{Space of continuous functions with at most $p$-th-order growth}
	\sym{$C([0,T];\mcP_p(\Pi))$}{Space of continuous curves of $\mcP_p(\Pi)$-measures}
	% \sym{$\CEXP([0,T];\GCZP)$}{Space of $\Exp$-continuous curves of generators}
	%%%D
	\sym{$d$}{Metric on $\Pi$}
	\sym{$\bs d_p$}{Normalized $p$-th product metric on $\Pi^N$}
	\sym{$D(\mcA)$}{Domain of the generator $\mcA$}
	\sym{$D^\sim(\mcA)$}{Extended domain of the generator $\mcA$, $D(\mcA)+\mbr$}
	%\sym{$D^+F$}{(Right-hand upper) Dini derivative of $F$}
	\sym{$\de_x$}{Dirac delta measure at $x\in\Pi$}
	%%%E
	\sym{$\{e^{t\mcA}\}_{t\ge 0}$}{Feller semigroup generated by $\mA$}
	\sym{$\mathrm{Exp}_p(\al,\be)$}{$p$-exponential stability condition}
	%%%F
	\sym{$f,g,\phi,\psi$}{Real-valued functions on $\Pi$}
	\sym{$\bs F,\bs G,\bs\Phi,\bs\Psi$}{Real-valued functions on $\Pi^N$}
	%%G
	\sym{$\ga$}{Coupling of $\mu,\nu\in\mcp(\Pi)$}
	\sym{$\bs\ga$}{Coupling of $\bs\mu,\bs\nu\in\mcp(\Pi^N)$}
	\sym{$\Ga(\mu,\nu)$}{Set of all couplings of $\mu,\nu$}
	\sym{$\mcG(\Pi)$}{Class of Feller generators on $C_0(\Pi)$}
	\sym{$\mcG_p^0(\Pi)$}{Class of $\mcW_p$-continuous Feller generators on $\Pi$}
	%\sym{$\mcG_{\Bc}^0(\Pi^N)$}{Space of $\Bc$-continuous Feller generators on $\Pi^N$}
	%\sym{$\mcG_2^0(\mbr^d)$}{Space of $c$-continuous Feller generators on $\mbr^d$, where $c(x,y)=\frac{1}{2}|x-y|^2$}
    %\sym{$\GLVRnt$}{Space of L\'evy generators on $\mbr^d$}
    \sym{$\GLVR$}{Class of L\'evy generators on $\mbr^d$ with finite second moment}
	%%K
	%\sym{$K,L$}{Markov operators}
	%\sym{$\kappa,\la$}{Markov kernels}
	\sym{$\{\kappa_t\}_{t\ge 0}$}{Family of transition kernels}
        %%%L
    %\sym{$\LVMnt$}{Space of L\'evy measures on $\mbr^d$}
    \sym{$\LVM$}{Space of L\'evy measures on $\mbr^d$ with finite second moment}
	%%%M
	\sym{$\mu,\nu,\rho$}{Probability measures on $\Pi$}
	\sym{$\{\mu_t\}_{t\in[0,T]}$}{Curve of probability measures}
	\sym{$\mu(\Bx)$}{Empirical measure of $\Bx\in \Pi^N$}
	\sym{$\bs\mu,\bs\nu,\bs\rho$}{Probability measures on $\Pi^N$}
	\sym{$\om_c(\cdot,\cdot;\mcA,\mcB)$}{$c$-transport derivative between flows generated by $\mcA,\mcB$}
    \sym{$\om_c^*(\cdot,\cdot;\mcA,\mcB)$}{Dual $c$-transport derivative}
    \sym{$\om_p,\om_p^*$}{Transport derivative and its dual for $c=c_p=p^{-1}d^p$}
    \sym{$\bs\om_p=\bs\om_{p,N}$}{$p$-transport derivative on $\Pi^N$ for the normalized product cost}
   	%\sym{$\mcM_d(\mbr)$}{Space of $d\times d$ real-valued matrices}
	%%P
	\sym{$\mcP(\Pi)$}{Space of probability measures on $\Pi$}
	\sym{$\mcP_p(\Pi)$}{Space of probability measures with a finite $p$-th moment on $\Pi$}
	%\sym{$\mcP_{\Bc}(\Pi^N)$}{Space of probability measures with a finite $\Bc$-moment on $\Pi^N$}
	%\sym{$\mcP_p(\mbr^d)$}{Space of probability measures on $\mbr^d$ with finite $p$-th moment} 
	\sym{$(\Pi,d)$}{State space (locally compact Polish space)}
	%\sym{$(\Pi,c)$}{Semimetric space}
	\sym{$\Pi^N$}{$N$-fold product space of $\Pi$}
	%%%R
	\sym{$\{\br_t\}_{t\ge 0}$}{Solution of mean-field evolution problem}
	% \sym{$\RPC([0,T];\GCZP)$}{Space of (right-continuous) piecewise constant curves of $\GCZP$-generators}
        %%%S
        \sym{$\SPSD(\mbr)$}{Space of symmetric nonnegative semidefinite $d\times d$ real-valued matrices}
	%%%T
    %\sym{$\Si(x,\mu,\nu)$}{Nonnegative-valued function on $\Pi\times\PCP^2$ in Hypothesis \refApp}
	%\sym{$\mathrm{tr}(a)$}{Trace of the matrix $a$}
	%%%W
	\sym{$\mcW_p$}{Wasserstein-$p$ metric on $\mcP_p(\Pi)$}
    \sym{$\bs\mcW_p=\bs\mcW_{p,N}$}{Wasserstein-$p$ metric for the normalized product cost on $\Pi^N$}
        \sym{$\BWD$}{Bures-Wasserstein distance between nonnegative definite matrices}
        \sym{$\WLA$}{L\'evy-Wasserstein distance between L\'evy measures}
        %\sym{$\WGLV$}{$\WGLV(\mA,\mB)^2= \frac 12 |b-\tilde b|^2+ \BWD(a,\tilde a)^2+\WLA(\Ta,\tilde \Ta)^2$}
	\sym{$x,y,z$}{Variables on $\Pi$}
	\sym{$\Bx,\By,\bs{z}$}{Variables on $\Pi^N$}
	\sym{$\Bx_k'$}{$k$-th truncated variable of $\Bx$}
	\sym{$\{X_t\}_{t\ge 0},\{Y_t\}_{t\ge 0}$}{Feller or Markov processes on $\Pi$}
	\sym{$\{\bs{X}_t\}_{t\ge 0},\{\bs{Y}_t\}_{t\ge 0}$}{Feller or Markov processes on $\Pi^N$}
	%\sym{$\Xi(x,\mu,\nu)$}{Nonnegative-valued function on $\Pi\times\PCP^2$ in Hypothesis \refAp}
	\sym{$\zeta_\be(t)$}{$\frac 1\be(e^{\be t}-1)$ (or $=t$ if $\be=0$)}
\end{longtable}
%%------------------------------------------------------------------------------------------------------------------------
% end of Part 6: List of symbols
%%------------------------------------------------------------------------------------------------------------------------

%% file: Section2.tex
\section{Preliminaries}\label{chap2}

This section collects the basic analytical and probabilistic tools used throughout the present work. We first recall Markov kernels, Feller semigroups, and their generators, together with their connection to Feller processes. We then review the basic notions of optimal transport and Wasserstein spaces. For further details, we refer to \cite{ethier2009markov,liggett2010continuous,MR3050280,MR2459454}.

Throughout the present work, the state space $(\Pi,d)$ is always assumed to be a \emph{locally compact Polish space}. We will frequently use the following standard notation for spaces defined on $\Pi$:
\begin{itemize}
	\item $C(\Pi)$: the space of continuous (real-valued) functions on $\Pi$;
	\item $C_b(\Pi)$: the space of bounded continuous functions $\phi:\Pi\to\mbr$, which is a Banach space equipped with the \emph{supremum norm}
	\begin{align*}
		\norm{\phi}_{\infty} = \sup_{x\in\Pi}\abs{\phi(x)}<\infty;
	\end{align*}
	\item $C_0(\Pi)$: the space of continuous functions $\phi:\Pi\to\mbr$ that vanish at infinity, that is, for all $\ep>0$, there exists a compact set $K\subset \Pi$ such that
	\[|\phi(x)|<\ep,\quad\mbox{for $x\notin K$.}\]
	$(C_0(\Pi),\norm{\cdot}_{\infty})$ is a closed subspace of $C_b(\Pi)$, and is hence also a Banach space;
	\item $C_0^\sim(\Pi)$: the space of continuous functions $\phi:\Pi\to\mbr$ having a constant limit at infinity, that is, there exists $a\in\mbr$ such that $\phi-a\in C_0(\Pi)$.
	\item $\mfk{B}(\Pi)$: the Borel $\si$-algebra on $\Pi$, which is the smallest $\si$-algebra on $\Pi$ containing all open sets of $\Pi$;
	\item $\mcP(\Pi)$: the space of (Borel) probability measures on $\Pi$;
	\item $\mcG(\Pi)$: the space of Feller generators on $\Pi$, see \eqref{eq:prob-gen}.
\end{itemize}
We write $\inn{\cdot,\cdot}:\mcP(\Pi)\times C_b(\Pi)\to \mbr$ for the \emph{natural pairing} between a probability measure and a continuous function, which is given by
\begin{align*}
	\inn{\mu,\phi} = \int_\Pi \phi(x)\,d\mu(x).
\end{align*}
This notion is naturally extended to those unbounded measurable functions $\phi$ such that the Lebesgue integral is well-defined. Given functions $\phi,\psi\in C(\Pi)$, we define $\phi\oplus\psi\in C(\Pi^2)$ by
\begin{align*}
	(\phi\oplus\psi)(x,y)= \phi(x)+\psi(y).
\end{align*}

%=============================================================================
\subsection{Markov kernels}
A \emph{continuous Markov kernel} is a map $\kappa:\Pi\to\mcP(\Pi)$ such that $x\mapsto\kappa(x)$ is continuous with respect to the weak topology on $\mcP(\Pi)$; equivalently, for every $\phi\in C_b(\Pi)$, the map
\[
x\mapsto\inn{\ka(x),\phi}=\int_\Pi\phi(z)\,\kappa(x,dz)
\]
is continuous.
A Markov kernel can also be viewed as a map $\ka:\Pi\times\mfk{B}(\Pi)\to[0,1]$, where $\ka(x,E)\in[0,1]$ is the probability of the event $E\in\mfk{B}(\Pi)$, measured by the probability measure $\ka(x)$. A Markov kernel naturally defines a bounded operator $K:C_b(\Pi)\to C_b(\Pi)$ with $\norm{K}_{\mathrm{op}}\le1$ by
\begin{align}\label{eq2.2:temp}
	(K\phi)(x) = \inn{\ka(x),\phi} = \int_\Pi \phi(z)\,\ka(x,dz).
\end{align}
The operator $K$ will be called the \emph{Markov operator associated to the kernel $\ka$}. Note that the weak continuity of $\ka$ implies $K\phi\in C_b(\Pi)$ for all $\phi\in C_b(\Pi)$.

Given a Markov operator \(K\) on \(C_b(\Pi)\), we define its \emph{right action} on probability measures \(\mu\in\mcp(\Pi)\) by duality:
\[
\inn{\mu K,\phi}
=
\inn{\mu,K\phi},
\qquad
\phi\in C_b(\Pi).
\]
By the Riesz--Markov representation theorem, \(\mu K\) is uniquely realized as a probability measure. In particular, if \(K\) is the Markov operator associated with a Markov kernel \(\ka\), so that \((K\phi)(x)=\inn{\ka(x),\phi}\), then
\[
\ka(x)=\de_xK.
\]

\subsection{Feller semigroups and generators}
A \emph{Feller semigroup}\footnote{Our terminology follows the convention of \cite{liggett2010continuous}, where a semigroup satisfying the above properties is referred to as a \emph{probability semigroup}. In the broader literature, the term \emph{Feller semigroup} is often used without the conservativeness assumption. We adopt the conservative convention throughout this work.} is a family \(\{T_t\}_{t\ge0}\) of bounded linear operators on \(C_0(\Pi)\) satisfying the following properties:
\begin{itemize}
	\item (Semigroup property). \(T_0=I\), and \(T_{t+s}=T_tT_s\) for all \(t,s\ge0\).
	\item (Strong continuity). For every \(\phi\in C_0(\Pi)\), \(\lim_{t\searrow0}\|T_t\phi-\phi\|_\infty=0\).
	\item (Positivity). For all $t\ge 0$ and $\phi\in C_0(\Pi)$ with $\phi\ge 0$, it holds $ T_t\phi\ge 0$. 
	\item (Conservativeness). There exists a sequence \(\{\phi_n\}_{n\ge1}\subset C_0(\Pi)\) such that \(0\le\phi_n\le1\), \(\phi_n\to1\) pointwise, and \(T_t\phi_n\to1\) pointwise as \(n\to\infty\) for all \(t\ge0\).
\end{itemize}
The \emph{(infinitesimal) generator} of a Feller semigroup $\{T_t\}_{t\ge 0}$ is a (closed) operator $\mcA:D(\mcA)\to C_0(\Pi)$ defined by
\begin{align}\label{eq:prob-gen}
	\mcA\phi = \lim_{t\searrow 0}\frac{1}{t}(T_t\phi-\phi),\qquad \phi\in D(\mcA),
\end{align}
where $D(\mcA)\subset C_0(\Pi)$ is the \emph{domain} of $\mcA$, consisting of functions $\phi$ for which the limit in \eqref{eq:prob-gen}, taken in the supremum norm, exists. An operator $(\mcA,D(\mcA))$ is called a \emph{Feller generator} if it is the (infinitesimal) generator of some Feller semigroup $\cb{T_t}_{t\ge0}$, which we denote 
$$T_t=e^{t\mcA}.$$ 
In this case, we say that $\{e^{t\mA}\}_{t\ge 0}$ is the Feller semigroup generated by $\mA$. 

A well-known fact \cite{liggett2010continuous} is that every Feller semigroup admits a \emph{transition-kernel representation}
\[
T_t\phi(x)=\int_{\Pi}\phi(z)\,\ka_t(x,dz),
\]
where \(\{\ka_t\}_{t\ge0}\) is a family of Markov kernels. Thus \(T_t\) is the Markov operator associated with the kernel \(\ka_t\). From this integral representation, \(T_t\) extends naturally to all \(\phi\in C_b(\Pi)\), and more generally to unbounded measurable functions whenever the integral is well defined.

Let us also introduce the \emph{extended domain} of \(\mcA\), denoted by \(D^\sim(\mcA)\), as the space of all functions \(\phi\in C_0^\sim(\Pi)\) such that \(\phi=\phi_0+a\) for some \(\phi_0\in D(\mcA)\) and \(a\in\mbr\). For such \(\phi\), we define
\[
\mcA\phi:=\mcA(\phi_0+a)=\mcA\phi_0.
\]
This definition is natural from \eqref{eq:prob-gen} because every Feller semigroup preserves constants; indeed, by the integral representation, $T_ta=a$ for every constant function \(a\in\mbr\).

%=============================================================================
%\subsection*{Connection between Feller processes, semigroups and generators}
Let us briefly recall the fundamental correspondence between Feller processes, Feller semigroups, and their generators, which underlies the semigroup approach adopted in the present work for the study of particle systems.
A $\Pi$-valued \emph{Feller process} may be viewed, informally, as a Markov process whose transition probabilities preserve the space $C_0(\Pi)$ and depend continuously on the initial state; see \cite{liggett2010continuous} for the precise definition. Such a process naturally induces a Feller semigroup $\{T_t\}_{t\ge0}$ on $C_0(\Pi)$ through
\begin{align}\label{eq2.5:temp}
	(T_t\phi)(x) = \mbe^x[\phi(X_t)].
\end{align}
At the level of probability measures, the corresponding dual action is given by $\mu T_t = \mathrm{law}(X_t)$, whenever $X_0\sim\mu$. This connection allows us to study Feller processes using a functional analytic approach. The generator $(\mcA,D(\mA))$ is, in turn, associated with the semigroup through the infinitesimal relation \eqref{eq:prob-gen}, which describes the instantaneous evolution of the semigroup.

The converse directions are substantially more involved. Given an operator $(\mcA,D(\mcA))$, the question of whether it generates a $C_0$-semigroup is referred to as the \emph{generation problem}. The generation of a Feller semigroup can be treated through the \emph{Lumer--Phillips theorem} by verifying appropriate maximal dissipativity conditions; additional positivity conditions are then required in the Feller setting. Next, a Feller semigroup can be realized as the transition semigroup of a Feller process through a standard extension theorem; see \cite{liggett2010continuous}. The resulting process is also characterized through the martingale problem associated with its generator: for every $\phi\in D(\mcA)$,
\begin{align*}
	\phi(X_t)-\phi(X_0)-\int_0^t\mcA\phi(X_s)\,ds
\end{align*}
is a martingale. Thus, the martingale problem provides a further connection between the generator and the corresponding stochastic process.

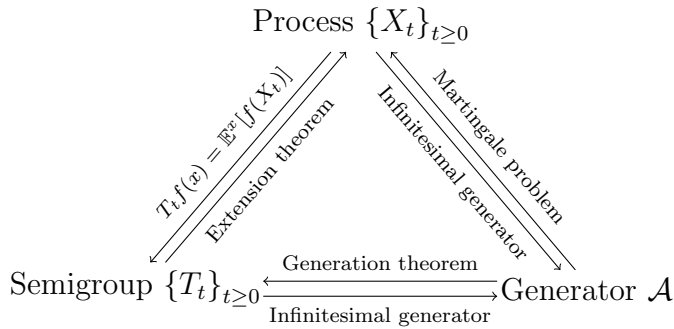
\begin{figure}[h]
	\centering
	\begin{tikzpicture}
		%-----------------------------------------------------
		\node (semigroup) at (0,0) {Semigroup $\cb{T_t}_{t\ge 0}$};
		\node (process) at (3,3.5) {Process $\cb{X_t}_{t\ge 0}$};
		\node (generator) at (6,0) {Generator $\mcA$};
		%-----------------------------------------------------
		\draw[->] (1.7,-0.1) -- (4.81,-0.1) node[midway, below, sloped, font=\fontsize{8}{10}\selectfont] {Infinitesimal generator};
		\draw[->] (4.81,0.1) -- (1.7,0.1) node[midway, above, sloped, font=\fontsize{8}{10}\selectfont] {Generation theorem};
		%-----------------------------------------------------
		\draw[->] (2.6,3.15) -- (0.2,0.35) node[midway, above, sloped, font=\fontsize{8}{10}\selectfont] {$T_tf(x) = \mbe^x[f(X_t)]$};
		\draw[->] (0.4,0.35) -- (2.8,3.15) node[midway, below, sloped, font=\fontsize{8}{10}\selectfont] {Extension theorem};
		%-----------------------------------------------------
		\draw[->] (5.85,0.25) -- (3.4,3.15) node[midway, above, sloped, font=\fontsize{8}{10}\selectfont] {Martingale problem};
		\draw[->] (3.2,3.15) -- (5.65,0.25) node[midway, below, sloped, font=\fontsize{8}{10}\selectfont] {Infinitesimal generator};
		%-----------------------------------------------------
		% \draw[step=1cm,gray,very thin] (-2,-2) grid (7,4);
	\end{tikzpicture}
	\caption{Connection between Feller processes, semigroups and generators}
\end{figure}

%=============================================================================
\subsection{Optimal transport}
The optimal transport cost, a central notion in the theory of optimal transport, is defined as follows.

\begin{definition}[$c$-optimal cost]\label{def:optimal-cost}
	Let $(\Pi,d)$ be a Polish space, $c:\Pi^2\to\mbr_+$ be a lower semicontinuous cost function, and $\mu,\nu\in\mcP(\Pi)$ be two probability measures.
	\begin{enumerate}[(i)]
		\item A \emph{coupling measure} of $\mu,\nu$ is a measure $\ga\in\mcP(\Pi^2)$ satisfying, for every measurable $E\subset\Pi$,
		\begin{align*}
			\ga(E\times\Pi) = \mu(E),\quad \ga(\Pi\times E) = \nu(E).
		\end{align*}
		We denote by $\Ga(\mu,\nu)\subset\mcP(\Pi^2)$ the set of all coupling measures of $\mu,\nu$.
		\item The \emph{$c$-optimal transport cost}, or simply \emph{$c$-optimal cost} between $\mu,\nu$ is defined by
		\begin{align}\label{eq2.8:temp}
			\mcC_c(\mu,\nu) = \inf_{\ga\in\Ga(\mu,\nu)} \int_{\Pi^2} c(x,y)\,d\ga(x,y).
		\end{align}
	\end{enumerate}
\end{definition}

Since $c$ is nonnegative and lower semicontinuous, the $c$-optimal transport cost is lower semicontinuous under weak convergence of its marginals (see \cite[Lemma 4.3 and Remark 6.12]{MR2459454}): if $\mu_n\to\mu$ and $\nu_n\to\nu$ weakly, then
\begin{align*}
	\mcC_c(\mu,\nu) \le \liminf_{n\to\infty} \mcC_c(\mu_n,\nu_n).
\end{align*}
Moreover, the infimum in \eqref{eq2.8:temp} is attained by some $\ga_*\in\Ga(\mu,\nu)$, called a \emph{$c$-optimal coupling} of $\mu,\nu$ (see \cite[Theorem 1.7]{santambrogio2015optimal}). That is,
\begin{align*}
	\mcC_c(\mu,\nu) = \int_{\Pi^2} c(x,y)\,d\ga_*(x,y).
\end{align*}
The optimal cost is finite whenever there exists a coupling with finite cost; in particular, it suffices that
\begin{align*}
	\int_{\Pi^2} c(x,y)\,\mu(dx)\nu(dy)<\infty.
\end{align*}

The $c$-optimal cost is defined by a minimization problem. 
As in many optimization problems, it is useful to consider the corresponding \emph{dual formulation}.
In particular, we have the following classical result \cite[Theorem 1.42]{santambrogio2015optimal}.
\begin{theorem}[Kantorovich duality]\label{thm:kantorovich-duality}
	Assume the setting of Definition \ref{def:optimal-cost}. Then
	\begin{align}\label{eq2.10:temp}
		\mcC_c(\mu,\nu) = \sup_{\phi,\psi} \sqb{\inn{\mu,\phi}+\inn{\nu,\psi}},
	\end{align}
	where the supremum is taken over all bounded continuous functions $(\phi,\psi)\in C_b(\Pi)^2$ satisfying
	\begin{align*}
		(\phi\oplus\psi)(x,y)=\phi(x)+\psi(y) \le c(x,y) \quad \text{for all $(x,y)\in \Pi^2$.}
	\end{align*}
\end{theorem}

\begin{remark}\label{rem:kanto}
	By a standard truncation and monotone approximation argument, the admissible class in Theorem \ref{thm:kantorovich-duality} may be restricted to the space \(C_0^\sim(\Pi)\), consisting of continuous functions that are the sum of a \(C_0(\Pi)\) function and a constant. More precisely, if \(\phi,\psi\in C_b(\Pi)\) satisfy \(\phi\oplus\psi\le c\), then one may approximate \(\phi,\psi\) from below by functions in \(C_0^\sim(\Pi)\) preserving the same constraint, and the corresponding integrals converge by monotone convergence.
\end{remark}

The following elementary lemma allows us to pass from a pointwise transport-cost bound for a pair of Markov operators to a global bound for arbitrary initial measures. 

\begin{lemma}\label{lem:pointwise-to-measure-lem}
	Let $c:\Pi^2\to\mbr_+$ be a continuous cost function, and $K,L$ be two Markov operators. Then, for all $\mu,\nu\in\mcP(\Pi)$,
	\begin{align*}
		\mcC_c(\mu K,\nu L) \le \int_{\Pi^2} \mcC_c(\de_x K,\de_y L)\,d\gamma(x,y),
	\end{align*}
	where $\ga\in\Ga(\mu,\nu)$ is any coupling of $\mu,\nu$.
\end{lemma}

\begin{remark}
	As a special case of the inequality above, if $\mu=\nu$, then the inequality simplifies to:
	\begin{align*}
		\mcC_c(\mu K,\mu L)&\le \int_{\Pi} \mcC_c(\de_x K,\de_x L)\,d\mu(x). 
	\end{align*}
\end{remark}

\begin{proof}[Proof of Lemma \ref{lem:pointwise-to-measure-lem}]
	%First, recall our shorthand notation $\mu K= K^* \mu$. 
	Let $\ka$ and $\la$ be the Markov kernels associated to $K$ and $L$, respectively, i.e. $\ka(x)=\de_x K$ and $\la(y)=\de_y L$. Let $\phi,\psi\in C_b(\Pi)$ be a pair of functions such that $\phi\oplus\psi\le c$. 
	Observe that
	\begin{align*}
		K\phi(x)+L\psi(y) &= \int_\Pi \phi(z)\,\ka(x,dz) + \int_\Pi \psi(z)\,\la(y,dz) = \inn{\ka(x),\phi} + \inn{\la(y),\psi}\\
		&\le \mcC_c(\ka(x),\la(y)),
	\end{align*}
	where the last inequality follows from Kantorovich duality. Then for any $\ga\in\Ga(\mu,\nu)$, it follows that
	\begin{align*}
		\inn{\mu K,\phi} + \inn{\nu L,\psi} &= \inn{\mu,K\phi} + \inn{\nu,L\psi} = \int_\Pi K\phi(x)\,d\mu(x) + \int_\Pi L\psi(y)\,d\nu(y)\\
		&= \int_{\Pi^2}\sqb{K\phi(x)+L\psi(y)}\,d\ga(x,y) \le \int_{\Pi^2}\mcC_c(\ka(x),\la(y))\,d\ga(x,y).
	\end{align*}
	Taking the supremum over all $\phi,\psi$ such that $\phi\oplus\psi\le c$, we obtain the desired bound by duality.
\end{proof}

%=============================================================================

\subsection{\texorpdfstring{Wasserstein-\(p\) metric spaces and their topological properties}{Wasserstein-p metric spaces and their topological properties}}\label{sec:was-p}
Let $(\Pi,d)$ be a state space.
For \(p \in[1,\infty)\), define the \(p\)-th cost function \(c_p: \Pi^2 \to \mathbb{R}_+\) by
\[
c_p(x,y) := \frac{1}{p}\, d(x,y)^p.
\]
The prefactor \(\f 1p\) is a matter of convention; one may set it to \(1\) without affecting the results, up to a change in constants.
Denote by \(\mcp_p(\Pi)\) the space of probability measures with finite \(p\)-th moment: for some, equivalently for all, \(z\in\Pi\),
\[
\int_\Pi c_p(z,x)\,d\mu(x)<\infty.
\]
It is well known (see \cite{MR2459454}) that the \(p\)-th root of the \(c_p\)-optimal transport cost given by
\[
\mcW_p(\mu,\nu):=\mcc_p(\mu,\nu)^{1/p}
=
\left(
\inf_{\ga\in\Ga(\mu,\nu)}
\int_{\Pi^2} c_p(x,y)\,d\ga(x,y)
\right)^{1/p}
\]
defines a \emph{complete separable metric} on \(\mcp_p(\Pi)\), called the \emph{Wasserstein-\(p\) distance}. The pair \((\mcp_p(\Pi),\mcW_p)\) is the \emph{Wasserstein-\(p\) space}.

As is well known in optimal transport theory, convergence in $(\mcp_p(\Pi),\mcW_p)$ can be characterized by weak convergence plus convergence of suitable test functions. More precisely, let \(C_{p}(\Pi)\subset C(\Pi)\) denote the space of continuous functions \(\phi:\Pi\to\mbr\) such that, for some \(M\ge0\) and some \(z\in\Pi\),
\[
|\phi(x)|\le M\rb{1+c_p(z,x)}
\qquad
\text{for all }x\in\Pi.
\]
The following characterization is classical; see \cite[Definition 6.8 and Theorem 6.9]{MR2459454}.
\begin{proposition}[Equivalence of \(\mcW_p\)-convergence]\label{prop:equiv-of-c-convergence}
	For $n\ge 1$, let \(\mu_n,\mu\in\mcp_p(\Pi)\). The following are equivalent:
	\begin{enumerate}[(i)]
		\item \(\mu_n\xrightarrow{\mcW_p}\mu\);
		\item \(\mu_n\to\mu\) weakly, and for some, equivalently all, \(z\in\Pi\),
		\[
		\lim_{n\to\infty}
		\int_\Pi c_p(z,x)\,d\mu_n(x)
		=
		\int_\Pi c_p(z,x)\,d\mu(x);
		\]
		\item \(\mu_n\to\mu\) weakly, and for every \(\phi\in C_{p}(\Pi)\),
		\[
		\lim_{n\to\infty}
		\int_\Pi \phi\,d\mu_n
		=
		\int_\Pi \phi\,d\mu.
		\]
	\end{enumerate}
\end{proposition}

%=============================================================================
%\subsection{The space of continuous curves in \texorpdfstring{$\mcp_p(\Pi)$}{P\_p(Pi)}}
%=============================================================================
%\subsection{The space \texorpdfstring{$\mcG_p^0(\Pi)$}{G\_p\^{}0} of Wasserstein-\texorpdfstring{$p$}{p} continuous generators}

\subsection{\texorpdfstring{The spaces \(C([0,T];\mcp_p(\Pi))\) and \(\mcG_p^0(\Pi)\)}{The spaces C([0,T];P\_p(Pi)) and G\_p\^{}0}}\label{subsec2.4.3}

In the present work, we often consider curves of probability measures in \(\mcp_p(\Pi)\), for instance as solutions of a mean-field evolution problem, or as the law of a single particle. We now introduce the space in which such curves live.

\begin{definition}[The space of continuous curves in \(\mcp_p(\Pi)\)]
	Let \(T>0\) and \(p\ge1\). We denote by \(C([0,T];\mcp_p(\Pi))\) the space of all continuous curves
	\[
	\rho:[0,T]\to\mcp_p(\Pi),
	\qquad
	t\mapsto \rho_t,
	\]
	where continuity is with respect to the Wasserstein-\(p\) metric \(\mcW_p\):
	\[
	\lim_{s\to t}
	\mcW_p(\rho_s,\rho_t)=0
	\qquad
	\text{for all }t\in[0,T].
	\]
	We equip \(C([0,T];\mcp_p(\Pi))\) with the uniform distance
	\[
	\mcD_p(\rho,\sigma)
	:=
	\sup_{t\in[0,T]}
	\mcW_p(\rho_t,\sigma_t).
	\]
\end{definition}

Since \((\mcp_p(\Pi),\mcW_p)\) is complete, it follows from standard topological arguments that \(\mcD_p\) defines a complete metric on \(C([0,T];\mcp_p(\Pi))\).

In the coming sections, we will also work with a subclass of Feller generators \(\mcA\) whose induced Markov flow of probability measures \(\mu_t=\mu e^{t\mcA}\) lies in \(C([0,T];\mcp_p(\Pi))\) for every initial distribution \(\mu\in\mcp_p(\Pi)\) and every, or equivalently some, \(T>0\). We refer to such generators as \emph{Wasserstein-\(p\) continuous} Feller generators.

\begin{definition}[The space of Wasserstein-\(p\) continuous generators]\label{def:gc0-generators}
	Let \(\mcG_p^0(\Pi)\subset\mcG(\Pi)\) be the subclass of all Feller generators \(\mcA\) satisfying, for all \(\mu\in\mcp_p(\Pi)\),
	\begin{enumerate}[(i)]
		\item \(\mu e^{t\mcA}\in\mcp_p(\Pi)\) for all \(t\ge0\);
		\item \(t\mapsto\mu e^{t\mcA}\) is continuous in \((\mcp_p(\Pi),\mcW_p)\), that is, if \(t_n\to t\), then
		\[
		\mcW_p(\mu e^{t_n\mcA},\mu e^{t\mcA})\to 0.
		\]
	\end{enumerate}
\end{definition}

For later use, it is convenient to characterize membership in \(\mcG_p^0(\Pi)\) through the convergence of suitable test functions rather than through the Wasserstein metric directly. For a Feller generator \(\mcA\in\mcG(\Pi)\), the map \(t\mapsto \mu e^{t\mcA}\) is already weakly continuous for every \(\mu\in\mcp(\Pi)\). By Proposition \ref{prop:equiv-of-c-convergence}, Wasserstein-\(p\) continuity is therefore equivalent to the additional continuity of the \(p\)-th moment, or equivalently of the action on all test functions in \(C_{p}(\Pi)\). This yields the following characterization.

\begin{proposition}\label{prop:characterization-Gp0}
	Let \(p\ge 1\) and \(\mcA\in\mcG(\Pi)\). The following are equivalent.
	\begin{enumerate}[(i)]
		\item \(\mcA\in\mcG_p^0(\Pi)\);
		\item for all \(\mu\in\mcp_p(\Pi)\), \(t\ge0\), and \(f\in C_{p}(\Pi)\), we have \(\mu e^{t\mcA}\in\mcp_p(\Pi)\), and \(\mbr_+\ni t\mapsto\inn{\mu,e^{t\mcA}f}\) is continuous;
		\item for all \(\mu\in\mcp_p(\Pi)\) and \(t\ge0\), we have \(\mu e^{t\mcA}\in\mcp_p(\Pi)\), and \(\mbr_+\ni t\mapsto\inn{\mu,e^{t\mcA}d(z,\cdot)^p}\) is continuous for some fixed \(z\in\Pi\).
	\end{enumerate}
\end{proposition}

%% file: Section3.tex
\section{Transport derivatives and stability between Markov flows}\label{chap3}

In this preparatory section, we introduce and study the \emph{transport derivatives} between two Markov flows and the resulting stability of the optimal transport cost.
By \emph{Markov flows}, we mean curves of probability measures \(\{\mu_t\}_{t\ge0}\) of the form \(\mu_t = \mu e^{t\mcA},\)
where \(\{e^{t\mcA}\}_{t\ge0}\) is a Feller semigroup and \(\mu\in\mcP(\Pi)\). Roughly speaking, the transport derivative measures the infinitesimal divergence of two Markov flows, measured by the \(c\)-optimal transport cost with respect to a cost function \(c\).
The term \emph{stability of the optimal cost} refers to the continuous dependence of the transport cost
\[
\mcC_c(\mu_t,\nu_t)
=
\mcC_c(\mu e^{t\mcA},\nu e^{t\mcB})
\]
with respect to perturbations of the initial measures \(\mu,\nu\in\mcP(\Pi)\) and of the generators \(\mcA,\mcB\). A prototypical result of this section will be a bound on \(\mcC_c(\mu_t,\nu_t)\) obtained from a bound on the associated transport derivatives. Such a bound will be called a \emph{stability estimate}, and it will play an important role in the subsequent sections.

The transport derivative and the stability estimates developed in this section will be used repeatedly in the sequel. In the next section, they enter the analysis of time-inhomogeneous Fokker--Planck equations of the form
\[
\partial_t \rho_t = \rho_t \mcA_t,
\qquad t\in[0,T],
\]
where \(\mcA:[0,T]\to\mcG(\Pi)\) is a curve of generators. There, the stability estimates provide the key ingredient for proving existence and uniqueness of the evolution problem.

The same estimates are equally important in Section \ref{chap5}, where we prove propagation of chaos. The main propagation-of-chaos result estimates the Wasserstein-\(p\) transport cost between the joint law of the \(N\)-particle system and the tensorized mean-field law. This estimate is obtained by controlling the transport derivatives between the particle generator and the tensorized mean-field generator.

\emph{Organization.}
The rest of this section is organized as follows.
In the coming subsection, we provide the definition of the $c$-transport derivative $\om_c,\om_c^*$ between Feller generators $\mA,\mB$. The main results of this section, Theorem \ref{thm:chap3-main-thm} and Corollary \ref{cor3.20:chap3-main-cor}, are stated in Sections \ref{sec3.3}, \ref{sec:c-exp} and proven in Sections \ref{sec3.3}, \ref{sec:etoa}.
In the final two subsections, we discuss a duality formula for $\om_c$ and its consequence, the subadditivity of $\om_c$, which plays an important role in the proof of propagation of chaos.

%=============================================================================
\subsection{Transport derivatives between Markov flows}
Let us begin by introducing the notions of transport derivative that will be important throughout this work. 

\begin{definition}[$c$-Transport derivative and its dual]\label{def:om-definition}
	Let $c:\Pi^2\to\mbr_+$ be a continuous cost.
	Let \(\mcA,\mcB\in\mcg(\Pi)\), \(\mu,\nu\in\mcp(\Pi)\) be such that $\mcc_c(\mu,\nu)<\infty$, and \(x,y\in\Pi\). 
	\begin{enumerate}[(i)]
		\item The \emph{$c$-transport derivative} between \(\mcA\) and \(\mcB\) at $(\mu,\nu)$ is
		defined by
		\[
		\om_c(\mu,\nu;\mcA,\mcB)
		:= \left.D^+\right|_{t=0} \mcc_c(\mu e^{t\mcA},\nu e^{t\mcB})
		= \limsup_{t\searrow 0}
		\frac{\mcc_c(\mu e^{t\mcA},\nu e^{t\mcB})-\mcc_c(\mu,\nu)}{t}\in[-\infty,\infty].
		\]
		By a slight abuse of notation, denote
		\[
		\om_c(x,y;\mcA,\mcB)
		:= \om_c(\de_x,\de_y;\mcA,\mcB).
		\]
		
		\item The \emph{dual $c$-transport derivative} between \(\mcA\) and \(\mcB\) at $(x,y)$ is defined by
		\[
		\om_c^*(x,y;\mcA,\mcB)
		:= \sup_{\substack{\phi\in D(\mcA),\,\psi\in D(\mcB)\\
				(\phi\oplus\psi)-c\text{ has a global maximum at }(x,y)}}
		\bigl\{ \mcA\phi(x)+\mcB\psi(y) \bigr\} \in[-\infty,\infty].
		\]
		If no such pair \((\phi,\psi)\) exists, the supremum is understood to be \(-\infty\).
	\end{enumerate}
\end{definition}

\begin{remark}\label{rem3.4:temp}
	In the definition of the dual transport derivative, the supremum may equivalently be taken over
	\(\phi\in D^\sim(\mcA)\) and \(\psi\in D^\sim(\mcB)\), since adding constants does not affect
	the maximizer of \((\phi\oplus\psi)-c\) nor the value of
	\(\mcA\phi(x_0)+\mcB\psi(y_0)\). The definition is inspired by the notion of
	\emph{viscosity solutions} in PDE theory; see, for instance,
	\cite{Costantini2018viscosity,fleming2006controlled,dai2018viscosity} for viscosity solutions in the setting of
	Feller/Markov generators. 
\end{remark}

Formally speaking, $\om_c(\mu,\nu;\mcA,\mcB)$ is an upper bound for the rate of change of the transport cost $t\mapsto\mcC_c(\mu e^{t\mcA},\nu e^{t\mcB})$ at time $t=0$. It will play a crucial role in establishing bounds for $\mcC_c(\mu e^{t\mcA}, \nu e^{t\mcB})$ such as the ones we will see later in Theorem \ref{thm:chap3-main-thm}. On the other hand, the dual transport derivative $\om_c^*(x,y;\mcA,\mcB)$ is the natural infinitesimal counterpart of the Kantorovich dual formulation of the optimal transport cost. In the classical Kantorovich duality, the transport cost is expressed as a supremum over pairs of potentials $\phi,\psi$ satisfying $\phi\oplus\psi\le c$. In the present setting, $\om_c^*$ is defined by a supremum over test functions whose sum $\phi\oplus\psi$ touches the cost $c$ from below at $(x,y)$, and the quantity $\mcA\phi(x)+\mcB\psi(y)$ measures the infinitesimal joint action of the generators at that contact point.

Let us highlight the significance of the transport derivatives $\om_c$ and $\om_c^*$ in this work. As will be shown in the subsequent sections, the key results in each section rely crucially on bounds for $\om_c$. In particular, an appropriate bound on $\om_c$ yields a stability estimate for the $c$-optimal transport cost, which in turn has important consequences for mean-field models, including well-posedness of mean-field evolution problems and propagation of chaos. In contrast, the dual transport derivative $\om_c^*$ is defined directly through the local action of the generators $\mcA$ and $\mcB$ on test functions and is therefore often easier to compute in practice. It is especially useful for deriving bounds on $\om_c$ via duality.

We note that transport derivatives between Markov flows have already been studied in specific settings. For instance, in \cite{alfonsi2018}, the authors consider \(\Pi=\mbr^d\), the cost \(c(x,y)=|x-y|^p\) with \(p\ge1\), and bounded pure-jump operators \(\mcA,\mcB\). In that setting, they prove existence of the derivative and obtain a duality formula for it; we return to this point in Remark \ref{rmk3.5:temp}. More recently, the transport derivative was computed explicitly for L\'evy generators and the quadratic cost \(c(x,y)=\frac12|x-y|^2\) in \cite{KangLim2025OptimalMarkovianCouplings}. This explicit formula will play an important role in Section \ref{chap6}, and we postpone its discussion until then.

The next lemma records the infinitesimal analogue of the classical weak Kantorovich duality:
the dual transport derivative is always bounded above by the primal transport derivative.
This is referred to as \emph{weak duality}.
The reverse inequality, namely whether \emph{strong duality} holds, is more subtle and will be addressed later in this work.

\begin{lemma}\label{lem:weak-dual}
	Let \(\mcA,\mcB\in\mcg(\Pi)\). For every \(x,y\in\Pi\),
	\[
	\om_c^*(x,y;\mcA,\mcB)\le \om_c(x,y;\mcA,\mcB).
	\]
\end{lemma}

\begin{proof}
	If \(\om_c(x,y;\mcA,\mcB)=+\infty\), the claim is trivial. Assume
	\(L:=\om_c(x,y;\mcA,\mcB)<\infty\), and fix \(M>L\). By definition of the limsup,
	there exists \(t_0>0\) such that for all \(t\in(0,t_0]\),
	\[
	\mcC_c(\de_x e^{t\mcA},\de_y e^{t\mcB})
	\le c(x,y)+tM.
	\]
	Let \(\phi\in D(\mcA)\), \(\psi\in D(\mcB)\) be a pair of functions such that $(\phi\oplus\psi)-c$ achieves a global maximum at $(x,y)$, and set
	\(
	\lambda:=\phi(x)+\psi(y)-c(x,y).
	\)
	Define \(\tilde\phi:=\phi-\lambda\), \(\tilde\psi:=\psi\). Then
	\(\tilde\phi\oplus\tilde\psi\le c\) and
	\(\tilde\phi(x)+\tilde\psi(y)=c(x,y)\). By Kantorovich duality,
	\[
	e^{t\mcA}\tilde\phi(x)+e^{t\mcB}\tilde\psi(y)
	\le \mcC_c(\de_x e^{t\mcA},\de_y e^{t\mcB})
	\le c(x,y)+tM
	= \tilde\phi(x)+\tilde\psi(y)+tM.
	\]
	Rearranging and dividing by \(t>0\),
	\[
	\frac{e^{t\mcA}\tilde\phi(x)-\tilde\phi(x)}{t}
	+\frac{e^{t\mcB}\tilde\psi(y)-\tilde\psi(y)}{t}
	\le M.
	\]
	Letting \(t\searrow0\), the left-hand side tends to
	\[
	\mcA\tilde\phi(x)+\mcB\tilde\psi(y)
	= \mcA\phi(x)+\mcB\psi(y),
	\]
	since adding constants does not change the action of the generators.
	Thus \(\mcA\phi(x)+\mcB\psi(y)\le M\). Taking the supremum over all admissible
	pairs \((\phi,\psi)\) gives \(\om_c^*(x,y;\mcA,\mcB)\le M\), and then letting
	\(M\searrow L\) yields the desired inequality. The case \(L=-\infty\) follows
	by taking \(M\in\mbr\) arbitrary.
\end{proof}

%=============================================================================
\subsection{Equivalence of stability estimates and duality formulas}\label{sec3.3}
Fix \(p\ge 1\), and consider a continuous cost function \(c:\Pi^2\to\mbr_+\) with \(c\in C_{p}(\Pi^2)\); that is, \(c\) satisfies a \(p\)-moment bound: for some \(M\ge 1\) and \((x_0,y_0)\in\Pi^2\),
\[
c(x,y)
\le
M\bigl[1+d(x_0,x)^p+d(y_0,y)^p\bigr]
\qquad
\text{for all }(x,y)\in\Pi^2.
\]
In particular, the \(p\)-th cost \(c_p(x,y)=p^{-1}d(x,y)^p\) satisfies this bound. For such a cost, \(\mcc_c(\mu,\nu)<\infty\) whenever \(\mu,\nu\in\mcp_p(\Pi)\), and hence the \(c\)-transport derivative is well defined for every such pair of measures.

The main result of this section gives equivalent formulations of the following stability estimate for the \(c\)-transport derivative:
\[
\om_c(x,y;\mA,\mB)
\le
f(x)+g(y)+\be c(x,y),
\qquad
\text{for all }x,y\in\Pi,
\]
where \(\mA,\mB\in\mcg_p^0(\Pi)\), \(\be\ge 0\), and \(f,g\in C_{p}(\Pi)\). We remark that in subsequent sections, the results of this section will be applied only in the special case \(c=c_p\). Nevertheless, we state and prove them for general \(c\in C_{p}(\Pi^2)\) for the sake of generality.

\begin{theorem}\label{thm:chap3-main-thm}
   	Let $p\ge 1$ and $c\in C_{p}(\Pi^2)$ be a nonnegative continuous cost.
    Suppose $\be\ge0$, $\mcA,\mcB\in\mcg_p^0(\Pi)$, and $f,g\in C_{p}(\Pi)$.
    Then the following are equivalent:
    \begin{enumerate}[(a)]
        \item for all $x,y\in\Pi$ and $t\ge0$, it holds
        \begin{align*}
        	\mcC_c(\de_x e^{t\mcA},\de_y e^{t\mcB}) \le e^{\be t} c(x,y) + \int_0^t e^{\be(t-s)}e^{s\mcA}f(x)\,ds + \int_0^t e^{\be(t-s)} e^{s\mcB} g(y)\,ds;
        \end{align*}
%        \begin{align*}
%            \mcC_c(\de_x e^{t\mcA},\de_y e^{t\mcB}) \le e^{\be t} c(x,y) + \int_0^t e^{\be s}e^{s\mcA}f(x)\,ds + \int_0^t e^{\be s} e^{s\mcB} g(y)\,ds;
%        \end{align*}
        \item for all $\mu,\nu\in\mcp_p(\Pi)$ and $t\ge0$, it holds
        \begin{align*}
        	\mcC_c(\mu e^{t\mcA},\nu e^{t\mcB}) \le e^{\be t} \mcC_c(\mu,\nu) + \int_0^t e^{\be(t-s)}\inn{\mu,e^{s\mcA}f}ds + \int_0^t e^{\be(t-s)} \inn{\nu,e^{s\mcB}g}ds;
        \end{align*}
%        \begin{align*}
%            \mcC_c(\mu e^{t\mcA},\nu e^{t\mcB}) \le e^{\be t} \mcC_c(\mu,\nu) + \int_0^t e^{\be s}\inn{\mu,e^{s\mcA}f}ds + \int_0^t e^{\be s} \inn{\nu,e^{s\mcB}g}ds;
%        \end{align*}
        \item for all $\mu,\nu\in\mcp_p(\Pi)$, it holds
        \begin{align*}
            \om_c(\mu,\nu;\mcA,\mcB)\le \inn{\mu,f}+\inn{\nu,g}+\be \mcC_c(\mu,\nu);
        \end{align*}
        \item for all $x,y\in\Pi$, it holds
        \begin{align*}
            \om_c(x,y;\mcA,\mcB)\le f(x)+g(y)+\be c(x,y);
        \end{align*}
        \item for all $x,y\in\Pi$, it holds
        \begin{align*}
        	\om_c^*(x,y;\mcA,\mcB)\le f(x)+g(y)+\be c(x,y);
        \end{align*}
    \end{enumerate}
\end{theorem}

\begin{remark}\label{rmk3.5:temp}
    We specifically highlight the equivalence between Conditions (d) and (e), which will be referred to as \emph{the duality}. Recall that Kantorovich duality (Theorem \ref{thm:kantorovich-duality}) states
    \begin{align*}
        \mcc_c(\mu,\nu) &= \sup_{\phi,\psi : \phi \oplus \psi \le c} \sqb{\langle \mu, \phi \rangle + \langle \nu, \psi \rangle}.
    \end{align*}
    The equivalence between (d) and (e) is analogous to the identity above. Indeed, we shall see later that the proof of (d) $\leftrightarrow$ (e) is based on Kantorovich duality. We point out that this duality relation was discovered in \cite{alfonsi2018} in the special case of $\Pi = \mathbb{R}^d$, $c(x, y) = |x - y|^p$ for $p \geq 1$, and $\mathcal{A}, \mathcal{B}$ being bounded pure jump operators. More precisely, the authors show that the identity holds:
    \begin{align*}
    	\om_c(\mu, \nu; \mA, \mB) &= \int_{\mathbb{R}^d} \mA \phi(x) \, d\mu(x) + \int_{\mathbb{R}^d} \mB \psi(x) \, d\nu(x),
    \end{align*}
    where $(\phi, \psi)$ is any pair of Kantorovich potentials, that is, any pair of functions that solves the Kantorovich (maximization) problem. Note that since $\mA$ and $\mB$ are bounded operators, $\phi$ and $\psi$ are in the (generalized) domain of $\mA$ and $\mB$, respectively.
    In fact, we will prove an analogue of this identity later; see Corollary \ref{cor:general-duality-theta} and the following remark. 
\end{remark}

Let us now proceed to the proof of Theorem \ref{thm:chap3-main-thm}. We will first prove the straightforward implications (a) $\to$ (b) $\to$ (c) $\to$ (d) $\to$ (e), leaving the more technical implication (e) $\to$ (a) to the next subsection. 
\begin{proof}[Proof of Theorem \ref{thm:chap3-main-thm}, (a) $\to\cdots\to$ (e)]
    (a) $\to$ (b). Let $\ga$ be a $c$-optimal coupling of $\mu,\nu$. By Lemma \ref{lem:pointwise-to-measure-lem} and Fubini's theorem,
    \begin{align*}
        \mcC_c(\mu e^{t\mcA},\nu e^{t\mcB}) &\le \int_{\Pi^2} \mcC_c(\de_x e^{t\mcA},\de_y e^{t\mcB})\,d\ga(x,y)\\
        &\le e^{\be t} \mcC_c(\mu,\nu) + \int_\Pi \int_0^t e^{\be (t-s)}e^{s\mcA}f(x)\,ds\,d\mu(x) + \int_\Pi \int_0^t e^{\be (t-s)}e^{s\mcB}g(y)\,ds\,d\nu(y)\\
        &\le e^{\be t} \mcC_c(\mu,\nu) + \int_0^t e^{\be (t-s)}\inn{\mu,e^{s\mcA}f}ds + \int_0^t e^{\be (t-s)} \inn{\nu,e^{s\mcB}g}ds.
    \end{align*}
	
    (b) $\to$ (c). Rearranging terms in (b), we find
    \begin{align*}%\label{eq3.10:temp}
    	\frac{\mcC_c(\mu e^{t\mcA},\nu e^{t\mcB})-\mcC_c(\mu,\nu)}{t} &\le \frac{e^{\be t}-1}{t} \mcC_c(\mu,\nu) + \frac{e^{\be t}}{t}\int_0^t e^{-\be s}\sqb{\inn{\mu,e^{s\mcA}f}+\inn{\nu,e^{s\mcB}g}}ds.
    \end{align*}
    Since $\mA,\mB\in\mcg_p^0(\Pi)$, by Proposition \ref{prop:characterization-Gp0} the map $s\mapsto \inn{\mu, e^{s\mA}f}+\inn{\nu,e^{s\mB}g}$ is continuous. (c) follows by letting $t\searrow 0$. 
%    Passing $t\searrow0$ in \eqref{eq3.10:temp}, we find
%    \begin{align}\label{eq3.sdjkbjs}
%        \om_c(\mu,\nu;\mcA,\mcB) &\le \be \mcC_c(\mu,\nu) + \limsup_{t\searrow0}\frac{e^{\be t}}{t}\int_0^t e^{-\be s}\sqb{\inn{\mu,e^{s\mcA}f}+\inn{\nu,e^{s\mcB}g}}ds.
%    \end{align}
%    Since $\inn{\mu,e^{s\mcA}f}\xrightarrow{s\searrow0}\inn{\mu,f}$, we find
%    \begin{align*}
%        \limsup_{t\searrow0}\frac{1}{t}\int_0^t e^{\be (t-s)}\inn{\mu,e^{s\mcA}f}ds &\le \rb{\limsup_{t\searrow0}e^{\be t}}\rb{\limsup_{t\searrow0}\frac{1}{t}\int_0^t e^{-\be s}\inn{\mu,e^{s\mcA}f}ds}\\
%        &= \inn{\mu,f}.
%    \end{align*}
%    Similarly, the last term of \eqref{eq3.sdjkbjs} is at most $\inn{\nu,g}$. In total, we obtain
%    \begin{align*}
%        \om_c(\mu,\nu;\mcA,\mcB)\le \inn{\mu,f}+\inn{\nu,g}+\be \mcC_c(\mu,\nu).
%    \end{align*}

    (c) $\to$ (d). (d) follows from (c) with $\mu=\de_x$ and $\nu=\de_y$.

    (d) $\to$ (e). It follows from Lemma \ref{lem:weak-dual}.
\end{proof}

%=============================================================================
\subsection{Proof of Theorem \ref{thm:chap3-main-thm}, (e) \texorpdfstring{$\to$}{implies} (a)}\label{sec:etoa}
The proof of the implication (e) $\to$ (a) in Theorem \ref{thm:chap3-main-thm} is more involved. 
As hinted in the statement of Theorem \ref{thm:chap3-main-thm}(e) itself, the proof relies on the \emph{comparison principle argument} from the theory of viscosity solutions. To perform the comparison principle argument that involves the operator $\mA$, one has to guarantee that the test functions achieve their maximum at some point in a compact set. This is usually achieved by modifying the test functions by subtracting an auxiliary function $F_\ep(x)$ such that $\lim_{x\to\infty} F_\ep(x)=\infty$, while $\|\mA F_\ep(x)\|<\ep$, then letting $\ep\searrow 0$. For instance, if $\mA=\De$ is the $d$-dimensional Laplacian, one usually chooses $F_\ep(x)=\frac \ep2 |x|^2$. The following lemma shows the existence of such auxiliary functions in our case.

\begin{lemma}\label{lem3.18:temp}
	Let $\mcA\in \mcG(\Pi)$ and $C>0$.
	There is a sequence of functions $\{F_n\}_n\subset D^\sim(\mA)$ such that $0\le F_n\le C$, $\lim_{x\to\infty} F_n(x)=C$, $F_n\to 0$ locally uniformly as $n\to\infty$, and $\lim_{n\to\infty}\|\mA F_n\|_\infty=0$. 
\end{lemma}

\begin{proof}
	Given $\la>0$, let $R_\la(\mA)$ denote the $\la$-resolvent operator of $\mA$, and let $I_\la=\la R_\la(A)$. It is a well-known fact that $I_\la$ is a Markov operator. 
	Suppose $G\in C_0(\Pi)$ is such that $0\le G\le 1$. Let $\tilde F:=I_\la G\in D(\mA)$, that is, $\la G = (\la I-\mA)\tilde F$, then 
	\begin{align*}
		\|\tilde F\|_\infty\le \|I_\la G\|_\infty\le \|G\|_\infty\le 1.
	\end{align*}
	It follows that
	\begin{align}\label{eq3.11:Af-bound}
		\|\mA \tilde F\|_\infty=\la \|\tilde F-G\|_\infty\le 2\la.
	\end{align}
	
	Now, fix $\la>0$, $\ep\in(0,1)$ and a compact set $K\subset \Pi$. We claim there is a function $G\in C_0(\Pi)$ with $0\le G\le 1$ such that $(1-\ep)\mbo_K\le I_\la G\le 1$. To see this claim, we fix a sequence $\{G_n\}_n\subset C_0(\Pi)$ such that $0\le G_n\le 1$ and $G_n\nearrow  1$ locally uniformly. Then we have $I_\la G_n\nearrow 1$ pointwise as $n\to\infty$, because $I_\la$ is a Markov operator. By Dini's theorem, $I_\la G_n$ converges locally uniformly to 1. This means that there is a sufficiently large $n$ such that $(1-\ep)\mbo_{K}\le I_\la G_n(x)\le 1$. Choose $G=G_n$.
	
	To prove the lemma, by rescaling the functions, it suffices to prove the case $C=1$.
	Fix two sequences $\la_n\searrow 0$, $\ep_n\searrow 0$ and a sequence of increasing compact sets such that $K_n\nearrow \Pi$. By the claim above, there is $G_n\in C_0(\Pi)$ such that $0\le G_n\le 1$ and $(1-\ep_n)\mbo_{K_n}\le \tilde F_n\le 1$ where $\tilde F_n=I_{\la_n}G_n$. In particular, $\tilde F_n\to 1$ locally uniformly. Moreover, \eqref{eq3.11:Af-bound} implies $\|\mA \tilde F_n\|_\infty\le 2 \la _n\to 0$. Set $F_n=1-\tilde F_n$, which finishes the construction.
\end{proof}

Finally, we may now complete the proof of the implication (e) $\to$ (a). We will make use of the well-known \emph{Duhamel principle} (see \cite[Lemma 1.3]{engel2006short}, \cite[Page 49]{evans2010partial} or \cite[Page 135]{john2004partial}), which states as follows. For any $\be \ge 0$, $\phi\in D(\mA)$, and $f\in C_0(\Pi)$, the function $\{\Phi_t\}_{t\ge 0}$ defined by
\begin{align}\label{duhamel}
	\Phi_t(x)&= e^{t(\mA-\be)}\phi(x)- \int_0^t e^{s(\mA-\be)}f(x)ds
\end{align}
satisfies $\Phi_t\in D(\mA)$ for each $t\ge 0$ and solves the following PDE in the classical sense:
\begin{align*}
	\partial_t \Phi_t&= (\mA-\be)\Phi_t- f.
\end{align*}
Since Feller semigroups preserve constant functions, the formula \eqref{duhamel} holds even for $\phi\in D^\sim(\mA), f\in C_0^\sim(\Pi)$. In particular, for such $\phi,f$, $\{\Phi_t\}_{t\ge0}$ given by \eqref{duhamel} satisfies the PDE above, with $\Phi_t\in D^\sim(\mA)$ for each $t\ge 0$.

\begin{proof}[Proof of Theorem \ref{thm:chap3-main-thm}, (e) $\to$ (a)]
	It suffices to prove the implication for $\be>0$. Indeed, if Condition (e) holds for $\be=0$, then it also holds for every $\be>0$, since $c\ge0$. Applying the result for $\be>0$ and then letting $\be\searrow0$ recovers the bound for $\be=0$.
	
	We first prove the implication with the assumption that $f,g\in C_{p}(\Pi)$ are bounded below.  
	Fix $T>0$ and $\phi\in D^\sim(\mcA)$, $\psi\in D^\sim(\mcB)$ such that $\phi\oplus\psi\le c$. We will first establish the following bound: for all $t\in[0,T]$ and $x,y\in\Pi$,
	\begin{align}\label{eq3.12:temp}
		\quad e^{t\mcA}\phi(x) + e^{t\mcB}\psi(y) \le e^{\be t}c(x,y) + \int_0^t e^{\be(t-s)}e^{s\mcA}f(x)\,ds + \int_0^t e^{\be(t-s)}e^{s\mcB}g(y)\,ds.
	\end{align}
	Equivalently, for all $t\in[0,T]$, it holds $\Phi_t\oplus\Psi_t\le c$, where
	\begin{align*}
		\Phi_t(x) = e^{t(\mcA-\be)}\phi(x) - \int_0^t e^{s(\mcA-\be)}f(x)\,ds,\quad
		\Psi_t(y) = e^{t(\mcB-\be)}\psi(y) - \int_0^t e^{s(\mcB-\be)}g(y)\,ds.
	\end{align*}
	
	Let us introduce some constants and functions used in the proof. First, define
	\begin{align*}
		m_f := -\min\mcb{0,\inf_{x\in\Pi}f(x)}\in[0,\infty),\qquad m_g := -\min\mcb{0,\inf_{y\in\Pi}g(y)}\in[0,\infty),
	\end{align*}
	and the constant $C$ depending on $T,\phi,\psi,f,g$:
	\begin{align*}
		C: =  \norm{\phi}_\infty + \norm{\psi}_\infty + m_f T + m_g T +2.
	\end{align*}
	By Lemma \ref{lem3.18:temp}, let $\scb{F_n}_n\subset D^\sim(\mcA)$, $\scb{G_n}_n\subset D^\sim(\mcB)$ be two sequences of functions such that $0\le F_n,G_n\le C$, $F_n(x),G_n(x)\to C$ as $x\to\infty$, 
	$F_n,G_n\to 0$ locally uniformly as $n\to\infty$, and
	\begin{align}\label{eq3.13:temp}
		\la_n^F:=\norm{\mcA F_n}_\infty \to0, \qquad \la_n^G=\norm{\mcB G_n}_\infty \to0.
	\end{align}
	Since $F_n,G_n$ converge to $C$ as $x\to\infty$, we may find a compact set $K_n\subset\Pi$ such that
	\begin{align}\label{eq3.14:temp}
		F_n(x),G_n(x) \ge C-1, \qquad \mbox{for all $x\in K_n^c$. }
	\end{align}
	
	Now we introduce $\chi_{K_n}\in C_0(\Pi)$, a cut-off function such that $0\le\chi_{K_n}\le1$ and $\chi_{K_n}=1$ on $K_n$, and let
	\begin{align*}
		f_{n} &= (f+m_f)\chi_{K_n}-m_f, \qquad g_{n}= (g+m_g)\chi_{K_n}-m_g.
	\end{align*}
	We make the observation that $f_{n},g_{n}\in C_0^\sim(\Pi)$, and the following bounds hold:
	\begin{align}\label{eq3.15:temp}
		-m_f\le f_{n}\le f,\qquad -m_g\le g_{n}\le g. 
	\end{align}
	Next, let us introduce the following functions:
	\begin{align}
		%\Phikn(x) &= e^{t(\mcA-\be)}\phi(x) - \int_0^t e^{s(\mcA-\be)}f_{n}(x)\,ds,\nonumber\\
		%\Psikn(y) &= e^{t(\mcB-\be)}\psi(y) - \int_0^t e^{s(\mcB-\be)}g_{n}(y)\,ds,\nonumber\\
		\Phitn(x) &= e^{t(\mcA-\be)}\phi(x) - \int_0^t e^{s(\mcA-\be)}f_{n}(x)\,ds - \lafn\, t - F_n(x),\label{eq3.16:temp}\\
		\Psitn(y) &= e^{t(\mcB-\be)}\psi(y) - \int_0^t e^{s(\mcB-\be)}g_{n}(y)\,ds - \lagn\, t - G_n(y),\nonumber
	\end{align}
	Notice that for each $t\ge 0$, $\Phi_t^n\in D^\sim(\mcA)$, $\Psi_t^n\in D^\sim(\mcB)$ and they satisfy the following PDEs for all $t\ge 0$ in the classical sense, see the discussion of \eqref{duhamel}:
	\begin{align}
		\partial_t \Phi_t^n &= (\mA-\be)\Phi_t^n - f_n - \la_n^F  +\mA F_n - \be\la_n^F t - \be F_n, \label{eq3.17:temp}\\
		\partial_t \Psi_t^n &= (\mB-\be)\Psi_t^n - g_n - \la_n^G  +\mB G_n - \be\la_n^G t - \be G_n. \nonumber
	\end{align}
	
	Finally, let $\Om_n(t,\cdot,\cdot)=\Phitn\oplus\Psitn-c$, that is,
	\begin{align}
		\Om_n(t,x,y)&:=\Phi_t^n(x)+\Psi_t^n(y)-c(x,y) \nonumber\\
		&= \Ta_n(t,x,y) - \rb{\lafn + \lagn } t - F_n(x) - G_n(y) - c(x,y)  \label{eq3.18:temp},
	\end{align}
	where
	\begin{align}\label{eq3.19:Ta-def}
		\Ta_n(t,x,y)&:= e^{t(\mcA-\be)}\phi(x) + e^{t(\mcB-\be)}\psi(y) -\int_0^te^{s(\mA-\be)}f_{n}(x)\,ds-\int_0^te^{s(\mB-\be)}g_{n}(y)\,ds.
	\end{align}
	We will next show $\Om_n=\Phi_t^n\oplus \Psi_t^n-c\le 0$. 
	
	Assume to the contrary that the inequality $\Om_n(t,\cdot,\cdot)\le 0$ does not hold for some $t\in[0,T]$. We now show the existence of a global maximizer $(t_0,x_0,y_0)\in (0,T]\times \Pi^2$, which will then lead to a contradiction to Condition (e). We note that the function $\Om_n(t,x,y)$ has the coercive property:
	\begin{align}\label{eq3.20:temp}
		\sup_{t\in[0,T],(x,y)\in(K_n\times K_n)^c}\Om_n(t,x,y) \le -1.
	\end{align}
	To see this, notice first that since $e^{t\mcA}, e^{t\mcB}$ are contractions and \eqref{eq3.15:temp}, $\Ta_n$ from \eqref{eq3.19:Ta-def} satisfies the bound
	\begin{align*}
		\Ta_n(t,x,y)&\le \|\phi\|_\infty +\|\psi\|_{\infty} + (m_f+m_g)T = C-2. 
	\end{align*}
	By the definition of $K_n$ from \eqref{eq3.14:temp}, we then have $\Ta_n(t,x,y)-G_n(y)-c(x,y)\le -1$ for all $(t,x,y)\in [0,T]\times \Pi\times K_n^c$. Since the remaining terms in \eqref{eq3.18:temp} are nonpositive, this implies $\Om_n(t,x,y)\le -1$ for such $(t,x,y)$. Similarly, on $[0,T]\times K_n^c\times \Pi$, we have $\Om_n(t,x,y)\le -1$. This shows \eqref{eq3.20:temp}. Since $\sup\Om_n>0$ and $\Om_n(0,\cdot,\cdot)\le0$, this leads to the existence of a global maximizer $(t_0,x_0,y_0)$ for the function $\Om_n$ over $(0,T]\times K_n\times K_n$, with $\Om_n(t_0,x_0,y_0)>0$. 
	
	By the maximality at $(t_0,x_0,y_0)$, we have 
	\begin{align}\label{eq3.21:temp}
		\drvev{}{t}_{t=t_0}\Om_n(t,x_0,y_0)\ge0.
	\end{align}
	Recall that $\Phi_{t_0}^n\in D^\sim(\mA),\Psi_{t_0}^n\in D^\sim(\mB)$. 
	Since $(x_0,y_0)$ is a global maximizer of $\Phi_{t_0}^n\oplus\Psi_{t_0}^n-c$, Condition (e) of Theorem \ref{thm:chap3-main-thm} (see Remark \ref{rem3.4:temp}) implies 
	\begin{align}\label{eq3.22:temp}
		\mA\Phi_{t_0}^n(x_0)+\mB\Psi_{t_0}^n(y_0)\le f(x_0)+g(y_0)+\be c(x_0,y_0)
	\end{align}
	
	Now, let us consider the $t$-derivative of $t\mapsto \Om_n(t,x_0,y_0)$ at $t=t_0$. 
	In the coming computation, all functions are evaluated at $t=t_0$ and $(x,y)=(x_0,y_0)$. By \eqref{eq3.17:temp}, we compute
	\begin{align*}
		\drv{}{t}\Om_n &= \drv{}{t}\Big|_{t=t_0}\sqb{\Phitn+\Psitn}\\
		&= [(\mA-\be)\Phi_{t_0}^n \oplus (\mB-\be)\Psi_{t_0}^n] - (f_n\oplus g_n) + (\mA F_n\oplus\mB G_n-\la_n^F-\la_n^G) \\
		&\qquad - \be[(\la_n^F+\la_n^G)t+F_n\oplus G_n].\\
		&=: I -II+III-IV.
	\end{align*}
	Note that $IV\ge 0$ as $\be,F_n,G_n,\la_n^F,\la_n^G\ge 0$. Next, $III\le 0$ by the definition of $\la_n^F,\la_n^G$ in \eqref{eq3.13:temp}. Combining \eqref{eq3.22:temp} and $f(x_0)=f_{n}(x_0)$, $g(y_0)=g_{n}(y_0)$ (as $(x_0,y_0)\in K_n\times K_n$) and $\Om_n(t_0,x_0,y_0)>0$, we find
	\begin{align*}
		I-II&\le \be c(x_0,y_0)-\be \Phi_t^n(x_0)-\be \Psi_t^n(y_0) = -\be \Om_n(t_0,x_0,y_0)<0.
	\end{align*}
	Combining all the bounds, we conclude $\drvev{}{t}_{t=t_0}\Om_n<0.$
	This leads to a contradiction to \eqref{eq3.21:temp}. Therefore, we must have $\Om_n(t,x,y) \le 0$ on $[0,T]\times\Pi^2$. That is,
	\begin{align}
		e^{t\mcA}\phi(x) + e^{t\mcB}\psi(y) &\le \rb{\lafn + \lagn }t\,e^{\be t} + e^{\be t}\sqb{F_n(x)+G_n(y)}\nonumber\\
		&\quad + e^{\be t}c(x,y) + \int_0^t e^{\be(t-s)}e^{s\mcA}f_{n}(x)\,ds  + \int_0^t e^{\be(t-s)}e^{s\mcB}g_{n}(y)\,ds. \nonumber
	\end{align}
	By \eqref{eq3.15:temp}, the above holds with $f,g$ in place of $f_{n},g_{n}$. As $n\to\infty$, the terms in the first line of the right-hand side vanish, and hence we arrive at \eqref{eq3.12:temp}. 
	
	In total, we have shown that for all $\phi\in D^\sim(\mcA)$, $\psi\in D^\sim(\mcB)$ satisfying $\phi\oplus\psi\le c$, and $f,g$ bounded below, Condition (e) implies \eqref{eq3.12:temp}. Since $D^\sim(\mA),D^\sim(\mB)$ are dense in $C_0^\sim(\Pi)$, the bound \eqref{eq3.12:temp} extends to all $\phi,\psi\in C_0^\sim(\Pi)$ such that $\phi\oplus\psi\le c$. By Theorem \ref{thm:kantorovich-duality} (see also Remark \ref{rem:kanto}), we conclude
	\begin{align}\label{eqn:3.102temp}
		\mcC_c(\de_x e^{t\mcA},\de_y e^{t\mcB}) \le e^{\be t} c(x,y) + \int_0^t e^{\be (t-s)}e^{s\mcA}f(x)\,ds + \int_0^t e^{\be (t-s)} e^{s\mcB} g(y)\,ds. 
	\end{align}
	
	Now it remains to remove the boundedness-below assumption on $f,g$.	
	For $M\in\mbn$, let
	\begin{align*}
		f_{M}(x) = \max\scb{f(x),-M},\quad g_M(y)=\max\scb{g(y),-M}.
	\end{align*}
	Condition (e) of Theorem \ref{thm:chap3-main-thm} then holds with $f_M,g_M$ as well.
	Since $f_M,g_M$ are bounded below by $-M$, the result above implies \eqref{eqn:3.102temp} with $f_M,g_M$ in place of $f,g$.
	Furthermore, for fixed $x,y$, we have $e^{s\mcA}f_M(x)\searrow e^{s\mcA}f(x)$ and $e^{s\mcB}g_M(y)\searrow e^{s\mcB}g(y)$ pointwise in $s$ as $M\to\infty$. Applying the monotone convergence theorem then yields \eqref{eqn:3.102temp}.
\end{proof}

\subsection{The \texorpdfstring{$p$}{p}-exponential stability condition and its equivalence}\label{sec:c-exp}
For a fixed \(p\ge 1\), when the cost is \(c=c_p=p^{-1}d^p\), we use the shorthand
\[
\om_p(x,y;\mA,\mB)=\om_{c_p}(x,y;\mA,\mB),\qquad \om^*_p(x,y;\mA,\mB)=\om^*_{c_p}(x,y;\mA,\mB).
\]
The hypotheses for the other main theorems in the coming sections will be formulated in terms of the following condition on the transport derivatives: there exist $\al,\be\ge0$ such that
\[\om_p(\mu,\nu;\mcA,\mcB)\le \al+\be\mcC_p(\mu,\nu),\quad\text{for all }\mu,\nu\in\mcP_p(\Pi).\]
This condition will lead to an exponential stability bound by Gr\"{o}nwall's inequality. The following result is a specific case that follows directly from Theorem \ref{thm:chap3-main-thm}, obtained by taking $f\equiv \al,g\equiv 0$. We state it as an independent result as it will be frequently used. We introduce the following shorthand notation that will be used throughout this work. Given $\be\ge0$, we define $\zeta_{\be}:[0,\infty)\to[0,\infty)$ by
\begin{align}\label{eq3.25:zeta-be}
    \zeta_{\be}(t) =		
    \begin{cases}
        \displaystyle \frac{e^{\be t}-1}{\be}, &\text{if }\be>0,\\
        t, &\text{if }\be=0.
    \end{cases}
\end{align}
Note that $\zeta_\be(t)\to\zeta_0(t)$ locally uniformly as $\be\searrow0$.

\begin{corollary}[Equivalence of \(p\)-exponential stability condition]\label{cor3.20:chap3-main-cor}
	Let \(p\ge 1\), \(\mcA,\mcB\in\mcg_p^0(\Pi)\), and \(\al,\be\ge0\). The following are equivalent:
	\begin{enumerate}[(a)]
		\item For all \(x,y\in\Pi\) and \(t\ge0\),
		\[
		\mcc_p(\de_x e^{t\mcA},\de_y e^{t\mcB})
		\le
		e^{\be t}\,c_p(x,y)+\al\,\zeta_{\be}(t);
		\]
		\item For all \(\mu,\nu\in\mcp_p(\Pi)\) and \(t\ge0\),
		\[
		\mcc_p(\mu e^{t\mcA},\nu e^{t\mcB})
		\le
		e^{\be t}\,\mcc_p(\mu,\nu)+\al\,\zeta_{\be}(t);
		\]
		\item For all \(\mu,\nu\in\mcp_p(\Pi)\),
		\[
		\om_p(\mu,\nu;\mcA,\mcB)
		\le
		\al+\be\,\mcc_p(\mu,\nu);
		\]
		\item For all \(x,y\in\Pi\),
		\[
		\om_p(x,y;\mcA,\mcB)
		\le
		\al+\be\,c_p(x,y);
		\]
		\item For all \(x,y\in\Pi\),
		\[
		\om_p^*(x,y;\mcA,\mcB)
		\le
		\al+\be\,c_p(x,y).
		\]
	\end{enumerate}
\end{corollary}

To facilitate the discussion in the coming sections, we give the following definition.

\begin{definition}[$p$-exponential stability condition]\label{def:c-exp-stab}
	Let $p\ge 1$ and $\al,\be\ge 0$.
    \begin{enumerate}[(i)]
        \item We say $\mcA,\mcB\in\mcg_p^0(\Pi)$ satisfy the \emph{$p$-exponential stability condition with parameters $\al$ and $\be$}, or \emph{$\mathrm{Exp}_p(\al,\be)$-condition} for short, if they satisfy any of the equivalent conditions in Corollary \ref{cor3.20:chap3-main-cor} with $\al,\be\ge0$.
        \item We define $\mathrm{Exp}_p(\al,\be)$ to be the set of all pairs $(\mcA,\mcB)$ in $\mcg_p^0(\Pi)$ that satisfy the $\mathrm{Exp}_p(\al,\be)$-condition. Furthermore, we define
        \begin{align*}
            \mathrm{Exp}_p = \bigcup_{\al,\be\ge0}\mathrm{Exp}_p(\al,\be).
        \end{align*}
    \end{enumerate}
\end{definition}
\begin{remark}\label{rmk3.22:temp}
	One may use the $p$-exponential stability condition to \emph{topologize} the space $\mcg_p^0(\Pi)$. As seen in Corollary \ref{cor3.20:chap3-main-cor}, if $\mA,\mB$ are close in the sense that $(\mA,\mB)\in \mathrm{Exp}_p(\ep,\be)$, where $\ep>0$ is small, then the flows $\mu e^{t\mA},\mu e^{t\mB}$ starting at the same measure $\mu$ remain close, as the following bound holds:
	\begin{align*}
		\mcW_p^p(\mu e^{t\mA},\mu e^{t\mB})\le \ep \zeta_\be(t).
	\end{align*}
	More specifically, if $\{\mA_n\}_n$ is a sequence of $\mcg_p^0(\Pi)$-generators such that the pair $(\mA_n,\mB)$ satisfies the $\mathrm{Exp}_p(\ep_n,\be)$-condition with $\ep_n\to 0$, then $\mu e^{t\mA_n}\xrightarrow{\mcW_p}\mu e^{t\mB}$ for all $\mu\in\mcp_p(\Pi)$.
\end{remark}

%=============================================================================
%\subsection{General duality formula and subadditivity of \texorpdfstring{$\om_c$}{omega\_c}}

\subsection{Strong duality of transport derivatives}
Let us now state a generalization of the equivalence between (d) and (e) of Theorem \ref{thm:chap3-main-thm}, which will be useful in proving other estimates for the transport derivative $\om_c$. As shown in Lemma \ref{lem:weak-dual}, the weak duality $\om_c^*\le \om_c$ holds. The following corollary establishes a strong duality in the sense of upper semicontinuous envelopes: the upper semicontinuous envelope of $\om_c$ coincides with that of $\om_c^*$, so that the primal and dual transport derivatives cannot be distinguished at the level of such upper bounds.

\begin{corollary}\label{cor:general-duality-theta}
    Let $p\ge 1$ and $c\in C_{p}(\Pi^2)$ be a nonnegative continuous cost.
    Let $\mcA, \mcB \in \mcg_p^0(\Pi)$, and let $\mcl{D} \subset D(\mcA)$, $\mcl{D}' \subset D(\mcB)$ be core subspaces of their respective domains. Let $\ta:\Pi^2\to[-\infty,\infty)$ be an upper semicontinuous function satisfying $\ta(x,y)\le \tilde f(x)+\tilde g(y)$ for some $\tilde f,\tilde g\in C_{p}(\Pi)$. Then the following are equivalent:
    \begin{enumerate}[(i)]
        \item for all $x,y\in\Pi$, it holds
        \[\om_c(x,y;\mcA,\mcB)\le\ta(x,y);\]
        \item for all $x,y\in\Pi$, it holds
        \[\om_c^*(x,y;\mcA,\mcB)\le\ta(x,y);\]
        \item for all $\phi\in \mcl{D},\psi\in \mcl{D}'$ such that $(\phi\oplus\psi)-c$ achieves a global maximum at some $(x_0,y_0)\in\Pi^2$, it holds
        \begin{align*}
            \mcA\phi(x_0)+\mcB\psi(y_0)\le\ta(x_0,y_0).
        \end{align*}
    \end{enumerate}
\end{corollary}
\begin{remark}
	In particular, if $\om_c\le \tilde f\oplus \tilde g$
	for some \(\tilde f,\tilde g\in C_{p}(\Pi)\), then the upper semicontinuous envelopes
	of \(\om_c\) and \(\om_c^*\) coincide.
	This is the strong duality alluded to above
	(cf.\ Lemma \ref{lem:weak-dual}).
\end{remark}

Let us first prove (i)$\leftrightarrow$(ii), and (ii)$\to$(iii). 

\begin{proof}[Proof of Corollary \ref{cor:general-duality-theta}]
    (i) $\to$ (ii). 
    Let $\mcF=\scb{(f,g)\in C_{p}(\Pi)\times C_p(\Pi) : \ta\le f\oplus g}$. From the assumption, $\mcl{F}$ is nonempty.
    We have $\ta = \inf_{(f,g)\in\mcl{F}}f\oplus g$ pointwise. Indeed, for each $(x_0,y_0)\in\Pi^2$ and $a\in\mbr$ with $a>\ta(x_0,y_0)$, upper semicontinuity and the assumed $p$-growth bound yield, for some $C\ge0$,
    \[
        \ta(x,y)\le a+C\bigl[d(x_0,x)^p+d(y_0,y)^p\bigr],\qquad x,y\in\Pi.
    \]
    This separable majorant equals $a$ at $(x_0,y_0)$, giving the identity by taking the infimum over $a>\ta(x_0,y_0)$.
    From the assumption, $\om_c(\cdot,\cdot;\mcA,\mcB)\le \ta$. Hence we have
    \begin{align*}
        \om_c(x,y;\mcA,\mcB) \le f(x)+g(y)
    \end{align*}
    for all $(f,g)\in\mcF$. By Theorem \ref{thm:chap3-main-thm} (with $\be=0$), (d) $\to$ (e), the same bound holds with $\om_c^*$ in place of $\om_c$.
    Since this bound holds for all $(f,g)\in\mcF$, taking the infimum over all such pairs, we find $\om_c^*(\cdot,\cdot;\mA,\mB)\le \ta$.
    
    (ii) $\to$ (i). The proof of the converse is identical.

    (ii) $\to$ (iii). Trivial.  
\end{proof}

Let us now focus on the implication (iii) $\to$ (ii). While this may appear straightforward due to the density of the cores $\mcl{D}$ and $\mcl{D}'$ in $D(\mcA)$ and $D(\mcB)$, respectively, the argument in fact relies on a nontrivial approximation technique reminiscent of those used in the theory of viscosity solutions.

Suppose $\phi \in D(\mcA)$ and $\psi \in D(\mcB)$ are such that the function $(\phi \oplus \psi) - c$ attains its global maximum at some point $(x_0, y_0) \in \Pi^2$. By the density of $\mcl{D}$ in $D(\mcA)$ and $\mcl{D}'$ in $D(\mcB)$, we can find a sequence $\{(\phi_n, \psi_n)\} \subset \mcl{D} \times \mcl{D}'$ such that $\phi_n \to \phi$ in $D(\mcA)$ and $\psi_n \to \psi$ in $D(\mcB)$.
Assume further that for all sufficiently large \( n \), the function $(\phi_n \oplus \psi_n) - c$ admits a global maximizer $(x_n, y_n)$ and that $(x_n, y_n) \to (x_0, y_0)$ as \( n \to \infty \). Then, using the uniform convergence of $\mcA \phi_n \to \mcA\phi$, $\mcB \psi_n \to \mcB\psi$ and the upper semicontinuity of $\ta$, we obtain:
\[
\mcA \phi(x_0) + \mcB \psi(y_0) 
= \lim_{n \to \infty} \left[ \mcA \phi_n(x_n) + \mcB \psi_n(y_n) \right] 
\le \limsup_{n \to \infty} \ta(x_n, y_n) 
\le \ta(x_0, y_0).
\]

In the argument above we use the \emph{stability of the maximizer}. Let $\Phi:\Ga\to\mbr$ be a \emph{weakly coercive} continuous function: there is a compact set $K\subset \Ga$ such that 
\begin{align}\label{bound:weak-coercive-max}
	\sup_{z\notin K} \Phi(z) < \sup_{z\in K} \Phi(z).
\end{align}
Specifically, $\Phi$ attains its global maximum strictly inside the compact set $K$. 
Let $\{\Phi_n\}$ be a sequence of continuous functions such that $\Phi_n \to \Phi$ uniformly. Then for all sufficiently large $n$, the function $\Phi_n$ admits a global maximizer $z_n \in \Ga$, and we have:
\begin{itemize}
	\item $z_n \in K$ for all large $n$;
	\item there is a subsequence of $\{z_n\}$ that converges to a global maximizer of $\Phi$.
\end{itemize}
This property ensures that the location of the maximizers is stable under uniform perturbations, provided the original function is weakly coercive.

Given $(\phi, \psi) \in D(\mA) \times D(\mB)$, we aim to apply the stability result to $\Phi := (\phi \oplus \psi) - c$. However, $\Phi$ may not be weakly coercive, so maximizers may not be stable. To address this, we perturb $\Phi$ by adding a small multiple of a function that grows outside a compact neighborhood. The following lemma provides such functions within the core. In what follows, given a subspace $\mcl{D}\subset C_0(\Pi)$ (e.g., a core of a generator), denote by $\mcl{D}^\sim$ the space of all $\varphi\in C_b(\Pi)$ such that $\varphi+c\in \mcl{D}$ for some $c\in \mbr$.

\begin{lemma}\label{lem:pi}
	Let \(\mcA\in\mcg(\Pi)\) and let \(\mcl{D}\subset D(\mcA)\) be a core.
	For every \(x_0\in\Pi\) and every compact neighborhood \(K\) of \(x_0\), there exists
	a function \(\chi\in\mcl{D}^\sim\) such that
	\[
	\chi\ge 0,
	\qquad
	\chi(x_0)\le \frac12,
	\qquad
	\chi(x)\ge 2
	\quad
	\text{for all }x\notin K.
	\]
\end{lemma}

\begin{proof}
	By Urysohn's lemma there exists $h\in C_0^\sim(\Pi)$ such that $h\ge \f 14$, $h(x_0)=\f14$ and $h(x)\ge \f 94$ for all $x\notin K$. Since $\mcl{D}$ is a core of $\mA$, the space $\mcl{D}^\sim$ is dense in $C_0^\sim(\Pi)$. Hence there exists \(\chi\in\mcl{D}^\sim\) such that $\|\chi-h\|_\infty\le \f 14$. Note that $\chi$ satisfies the desired condition.
\end{proof}

We may now complete the proof of Corollary \ref{cor:general-duality-theta}. 

\begin{proof}[Proof of Corollary \ref{cor:general-duality-theta}, (iii) $\to$ (ii)]
	Let \(\phi\in D(\mA),\psi\in D(\mB)\) be such that
	\((\phi\oplus\psi)-c\) achieves a global maximum at \((x_0,y_0)\in \Pi^2\).
	We will show that for every pair of compact neighborhoods \(K_{x_0}\ni x_0\) and
	\(K_{y_0}\ni y_0\), there exists \((x_0',y_0')\in K_{x_0}\times K_{y_0}\) such that
	\begin{equation}\label{eq:shortgoal}
		\mA\phi(x_{0}')+\mB\psi(y_{0}')
		\le
		\ta(x_{0}',y_{0}').
	\end{equation}
	If this claim holds, then taking compact neighborhoods
	\(K_{x_0}^n\searrow\{x_0\}\) and \(K_{y_0}^n\searrow\{y_0\}\), and choosing
	\((x_n,y_n)\in K_{x_0}^n\times K_{y_0}^n\) as in \eqref{eq:shortgoal}, we have
	\((x_n,y_n)\to(x_0,y_0)\). Since \(\mcA\phi\) and \(\mcB\psi\) are continuous and
	\(\ta\) is upper semicontinuous, passing to the limit gives
	\[
	\mcA\phi(x_0)+\mcB\psi(y_0)
	\le
	\limsup_{n\to\infty}\ta(x_n,y_n)
	\le
	\ta(x_0,y_0),
	\]
	which is (ii).

	Let us now prove the claim. Fix $\ep>0$. 
	Let $\chi^\mA\in \mcl{D}^\sim$ and $\chi^\mB\in (\mcl{D}')^\sim$ be the cutoff functions given by Lemma \ref{lem:pi}. 
	These functions are nonnegative, satisfy $\chi^\mA(x_0)\le \frac12$, $\chi^\mB(y_0)\le \frac12$, and $\chi^\mA(x)\ge 2$, $\chi^\mB(y)\ge 2$ whenever $x\notin K_{x_0}$ and $y\notin K_{y_0}$, respectively. 
	Define
	\[
	\phi^\ep := \phi-\ep\chi^\mA,\qquad 
	\psi^\ep := \psi-\ep\chi^\mB.
	\]
	We claim that $(\phi^\ep\oplus\psi^\ep)-c$ is weakly coercive in the sense of \eqref{bound:weak-coercive-max} with $K=K_{x_0}\times K_{y_0}$.
	Indeed, for every $(x,y)\in K^c$, we have $\chi^\mA(x)+\chi^\mB(y)\ge 2$, while $\chi^\mA(x_0)+\chi^\mB(y_0)\le 1$. Hence,
	\begin{align*}
		[(\phi^\ep\oplus\psi^\ep)-c](x,y)
		&= \phi(x)+\psi(y)-c(x,y)-\ep\bigl(\chi^\mA(x)+\chi^\mB(y)\bigr)\\
		&\le \phi(x_0)+\psi(y_0)-c(x_0,y_0)-2\ep\\
		&\le \phi(x_0)+\psi(y_0)-c(x_0,y_0)-\ep\bigl(\chi^\mA(x_0)+\chi^\mB(y_0)\bigr)-\ep\\
		&= [(\phi^\ep\oplus\psi^\ep)-c](x_0,y_0)-\ep.
	\end{align*}
	This establishes the weak coercivity bound \eqref{bound:weak-coercive-max}.
	In particular, the global maximum of $(\phi^\ep\oplus\psi^\ep)-c$ must be attained within $K$.
	
	Since $\mcl{D}\subset D(\mA)$ and $\mcl{D}'\subset D(\mB)$ are cores, there exist sequences $\{(\phi_n,\psi_n)\}\subset \mcl{D}\times \mcl{D}'$ such that $\phi_n\to \phi$ in $D(\mA)$ and $\psi_n\to\psi$ in $D(\mB)$. Consider the sequence $\{(\phi_{n}^\ep,\psi_{n}^\ep)\}_n\subset \mcl{D}^\sim\times (\mcl{D}')^\sim$ given by
	\[
	\phi_{n}^\ep := \phi_n-\ep \chi^\mA,\qquad 
	\psi_{n}^\ep := \psi_n-\ep \chi^\mB.
	\]
	Since $(\phi_{n}^\ep,\psi_{n}^\ep)$ converges uniformly to $(\phi^\ep,\psi^\ep)$ as $n\to\infty$, by the stability of maximizers, after passing to a subsequence if necessary, there exists a sequence of maximizers $\{(x_{n}^\ep,y_{n}^\ep)\}$ of $(\phi_{n}^\ep\oplus\psi_{n}^\ep)-c$ converging to a maximizer $(x^\ep,y^\ep)\in K$ of $(\phi^\ep\oplus\psi^\ep)-c$.
	By Condition (iii), for each $n$,
	\[
	\mA\phi_{n}^\ep(x_{n}^\ep)+\mB\psi_{n}^\ep(y_{n}^\ep)\le \ta(x_{n}^\ep,y_{n}^\ep).
	\]
	Rearranging the terms,
	\begin{align*}
		\mA\phi_{n}(x_{n}^\ep)+\mB\psi_{n}(y_{n}^\ep)
		&\le \ta(x_{n}^\ep,y_{n}^\ep)+\ep\bigl(\mA\chi^\mA(x_{n}^\ep)+\mB\chi^\mB(y_{n}^\ep)\bigr)\\
		&\le \ta(x_{n}^\ep,y_{n}^\ep)+\ep\bigl(\|\chi^\mA\|_{D(\mA)}+\|\chi^\mB\|_{D(\mB)}\bigr).
	\end{align*}
	Letting $n\to\infty$ and using $(x_{n}^\ep,y_{n}^\ep)\to (x^\ep,y^\ep)$, the uniform convergence $\mA\phi_n\to\mA\phi$, $\mB\psi_n\to\mB\psi$, and the upper semicontinuity of $\ta$, we obtain
	\[
	\mA\phi(x^\ep)+\mB\psi(y^\ep)
	\le \ta(x^\ep,y^\ep)+\ep\bigl(\|\chi^\mA\|_{D(\mA)}+\|\chi^\mB\|_{D(\mB)}\bigr).
	\]
	
	In conclusion, for every $\ep>0$ there exists $(x^\ep,y^\ep)\in K_{x_0}\times K_{y_0}$ satisfying the above inequality. Letting $\ep\searrow 0$ and extracting a subsequence, we find $(x_0',y_0')\in K_{x_0}\times K_{y_0}$ such that \eqref{eq:shortgoal} holds. The proof is complete.
\end{proof}

\subsection{Subadditivity of transport derivatives}\label{sec3.5}

A consequence of Corollary \ref{cor:general-duality-theta} is a subadditivity property for transport derivatives at point initial data. More precisely, suppose that
\begin{align*}
    \om_c(x,y;\mcA_k,\mcB_k)\le\ta_k(x,y),\qquad k=1,2,
\end{align*}
for suitable upper semicontinuous majorants $\ta_1,\ta_2$. Under the growth, generation, and common-core assumptions stated below, these bounds imply
\begin{align*}
    \om_c(x,y;\mcA_1+\mcA_2,\mcB_1+\mcB_2)\le\ta_1(x,y)+\ta_2(x,y),\qquad x,y\in\Pi.
\end{align*}

Before proving such a bound, it is crucial to address the issue of whether $\mcA_1+\mcA_2$ generates a Feller semigroup. Specifically, given two Feller generators $\mcA_1,\mcA_2\in\mcG(\Pi)$, the question arises: is their sum $\mcA_1+\mcA_2$ also a Feller generator? This problem is well-studied in the context of perturbation theory for $C_0$-semigroups. Under various conditions on the generators $\mcA_k$, the answer is affirmative. For a detailed discussion, we refer the reader to \cite[Chapter III]{engel2006short}.
We will assume that $\mcA_1+\mcA_2$ and $\mcB_1+\mcB_2$ are indeed Feller generators in $\GCZP$ and focus on proving subadditivity. 

%Corollary \ref{cor:general-duality-theta} provides insight into the relationship between the expressions $\om_c(x,y;\mcA,\mcB)$ and $\sup_{\phi,\psi}\sqb{\mcA\phi(x_0) + \mcB\psi(y_0)}$. The latter expression is particularly useful when considering the sum of generators, as it is clear that
%\begin{align*}
%    \sup_{\phi,\psi}\sqb{(\mcA_1+\mcA_2)\phi(x) + (\mcB_1+\mcB_2)\psi(y)} &= \sup_{\phi,\psi}\sqb{\mcA_1\phi(x) + \mcB_1\psi(y)} + \sup_{\phi,\psi}\sqb{\mcA_2\phi(x) + \mcB_2\psi(y)},
%\end{align*}
%but such (sub)-additive property cannot be easily verified for $\om_c(x,y;\mcA_1+\mcA_2,\mcB_1+\mcB_2)$.

\begin{theorem}\label{thm:om-subadditivity}
	Let $p\ge 1$, $c\in C_p(\Pi^2)$ be a nonnegative continuous cost, and $\cb{\mcA_k,\mcB_k}_{k=1}^m$ be a family of $\GCZP$-generators. For each $1\le k\le m$, let $\ta_k:\Pi^2\to[-\infty,\infty)$ be an upper semicontinuous function satisfying $\ta_k \le f_k\oplus g_k$ for some $f_k,g_k\in C_{p}(\Pi)$. Let
    \begin{align*}
        \ta=\sum_{k=1}^m \ta_k,\quad \mcA=\sum_{k=1}^m \mcA_k,\quad \mcB=\sum_{k=1}^m\mcB_k.
    \end{align*}
    Assume $\mcA$, $\mcB$ are Feller generators in $\GCZP$, and there are two cores $\mcD,\mcD'$ of $\mcA,\mcB$ such that $\mcD\subset D(\mcA_k), \mcD'\subset D(\mcB_k)$ for each $1\le k\le m$.
    Suppose that, for each $1\le k\le m$,
    \begin{align}\label{eq3.19:wc-theta-k}
        \om_c(x,y;\mcA_k,\mcB_k) \le \ta_k(x,y),\quad\text{for all }x,y\in\Pi,
    \end{align}
    then
    \begin{align*}
        \om_c(x,y;\mcA,\mcB) \le \ta(x,y),\quad\text{for all }x,y\in\Pi.
    \end{align*}
\end{theorem}
\begin{remark}
	Theorem \ref{thm:om-subadditivity} yields subadditivity at the level of
	upper semicontinuous envelopes rather than pointwise. Indeed, taking the infimum
	over all admissible majorants \(\ta_k\) gives
	\[
	\bigl(\om_c(\cdot,\cdot;\mcA,\mcB)\bigr)^{\vee}
	\le
	\sum_{k=1}^m
	\bigl(\om_c(\cdot,\cdot;\mcA_k,\mcB_k)\bigr)^{\vee},
	\]
	where $F^\vee$ denotes the upper semicontinuous envelope of $F:\Pi^2\to[-\infty,\infty)$. 
\end{remark}

\begin{proof}
    Let $\phi\in\mcD$, $\psi\in\mcD'$ be such that $c-(\phi\oplus\psi)$ achieves a global minimum at some $(x_0,y_0)\in\Pi^2$. From \eqref{eq3.19:wc-theta-k} and the implication (i) $\to$ (ii) in Corollary \ref{cor:general-duality-theta}, we obtain
    \begin{align*}
        \mcA_k\phi(x_0)+\mcB_k\psi(y_0) \le \ta_k(x_0,y_0).
    \end{align*}
    Summing over $k$ gives
    \begin{align*}
        \mcA\phi(x_0)+\mcB\psi(y_0) \le \sum_{k=1}^m\ta_k(x_0,y_0) = \ta(x_0,y_0).
    \end{align*}
    Finally, notice that $\ta$ is an upper semicontinuous function satisfying $\ta \le f\oplus g$, where $f:=\sum_{k=1}^m f_k$, $g:=\sum_{k=1}^m g_k$ are in $C_{p}(\Pi)$. Applying Corollary \ref{cor:general-duality-theta} (iii) $\to$ (i) then leads to
    \begin{align*}
        \om_c(x,y;\mcA,\mcB) \le \ta(x,y)
    \end{align*}
    for all $x,y\in\Pi$.
\end{proof}

%% file: Section4.tex
\section{Abstract mean-field equations and their well-posedness}\label{chap4}

%=============================================================================
\subsection{Overview}\label{4.1:overview}

The objective of this section is threefold: first, to introduce the notion of a mean-field generator describing a mean-field system; second, to formulate the mean-field evolution equation associated with a mean-field generator, with particular attention to the appropriate notion of solution; and third, to establish well-posedness with respect to that notion of solution. The notion of mean-field generator introduced here will remain the central object of study in the rest of the paper, especially in the next section, where propagation of chaos is analyzed through the same generator formalism.

\subsubsection{Abstract formulation of mean-field models and equations}

From now on, we adopt the setting of Sections \ref{sec:was-p} and \ref{sec:c-exp}. Specifically, the state space \((\Pi,d)\) is a locally compact Polish space equipped with a metric \(d\), and the cost function is the \(p\)-cost
\[
c(x,y)=c_p(x,y)=\frac{1}{p}\,d(x,y)^p,
\qquad p\ge1.
\]
We recall the space \(\mcp_p(\Pi)\) of probability measures with finite \(p\)-th moment, the function space \(C_{p}(\Pi)\), and the generator class \(\mcg_p^0(\Pi)\), together with the Wasserstein-\(p\) distance \(\mcW_p\), introduced in Section \ref{chap2}. We also recall the \(p\)-transport derivative between \(\mA,\mB\in\mcg_p^0(\Pi)\) from Definition \ref{def:om-definition}, with \(c=c_p\):
\[
\om_p(\mu,\nu;\mcA,\mcB)
:=
\left.D^+\right|_{t=0}
\mcW_p^p(\mu e^{t\mcA},\nu e^{t\mcB})
=
\limsup_{t\searrow0}
\frac{\mcW_p^p(\mu e^{t\mcA},\nu e^{t\mcB})-\mcW_p^p(\mu,\nu)}{t}.
\]
Our main hypothesis will be formulated in terms of \(\om_p\).

Let us now introduce the notion of mean-field generators, which is the core of this work. This notion appears in a related form in \cite{Chaintron_2022_1}; here we recast it in the framework of Wasserstein-continuous generators and transport derivatives developed in Sections \ref{chap2} and \ref{chap3}.

\begin{definition}[Mean-field generators]\label{def4.1:mean-field-operator}
	Fix $p\ge 1$. A \emph{mean-field generator on $\mcp_p(\Pi)$} is defined as a map $\MFG$.
\end{definition}

\begin{notation}
	A mean-field generator $\MFG$ can also be treated as a collection $\cb{\mcA(\mu):\mu\in\mcP_p(\Pi)}$ of $\mcG_p^0(\Pi)$-generators indexed by $\mcP_p(\Pi)$.
	We use the following notation:
	\begin{enumerate}[(i)]
		\item For $\mu\in\mcP_p(\Pi)$, $\phi\in D(\mcA_\mu)$, we write $\mcA(\mu)=\mcA_\mu$ and $\mcA(\mu)(\phi)=\mcA_\mu(\phi)=\mcA(\phi;\mu)$.
		\item For each $\mu\in\mcP_p(\Pi)$, $\mcA(\mu)$ generates a Feller semigroup $\scb{e^{t\mcA(\mu)}}_{t\ge0}$, whose transition kernels we denote by $\cb{\ka_t(\cdot;\mu)}_{t\ge0}$.
	\end{enumerate}
\end{notation}

Given a mean-field generator \(\{\mcA(\mu):\mu\in\mcp_p(\Pi)\}\), the argument \(\mu\in\mcp_p(\Pi)\) will be called a \emph{mean-field measure}. Intuitively, the generator \(\mcA(\mu)\) describes the short-time evolution of a single particle under the influence of the surrounding \emph{mean-field particles}, which are assumed to be distributed according to the mean-field measure \(\mu\). Illustrative examples are given in Section \ref{chap6}.

%=============================================================================
\subsubsection{Mean-field equations and their well-posedness}\label{subsec4.asbkj}
Recall the space $C([0,T];\mcp_p(\Pi))$ of continuous curves in $\mcp_p(\Pi)$ introduced in Section \ref{subsec2.4.3}.
Let $\mu=\scb{\mu_t}_{t}\in C([0,T];\mcP_p(\Pi))$ be a curve of mean-field measures, and consider the initial value problem of the evolution equation for a given mean-field generator $\{\mA(\mu)\}_\mu$:
\begin{align}\label{eq4.7:linearized-mf-eq}\tag{PMFP}
	\left\{\begin{aligned}
		\partial_t \rho_t &= \rho_t\mcA(\mu_t),\quad t\in(0,T),\\
		\rho_0 &\in \mcP_p(\Pi).
	\end{aligned}\right.
\end{align}
A solution $\rho=\scb{\rho_t}_{t}\in C([0,T];\mcp_p(\Pi))$ is the distribution of a single particle under the influence of mean-field particles, assuming that, at each time $t\in[0,T]$, the mean-field particles follow the \emph{prescribed} distribution $\mu_t$. Formally, equation \eqref{eq4.7:linearized-mf-eq} means: in the infinitesimal time window $[t,t+h]$, $\rho_s$ evolves following the Markov flow with the Feller semigroup $\cb{e^{s\mcA(\mu_t)}}_{s\ge0}$:
\begin{align*}
	\rho_{t+h} \approx \rho_t e^{h\mcA(\mu_t)},\quad h \ll 1.
\end{align*}

The \emph{mean-field evolution problem} is a special case of the evolution problem \eqref{eq4.7:linearized-mf-eq}, where instead of a prescribed distribution $\mu_t$, the continuous curve of mean-field measures $\mu_t$ agrees with the solution $\rho_t$, that is,
\begin{align}\label{eq4.6:mean-field-equation}\tag{MFP}
	\left\{\begin{aligned}
		\partial_t \br_t &= \br_t \mcA(\br_t),\quad t\in(0,\infty),\\
		\br_0 &\in \mcP_p(\Pi).
	\end{aligned}\right.
\end{align}
In short, the distribution of a single particle $\br_t$ evolves as \eqref{eq4.7:linearized-mf-eq}, where the mean-field particles have exactly the same distribution $\br_t$. 

Notice that \eqref{eq4.7:linearized-mf-eq} is linear, while \eqref{eq4.6:mean-field-equation} is nonlinear. In the rest of this section, we develop a well-posedness theory for the mean-field problem \eqref{eq4.6:mean-field-equation}, under a continuity assumption on the mean-field generator expressed in terms of the \(p\)-transport derivative. This hypothesis will be in force throughout the remainder of this section and the next.

\begin{hypothesisA}[Lipschitz condition on the $p$-transport derivative]\label{hypoA}
	Assume that there exist constants $\al,\be\ge0$ such that the mean-field generator $\MFG$ satisfies, for all $\mu,\nu,\mu',\nu'\in\mcP_p(\Pi)$,
	\begin{align*}
		\om_p(\mu,\nu;\mcA(\mu'),\mcA(\nu')) \le \al\mcW_p^p(\mu',\nu')+\be\mcW_p^p(\mu,\nu).
	\end{align*}
	That is, $(\mcA(\mu'),\mcA(\nu'))\in\mathrm{Exp}_p(\al\mcW_p^p(\mu',\nu'),\be)$. Equivalently (see Corollary \ref{cor3.20:chap3-main-cor}), it holds for all $x,y\in\Pi$, $\mu',\nu'\in\mcP_p(\Pi)$ that
    \begin{align*}
		\om_p(x,y;\mcA(\mu'),\mcA(\nu')) \le \al\mcW_p^p(\mu',\nu')+\f{\be}{p} d(x,y)^p.
	\end{align*}
\end{hypothesisA}

Hypothesis {\refA} is essentially a Lipschitz continuity assumption on the map 
\[ (\mu,\mu') \mapsto \mu e^{t\mA(\mu')}. \]
In Section \ref{chap3}, particularly Corollary \ref{cor3.20:chap3-main-cor}, we provided a list of equivalent conditions for Hypothesis {\refA}, which will be useful for verifying it. 
A consequence of that corollary is the following estimate:
\begin{align*}
	\mcW_p^p(\mu e^{t \mathcal{A}(\mu')}, \nu e^{t \mathcal{A}(\nu')}) &\leq e^{\beta t} \mcW_p^p(\mu, \nu) + \al\zeta_\beta(t)\mcW_p^p(\mu', \nu').
\end{align*}
Thus, if the pairs $\mu, \nu$ and $\mu', \nu'$ are close in the Wasserstein-$p$ topology, the generated flows $\mu e^{t \mathcal{A}(\mu')}, \nu e^{t \mathcal{A}(\nu')}$ remain close, with their separation controlled by the time-dependent factors in the displayed estimate.
We remark here the significance of this hypothesis. It ensures the well-posedness of the mean-field evolution problem and guarantees the propagation of chaos for certain classes of $N$-particle systems.

We will prove well-posedness, that is, the existence, uniqueness and stability of solutions, under a notion of solutions that we call \emph{$\mcW_p$-stable solutions}. Although the statements of the main results of this section rely on the definition of $\mcW_p$-stable solutions, we shall present them here for the sake of clarity and flow. The first result shows the well-posedness of the linearized problem \eqref{eq4.7:linearized-mf-eq}, while the second shows the existence and uniqueness of a solution to the mean-field problem \eqref{eq4.6:mean-field-equation}.

\begin{theorem}[Well-posedness of the linearized problem]\label{thm4.345:chap4-main-result-linearized}
	Let $\MFG$ be a mean-field generator that satisfies Hypothesis {\refA}.
	\begin{enumerate}[(i)]
		\item For any $\mu\in C([0,T];\mcP_p(\Pi))$, the linearized problem \eqref{eq4.7:linearized-mf-eq} admits a unique $\mcW_p$-stable solution $\scb{\rho_t}_{t\in[0,T]}\in C([0,T];\mcP_p(\Pi))$.
		\item Suppose $\mu,\nu\in C([0,T];\mcP_p(\Pi))$. If $\rho,\si\in C([0,T];\mcP_p(\Pi))$ are the $\mcW_p$-stable solutions of
		\begin{align*}
			\left\{\begin{aligned}
				\partial_t \rho_t &= \rho_t\mcA(\mu_t),\quad t\in(0,T),\\
				\rho_0 &\in \mcP_p(\Pi),
			\end{aligned}\right.
			\quad\text{and}\quad
			\left\{\begin{aligned}
				\partial_t \si_t &= \si_t\mcA(\nu_t),\quad t\in(0,T),\\
				\si_0 &\in \mcP_p(\Pi),
			\end{aligned}\right.
		\end{align*}
		respectively, then it holds for all $t\in[0,T]$ that    
		\begin{align*}
			\mcW_p^p(\rho_t,\si_t) \le \mcW_p^p(\rho_0,\si_0) e^{\be t} + \al\zeta_\be(t) \sup_{s\in[0,T]} \mcW_p^p(\mu_s,\nu_s).
		\end{align*}
	\end{enumerate}
\end{theorem}

\begin{theorem}[Existence and uniqueness of a solution to the mean-field problem]\label{thm4.3:chap4-main-result}
    Suppose a mean-field generator $\MFG$ satisfies Hypothesis {\refA}.
    Then for every initial condition $\br_0\in\mcP_p(\Pi)$, there exists a unique $\mcW_p$-stable solution $\cb{\br_t}_{t\ge0}\in C([0,\infty);\mcP_p(\Pi))$ to the mean-field problem \eqref{eq4.6:mean-field-equation}.
\end{theorem}

%=============================================================================
\subsubsection{Abstract linear Fokker-Planck equations}
To prove the well-posedness of \eqref{eq4.6:mean-field-equation}, we follow the classical theory for parabolic semilinear/quasilinear equations. We first address the well-posedness of the linearized problem \eqref{eq4.7:linearized-mf-eq}, then recover a unique solution to the mean-field problem \eqref{eq4.6:mean-field-equation} by a standard fixed point argument.

More generally, the linearized problem \eqref{eq4.7:linearized-mf-eq} can be generalized to the case where, instead of a mean-field generator, we have a time-dependent generator $\mcA:[0,T]\to\mcG_p^0(\Pi)$, and the following linear problem:
\begin{align}\label{eq4.9:linear-fokker-planck}
	\left\{\begin{aligned}
		\partial_t \rho_t &= \rho_t \mcA_t,\quad t\in(0,T),\\
		\rho_0 &\in \mcP_p(\Pi).
	\end{aligned}\right.
\end{align}
The development of the theory presents two main challenges.

(1) \emph{Notions of solutions.} There are several notions of solutions in the literature for the nonhomogeneous Fokker--Planck problem \eqref{eq4.9:linear-fokker-planck} and the corresponding mean-field equation \eqref{eq4.6:mean-field-equation}, together with their associated well-posedness theories. For example, Kato \cite{kato1953integration} studied nonhomogeneous linear problems with time-dependent operators and established conditions ensuring well-posedness; see also the discussion in \cite[Section VI.9]{engel2000one}. A familiar interpretation of the measure-valued equation is the weak formulation
\begin{align*}
    \inn{\rho_t,\phi}
    = \inn{\rho_0,\phi}
    + \int_0^t \inn{\rho_s,\mcA_s\phi}\,ds,
\end{align*}
for test functions $\phi$ in a common subspace of the domains $D(\mcA_s)$, with suitable measurability and integrability of the integrand.

However, our transport-stability assumptions do not impose the common-domain or operator-continuity conditions used in such formulations and well-posedness theories. The domains $D(\mcA_t)$ may vary with time, and their intersection is not assumed to contain a class of test functions that separates probability measures. Moreover, the Wasserstein-$p$ stability estimates needed throughout this work do not follow directly from the cited classical theory. We therefore introduce the notion of $\mcW_p$-stable solutions and develop its well-posedness theory through piecewise constant approximation. This construction provides the required stability estimates and, under additional common-domain and continuity assumptions, can be connected to the weak formulation above.

(2) \emph{Continuity of the curve of generators.} Ensuring the well-posedness of the nonhomogeneous problem \eqref{eq4.9:linear-fokker-planck} requires some form of continuity assumption on the map $t \mapsto \mcA_t$. While we do not aim to develop a detailed topological structure for the space $\mcG_p^0(\Pi)$ of Feller generators or for the space of curves of Feller generators, as this could become quite technical, we can utilize the exponential stability condition from Section \ref{sec:c-exp} to provide a basic notion of continuity on the space of $\mcG_p^0(\Pi)$-curves; see Remark \ref{rmk3.22:temp}. This ``rudimentary'' topology on curves of generators will be sufficient for establishing the well-posedness of the nonhomogeneous problem.

We remark that there have been previous approaches to evolution equations for curves of probability measures, notably the JKO scheme in the Euclidean setting \cite{jordan1998variational,LEE2015131}. Our approach shares with the JKO scheme the underlying idea of constructing solutions via time-discretized approximations. It is further augmented by quantitative Wasserstein-\(p\) stability estimates derived from the transport derivative \(\om_p\) and the \(\mathrm{Exp}_p\)-condition of Section \ref{sec:c-exp}, which make it naturally compatible with the optimal transport framework. Rather than imposing strong operator-topology assumptions on the generator curve, the method relies only on a quantitative continuity condition; consequently, the resulting notion of $\mcW_p$-stable solution already encodes the desired stability estimate.

%=============================================================================
\subsubsection{Organization}
This section consists of two remaining subsections. In the next subsection, we will address the well-posedness of problem \eqref{eq4.9:linear-fokker-planck} using the \emph{method of piecewise constant approximation}. We will also define the notion of $\mcW_p$-stable solutions there.
The final subsection will cover the linearized problem \eqref{eq4.7:linearized-mf-eq} and the mean-field problem \eqref{eq4.6:mean-field-equation}, as a consequence of the well-posedness result established in the next subsection.

%=============================================================================
\subsection{Nonhomogeneous linear Fokker-Planck equations and their well-posedness}\label{sec4.4}
The main focus of this subsection is to establish the well-posedness theory for the initial value problem of the \emph{nonhomogeneous linear Fokker-Planck equation} \eqref{eq4.9:linear-fokker-planck}.
We will establish the well-posedness of the initial value problem under a suitable notion of solutions and appropriate continuity assumptions on the generators $t\mapsto\mcA_t$. Throughout this section, we will adopt the following notation for ``curves of generators''.
\begin{notation}
	A \emph{curve of $\GPZP$-generators} $\bs{\mcA}=\cb{\mcA(t)}_{t\in[0,T]}$ over an interval $[0,T]\subset\mbr$ is a function $\mcA:[0,T]\to\GPZP$. We sometimes write $\mcA(t)=\mcA_t$.
\end{notation}

We will adopt the \emph{method of piecewise constant approximation (PCA)} to establish the well-posedness of \eqref{eq4.9:linear-fokker-planck}. 
The central idea is to approximate the equation \eqref{eq4.9:linear-fokker-planck} using a piecewise constant approximation of the generator curve \( \{\mA_t\}_{t \ge 0} \), for which well-posedness is ensured by classical semigroup theory. 
The approach involves taking the limit of the solutions to the piecewise constant approximating equations to derive the solution for \eqref{eq4.9:linear-fokker-planck}. 
As mentioned earlier, the main technicality is to provide a (sufficient) continuity assumption on the curves $\bs\mcA$ that guarantees the existence and uniqueness of such a limit.

%=============================================================================
\subsubsection{The case of piecewise constant generators and their well-posedness}\label{subsec4.3.2}
Following the plan, we begin by considering curves $\mcA:[0,T]\to\GPZP$ that are piecewise constant and show that the nonhomogeneous linear Fokker-Planck equation \eqref{eq4.9:linear-fokker-planck} associated with such curves is well-posed. 
Let us give the following definition.
\begin{definition}
	A curve of $\GPZP$-generators $\bs{\mcA}$ is \emph{(right-continuous) piecewise constant} if there is a partition $0=t_0<t_1<\cdots<t_K=T$ such that for each $1\le k\le K$,
	\begin{align*}
		\mcA(t)=\mcA(t_{k-1}),\quad \mbox{ if } t\in [t_{k-1},t_k).
	\end{align*}
	We denote by $\RPC([0,T];\GPZP)$ \emph{the space of all (right-continuous) piecewise constant curves of $\GPZP$-generators}.
\end{definition}

\begin{definition}
	Given $\mcA:[0,T]\to\mcG(\Pi)$, we say that $\rho\in C([0,T];\mcP(\Pi))$ is a \emph{strong dual} solution to \eqref{eq4.9:linear-fokker-planck} if for each $t\in(0,T)$,
	\begin{align*}
		\lim_{h\searrow 0} \frac{\inn{\rho_{t+h},\phi}-\inn{\rho_t,\phi}}{h} = \inn{\rho_t,\mcA_t\phi}\quad\text{for all }\phi\in D(\mcA_t).
	\end{align*}
\end{definition}

\begin{remark}
	For a constant curve $\mcA_t\equiv\mcA$, $\scb{\rho_t}_{t}\in C([0,T];\mcP(\Pi))$ defined by $\rho_t=\rho_0 e^{t\mcA}$ is the unique strong dual solution of \eqref{eq4.9:linear-fokker-planck} by classical semigroup theory.
\end{remark}

\begin{proposition}\label{lem4.10:exist-unique-for-rpc}
	If $\bs\mcA\in\RPC([0,T];\GPZP)$, then for any initial condition $\rho_0\in\PPP$, \eqref{eq4.9:linear-fokker-planck} has a unique strong dual solution $\rho\in \CTPC$.
\end{proposition}
\begin{proof}
	Existence and uniqueness reduce to the case of a constant generator. Specifically, the solution is obtained by solving the problem on each interval where the generator remains constant, then gluing the solutions together in a piecewise manner. To elaborate further, let $\cb{t_k}_{k=0}^K$ be the corresponding partition of $\bs\mcA=\cb{\mcA_t}_{t\in[0,T]}$, and define the solution recursively at $t\in\cb{t_k}_{k=0}^K$ by
	\begin{align*}
		\rho_{t_0}=\rho_0,\quad \rho_{t_k} = \rho_{t_{k-1}} e^{(t_k-t_{k-1})\mcA(t_{k-1})}.
	\end{align*}
	For other $t\in[0,T]$, we have
	\begin{align*}
		\rho_t = \rho_{t_{k-1}} e^{(t-t_{k-1})\mcA(t_{k-1})},\quad\text{if }t\in[t_{k-1},t_k).
	\end{align*}
	Then $\rho=\cb{\rho_t}_{t\in[0,T]}$ is a strong dual solution to \eqref{eq4.9:linear-fokker-planck}. Uniqueness follows by applying the uniqueness of the constant case on each interval $[t_{k-1},t_k)$.
\end{proof}

For later use, let us introduce the following.
\begin{definition}[Solution Operator]\label{def4.11:sol-op-for-rpc}
	We define
	\[\mbs_0:\PPP\times\RPC([0,T];\GPZP)\to \CTPC\]
	by letting $\mbs_0(\rho_0;\bs{\mcA})$ be the unique strong dual solution $\rho\in \CTPC$ of \eqref{eq4.9:linear-fokker-planck}.
\end{definition}

We now prove the following integral inequality, which is needed to obtain a stability estimate in the coming discussion.
\begin{proposition}\label{lem4.12:stability-for-rpc}
	Suppose that \(\bs\mcA,\bs\mcB\in\RPC([0,T];\GPZP)\), \(\rho_0,\si_0\in\PPP\), and let \(\rho = \mbs_0(\rho_0;\bs{\mcA})\), \(\si = \mbs_0(\si_0;\bs{\mcB})\). Let \(\ta:[0,T]\to\mbr\) be finite and upper semicontinuous and satisfy
	\[
	\om_p(\rho_s,\si_s;\mcA_s,\mcB_s)
	\le
	\ta(s),
	\qquad
	s\in[0,T].
	\]
	Then
	\[
	\mcW_p^p(\rho_t,\si_t)
	\le
	\mcW_p^p(\rho_0,\si_0)
	+
	\int_0^t
	\ta(s)
	\,ds,
	\qquad
	0\le t\le T.
	\]
	In particular, the bound holds with
	\(
	\ta(s)
	=
	\bigl(
	\om_p(\rho_\cdot,\si_\cdot;\mcA_\cdot,\mcB_\cdot)
	\bigr)^\vee(s),
	\)
	the upper semicontinuous envelope of the transport derivative along the curves, provided that the envelope is finite on $[0,T]$.
\end{proposition}

The proof reduces to deriving an integral estimate for \(F(t):=\mcW_p^p(\rho_t,\sigma_t)\) from a bound on its right upper Dini derivative \(D^+F\), which is the transport derivative \(\om_p(\rho_t,\sigma_t;\mcA_t,\mcB_t)\). If \(F\) were differentiable, the estimate would follow immediately from the fundamental theorem of calculus. Since differentiability is not guaranteed here, we use the following elementary lemma, whose proof is omitted.

\begin{lemma}\label{lem:saks-consequence}
	Let \(F:[0,T]\to[0,\infty)\) and \(G:[0,T]\to\mbr\) be continuous functions. Suppose that for all \(t\in[0,T)\), 
	\[
	D^+F(t):=\limsup_{h\searrow0}\frac{F(t+h)-F(t)}{h}\le G(t).
	\]
	Then
	\[
	F(t)-F(s)\le\int_s^t G(\tau)\,d\tau, \qquad 0\le s\le t\le T.
	\]
\end{lemma}

Let us now prove Proposition \ref{lem4.12:stability-for-rpc}.
\begin{proof}[Proof of Proposition \ref{lem4.12:stability-for-rpc}]
	It suffices to prove the proposition in the case where \(\ta\) is continuous. Indeed, any finite upper semicontinuous function on \([0,T]\) can be approximated pointwise from above by a decreasing sequence of continuous functions. The general case then follows from the monotone convergence theorem.
	
    Consider first the case that $\bs\mcA\equiv\mcA,\bs\mcB\equiv\mcB$ are constant generators in $\GPZP$. Then the solutions $\mbs_0(\rho_0;\bs{\mcA})$ and $\mbs_0(\si_0;\bs{\mcB})$ are given by $\rho_0 e^{t\mcA}$ and $\si_0 e^{t\mcB}$, respectively. By the assumption and the definition of $p$-transport derivative, we have
    \begin{align*}
        D^+|_{t=s} \mcW_p^p(\rho_t,\si_t )= \om_p(\rho_s,\si_s;\mcA,\mcB)\le\ta(s).
    \end{align*}
    Then by Lemma \ref{lem:saks-consequence}, it holds for all $0\le s\le t\le T$ that
    \begin{align}\label{eq4.XX:temp}
        \mcW_p^p(\rho_0 e^{t\mcA},\si_0 e^{t\mcB}) \le \mcW_p^p(\rho_0 e^{s\mcA},\si_0 e^{s\mcB}) + \int_s^t \ta(\tau)\,d\tau.
    \end{align}
    The general case of piecewise constant generators follows from \eqref{eq4.XX:temp} and a simple induction argument over the partition. 
    First, by refining the partitions, we may assume that $\bs\mcA$ and $\bs\mcB$ have a common partition $\cb{t_k}_{k=0}^K$. 
    Clearly it suffices to prove the integral inequality for the terminal time $t=t_K$.
    For each $1\le k\le K$, we have the following recursive bound by \eqref{eq4.XX:temp}:
    \begin{align*}
        \mcW_p^p(\rho_{t_k},\si_{t_k}) &= \mcW_p^p(\rho_{t_{k-1}} e^{(t_k-t_{k-1})\mcA(t_{k-1})} , \si_{t_{k-1}} e^{(t_k-t_{k-1})\mcB(t_{k-1})}) \\
        &\le \mcW_p^p(\rho_{t_{k-1}},\si_{t_{k-1}}) + \int_{t_{k-1}}^{t_k} \ta(s)\,ds.
    \end{align*}
    By induction, we find the estimate for the terminal time $t_K$:
    \[\mcW_p^p(\rho_{t_K},\si_{t_K}) \le \mcW_p^p(\rho_0,\si_0) + \int_{0}^{t_K} \ta(s)\,ds.\qedhere\]
\end{proof}

%=============================================================================
\subsubsection{\texorpdfstring{$\Expp$}{Exp_p}-continuous curves of generators and PCA scheme}\label{subsec4.3.4}
Let us now return to the linear problem \eqref{eq4.9:linear-fokker-planck}. We will prove the well-posedness of the problem for curves of generators \(\bs\mcA=\{\mcA_t\}_{t\in[0,T]}\) whose time continuity is measured through the transport derivative.

\begin{definition}[$\Expp$-continuity]\label{def4.askjb}
	A curve $\bs\mcA=\cb{\mcA_t}_{t\in[0,T]}$ is \emph{$\Expp$-continuous} if
	\begin{enumerate}[(i)]
		\item $\mcA_t\in\GPZP$ for all $t\in[0,T]$.
		\item There exists $\be\ge0$ such that, for every $\ep>0$, there exists $\de>0$ such that if $|t-s|<\de$, then
		\begin{align*}
			\om_p(\mu,\nu;\mcA_t,\mcA_s) \le \ep+\be\mcW_p^p(\mu,\nu),\quad\mu,\nu\in\PPP.
		\end{align*}
		Namely, if $|t-s|<\de$, then $(\mcA_t,\mcA_s)\in\Expp(\ep,\be)$.
	\end{enumerate}
	We denote by $\CEXP([0,T];\GPZP)$ the family of all $\Expp$-continuous curves of generators.
\end{definition}
\begin{remark}
	Definition \ref{def4.askjb}(ii) is equivalent to: there exists $\be\ge0$ and $\ta:[0,\infty)\to[0,\infty)$ with $\lim_{s\searrow0}\ta(s)=0$ such that
	\begin{align*}
		\om_p(\mu,\nu;\mcA_t,\mcA_s) \le \ta(|t-s|)+\be\mcW_p^p(\mu,\nu),\quad\mu,\nu\in\PPP.
	\end{align*}
\end{remark}
We note that it is possible to establish the well-posedness of \eqref{eq4.9:linear-fokker-planck} under weaker continuity assumptions, such as piecewise continuous generators. However, we do not intend to explore these cases in detail. Instead, we will establish well-posedness of \eqref{eq4.9:linear-fokker-planck} under a sufficient continuity assumption.

Our method of constructing a solution in the case of $\Expp$-continuous curves of generators is based on the \emph{piecewise constant approximation (PCA)} of these curves. We begin by providing the following definition.
\begin{definition}[Piecewise constant approximation]\label{def4.16}
	Let $\bs\mcA\in \CEXP([0,T];\GPZP)$.
	\begin{enumerate}[(i)]
		\item Let $\Delta = \{0=t_0<t_1<\cdots<t_n=T\}$ be a partition of $[0,T]$. We write
		\begin{align*}
			\mathrm{mesh}(\Delta) = \max_{1\le k\le n}|t_k-t_{k-1}|.
		\end{align*}	
		The \emph{(right-hand) piecewise constant approximation of $\bs\mcA$ over the partition $\Delta$} is the curve $\bs{\mcA}^{\De}\in \RPC([0,T];\GPZP)$ given by
		\begin{align*}
			\mcA^\Delta(t)&= \mcA(\tau(t)),\qquad \tau(t)= \min\cb{t_k:0\le k\le n, t<t_k},\quad \tau(T)=T.
		\end{align*}
		That is, $\mcA^\Delta(t)=\mcA(t_k)$ if $t\in [t_{k-1},t_{k})$.
		\item Let $\cb{\De_n}_n$ be a sequence of partitions of $[0,T]$ such that $\lim_{n\to\infty} \mathrm{mesh}(\De_n)=0$. The sequence of curves $\{\bs\mcA^{(n)}\}_n$, where $\bs\mcA^{(n)}=\bs\mcA^{\De_n}$, will be called a \emph{piecewise constant approximating sequence} (in short, \emph{PCA sequence}) of $\bs\mcA$.
	\end{enumerate}
\end{definition}

Here is the main theorem for the construction of solutions to \eqref{eq4.9:linear-fokker-planck}. Recall from Definition \ref{def4.11:sol-op-for-rpc} that \(\mbs_0(\rho_0;\bs{\mcA})\) denotes the unique strong dual solution of \eqref{eq4.9:linear-fokker-planck} when \(\bs{\mcA}\in\RPC([0,T];\GPZP)\). Recall also from Section \ref{subsec2.4.3} the uniform metric
\[
\mcD_p(\rho,\sigma)
=
\sup_{t\in[0,T]}\mcW_p(\rho_t,\sigma_t),
\]
which makes \(C([0,T];\mcp_p(\Pi))\) a complete metric space.

\begin{theorem}\label{linear-fokker-planck-thm}
	Suppose $\bs\mcA\in\CEXP([0,T];\GPZP)$. For any $\rho_0\in \PPP$, there exists a unique $\rho\in \CTPC$ such that the following holds: for every PCA sequence $\{\bs\mcA^{(n)}\}_n$ of $\bs\mcA$, we have that $\rho^{(n)}:=\mbs_0(\rho_0;\bs{\mcA}^{(n)})$ converges to $\rho$ in the sense that
	\begin{align*}
		\lim_{n\to\infty} \mcD_p(\rho,\rho^{(n)})=0. 
	\end{align*}
\end{theorem}
\begin{proof}
	Let $\{\bs\mcA^{(n)}\}_n$ be a PCA sequence of $\bs\mcA$ and denote $\rho^{(n)} =\mbs_0(\rho_0;\bs\mcA^{(n)})$ for each $n$. We first show that $\scb{\rho^{(n)}}_n$ is $\mcD_p$-Cauchy. We can write
	\begin{align*}
		\mcA^{(n)}(t)=\mcA(\tau_n(t)),
	\end{align*}
	where $\{\tau_n\}_n$ is a sequence of ``time-changes'' as in Definition \ref{def4.16}, i.e. $\tau_n:[0,T]\to[0,T]$ is a sequence of increasing step functions such that $\tau_n(t)\to t$ uniformly as $n\to\infty$.
	
	Consider the expression
	\begin{align*}
		\om_p(\rho^{(m)}_s,\rho^{(n)}_s;  \mcA^{(m)}_s,\mcA^{(n)}_s) &= \om_p(\rho^{(m)}_s,\rho^{(n)}_s;  \mcA(\tau_m(s)),\mcA(\tau_n(s))).
	\end{align*}
	Since $\bs\mcA\in\CEXP([0,T];\GPZP)$, there exists $\be\ge0$ such that, for every $\ep>0$, there exists $\de>0$ such that if $|t-s|<\de$, then
	\begin{align*}
		\om_p(\mu,\nu;\mcA_t,\mcA_s) \le \ep+\be\mcW_p^p(\mu,\nu),\quad\mu,\nu\in\PPP.
	\end{align*}
	Observe that
	\begin{align*}
		\abs{\tau_m(s)-\tau_n(s)} \le \abs{\tau_m(s)-s} + \abs{\tau_n(s)-s} \le \mathrm{mesh}(\De_m) + \mathrm{mesh}(\De_n),\quad\forall s\in[0,T].
	\end{align*}
	Since $\mathrm{mesh}(\De_n)\xrightarrow{n\to\infty}0$, we can choose $N$ large enough so that $\mathrm{mesh}(\De_n)<\de/2$ for all $n\ge N$.
	Then whenever $m,n\ge N$,
	\begin{align*}
		\om_p(\rho^{(m)}_s,\rho^{(n)}_s;  \mcA^{(m)}_s,\mcA^{(n)}_s) &= \om_p(\rho^{(m)}_s,\rho^{(n)}_s;  \mcA(\tau_m(s)),\mcA(\tau_n(s))) \le \ep + \be\mcW_p^p(\rho^{(m)}_s,\rho^{(n)}_s).
	\end{align*}
	The right-hand side is continuous in \(s\), since \(\rho^{(m)}\) and \(\rho^{(n)}\) are continuous curves in \((\mcp_p(\Pi),\mcW_p)\) and the Wasserstein-\(p\) metric is continuous with respect to \(\mcW_p\)-convergence.
	Then by Proposition \ref{lem4.12:stability-for-rpc},
	\begin{align*}
		\mcW_p^p(\rho^{(m)}_t,\rho^{(n)}_t) \le \int_0^t \sqb{\ep + \be\mcW_p^p(\rho^{(m)}_s,\rho^{(n)}_s)}\,ds.
	\end{align*}
	By Gr\"{o}nwall's inequality, we obtain
	\begin{align*}
		\mcD_p(\rho^{(m)},\rho^{(n)})\le \sqb{\ep\zeta_\be(T)}^{1/p},
	\end{align*}
	where we recall from \eqref{eq3.25:zeta-be} that $\zeta_\be(t)=\be^{-1}(e^{\be t}-1)$ if $\be >0$ and $=t$ if $\be=0$.
	This shows that $\scb{\rho^{(n)}}_n$ is Cauchy. By the completeness of $\CTPC$, we find that the limit $\rho\in \CTPC$ exists.
	
	To see uniqueness of such a limit, let $\{\bs{\tilde\mcA}^{(n)}\}_n$ be another PCA sequence of $\bs\mcA$ and $\tilde\rho^{(n)} = \mbs_0(\rho_0;\bs{\tilde\mcA}^{(n)})$, with $\tilde\rho$ being the limit. Repeating the argument above, we find that $\mcD_p(\rho^{(n)},\tilde \rho^{(n)})$ 
	can be arbitrarily small whenever $n$ is large. By the triangle inequality of $\mcD_p$, their limits satisfy
	\begin{align*}
		\mcD_p(\rho,\tilde \rho)&\le \mcD_p(\rho,\rho^{(n)})+\mcD_p(\rho^{(n)},\tilde \rho^{(n)})+ \mcD_p(\tilde \rho^{(n)},\tilde \rho),
	\end{align*}
	and $n$ can be chosen to be large enough so that the right-hand side is arbitrarily small. This shows that $\rho=\tilde\rho$, so the limit is independent of the PCA sequence.
\end{proof}

%=============================================================================
\subsubsection{The notion of \texorpdfstring{$\mcW_p$}{W_p}-stable solutions and its well-posedness}
With the theorem above established, we may now give the notion of solutions.
\begin{definition}[$\mcW_p$-stable solutions]
	Given a curve of generators $\bs\mcA=\cb{\mcA_t}_{t\in[0,T]}$, we say $\rho\in \CTPC$ is a \emph{$\mcW_p$-stable solution} of \eqref{eq4.9:linear-fokker-planck} if the following holds: for every PCA sequence $\scb{\bs\mcA^{(n)}}$ of $\bs\mcA$, we have $\rho^{(n)}:=\mbs_0(\rho_0;\bs{\mcA}^{(n)})$ converges to $\rho$ in the sense that
	\begin{align*}
		\lim_{n\to\infty}\mcD_p(\rho,\rho^{(n)})= \lim_{n\to\infty}\sup_{t\in[0,T]}\mcW_p(\rho_t,\rho_t^{(n)})=0.
	\end{align*}
\end{definition}

\begin{remark}
	For fixed initial data $\rho_0$, the definition implies that a $\mcW_p$-stable solution, if it exists, is unique, since the space $\CTPC$ is Hausdorff. In this case, we also denote by $\mbs(\rho_0;\bs\mcA)$ the $\mcW_p$-stable solution with initial data $\rho_0$ for \eqref{eq4.9:linear-fokker-planck}.
\end{remark}

\begin{remark}\label{rem:stable-weak-formulation}
    Suppose that there is a subspace $\mathcal D$, dense in $C_0(\Pi)$ with respect to the uniform norm, such that
    \[
        \mathcal D\subseteq\bigcap_{s\in[0,T]}D(\mcA_s),
    \]
    and, for every $\phi\in\mathcal D$, the map $s\mapsto\mcA_s\phi$ is continuous in $C_0(\Pi)$ with respect to the uniform norm. Then the $\mcW_p$-stable solution satisfies the weak formulation of the evolution equation:
    \begin{align*}
        \inn{\rho_t,\phi}
        =\inn{\rho_0,\phi}
        +\int_0^t\inn{\rho_s,\mcA_s\phi}\,ds,
        \qquad t\in[0,T],\quad\phi\in\mathcal D.
    \end{align*}
    Indeed, the piecewise constant approximations satisfy this identity with their corresponding generators. Uniform continuity of $s\mapsto\mcA_s\phi$ and convergence of the approximating curves uniformly in Wasserstein distance allow passage to the limit. To justify convergence of the integrands uniformly in time, note that $\{\mcA_s\phi:s\in[0,T]\}$ is compact in $C_0(\Pi)$ and can be approximated by finitely many test functions in the uniform norm. For the nonlinear equation, the same conclusion holds when these conditions are satisfied by $\mcA_s=\mcA(\rho_s)$. This establishes consistency with the weak formulation; uniqueness among all weak solutions is a separate question.
\end{remark}

In conclusion, we arrive at the following existence and uniqueness result for \eqref{eq4.9:linear-fokker-planck}.
\begin{corollary}\label{cor4.19temp}
	For every $\bs\mcA\in\CEXP([0,T];\GPZP)$ and $\rho_0\in\PPP$, there exists a unique $\mcW_p$-stable solution to the initial value problem \eqref{eq4.9:linear-fokker-planck}.
\end{corollary}

%=============================================================================
\subsubsection{Stability estimates for $\mcW_p$-stable solutions}
Before closing this subsection, we prove the integral inequalities that will be used to obtain stability estimates later.

For piecewise constant curves of generators \(\bs\mcA,\bs\mcB\), the corresponding solutions satisfy
\[
\left.D^+\right|_{t=s}
\mcW_p^p(\rho_t,\si_t)
=
\om_p\bigl(\rho_s,\si_s;\mcA_s,\mcB_s\bigr).
\]
Hence, by Proposition \ref{lem4.12:stability-for-rpc}, a bound on the \(p\)-transport derivative yields an integral bound on the Wasserstein distance.
For general \(\Expp\)-continuous curves of generators, this identity may fail because of the time dependence of the generators. The following lemma extends the integral bound to $\mcW_p$-stable solutions in this setting by passing through piecewise constant approximations. Throughout, for \(\bs\mcA\in\CEXP([0,T];\GPZP)\) and \(\rho_0\in\mcp_p(\Pi)\), using the same notation from Definition \ref{def4.11:sol-op-for-rpc} we denote by
\[
\rho=\mbs(\rho_0;\bs\mcA)\in C([0,T];\mcp_p(\Pi))
\]
the unique $\mcW_p$-stable solution of \eqref{eq4.9:linear-fokker-planck}.

\begin{lemma}\label{lem4.askjba}
	Suppose \(\bs\mcA,\bs\mcB\in\CEXP([0,T];\GPZP)\), \(\rho_0,\si_0\in\PPP\), and let \(\rho = \mbs(\rho_0;\bs\mcA)\) and \(\si = \mbs(\si_0;\bs\mcB)\). Let \(\ta:[0,T]\times\PPP^2\to\mbr\) be continuous. If
	\[
	\om_p(\mu,\nu;\mcA_s,\mcB_s)
	\le
	\ta(s,\mu,\nu)
	\qquad
	\text{for all }s\in[0,T],\ \mu,\nu\in\PPP,
	\]
	then
	\[
	\mcW_p^p(\rho_t,\si_t)
	\le
	\mcW_p^p(\rho_0,\si_0)
	+
	\int_0^t
	\ta(s,\rho_s,\si_s)
	\,ds,
	\qquad
	t\in[0,T].
	\]
\end{lemma}

\begin{remark}
	Here continuity of \(\ta\) is understood with respect to the product topology on \([0,T]\times\mcp_p(\Pi)^2\), where \(\mcp_p(\Pi)\) is equipped with the Wasserstein-\(p\) topology. Equivalently, if \(s_n\to s\), \(\mu_n\xrightarrow{\mcW_p}\mu\), and \(\nu_n\xrightarrow{\mcW_p}\nu\), then \(\ta(s_n,\mu_n,\nu_n)\to\ta(s,\mu,\nu)\). In fact, the continuity assumption may be relaxed to upper semicontinuity plus local boundedness; the proof remains valid with only minor changes.
\end{remark}

\begin{proof}[Proof of Lemma \ref{lem4.askjba}]
	Consider PCA sequences $\scb{\bs\mcA^{(n)}}_n$, $\scb{\bs\mcB^{(n)}}_n$ of $\bs\mcA$, $\bs\mcB$ such that for each $n$, $\bs\mcA^{(n)}$ and $\bs\mcB^{(n)}$ have a common partition. We can write
	\begin{align*}
		\mcA^{(n)}(s)=\mcA(\tau_n(s)),\quad \mcB^{(n)}(s)=\mcB(\tau_n(s)),
	\end{align*}
	where $\tau_n$ is a sequence of ``time-changes'' as before (see Definition \ref{def4.16}). Let $\rho^{(n)} = \mbs(\rho_0;\bs\mcA^{(n)})$ and $\si^{(n)} = \mbs(\si_0;\bs\mcB^{(n)}).$
	Notice that by the assumption
	\begin{align*}
		\om_p(\rho_s^{(n)},\si_s^{(n)};\mcA^{(n)}(s),\mcB^{(n)}(s)) 
		\le \ta(\tau_n(s),\rho_s^{(n)},\si_s^{(n)}) 
	\end{align*}
	By Proposition \ref{lem4.12:stability-for-rpc}, we have
	\begin{align*}
		\mcW_p^p(\rho_t^{(n)},\si_t^{(n)}) &\le \mcW_p^p(\rho_0,\si_0) + \int_0^t \ta(\tau_n(s),\rho_s^{(n)},\si_s^{(n)})^\vee\,ds,
	\end{align*}
	where $\ta(\tau_n(s),\rho_s^{(n)},\si_s^{(n)})^\vee$ denotes the upper semicontinuous envelope of $s\mapsto\ta(\tau_n(s),\rho_s^{(n)},\si_s^{(n)})$.
	By the assumption on $\ta$, since $s\mapsto\ta(s,\rho_s^{(n)},\si_s^{(n)})$ is upper semicontinuous and $\tau_n$ is piecewise constant with finite jumps, we see that $s\mapsto\ta(\tau_n(s),\rho_s^{(n)},\si_s^{(n)})$ is continuous except at those finite jump points. Hence, $\ta(\tau_n(s),\rho_s^{(n)},\si_s^{(n)})$ agrees with its envelope almost everywhere on $[0,T]$, and thus
	\begin{align}\label{eq4.asdbksdj}
		\mcW_p^p(\rho_t^{(n)},\si_t^{(n)}) &\le \mcW_p^p(\rho_0,\si_0) + \int_0^t {\ta(\tau_n(s),\rho_s^{(n)},\si_s^{(n)})}\,ds. 
	\end{align}
	
	Next, we need a uniform bound in \(n\) for the integrand in \eqref{eq4.asdbksdj}, in order to apply the reverse Fatou lemma. This follows by a simple compactness argument. Since \(\rho^{(n)}\to\rho\) in \(C([0,T];\mcp_p(\Pi))\) with respect to \(\mcD_p\), the set \(K=\{\rho\}\cup\{\rho^{(n)}:n\ge1\}\) is compact in the curve space. The evaluation map \(\mathrm{ev}(\gamma,t)=\gamma_t\) is continuous from \(C([0,T];\mcp_p(\Pi))\times[0,T]\) to \(\mcp_p(\Pi)\). Hence
	\[
	\{\rho_t^{(n)}:n\ge1,\ t\in[0,T]\}
	\subseteq
	\mathrm{ev}(K\times[0,T])
	\]
	is relatively compact in \(\mcp_p(\Pi)\). The same holds for the \(\sigma\)-curves. Since \(\theta\) is continuous, it is bounded on the product of \([0,T]\) with these relatively compact sets, and therefore
	\[
	\sup_{s\in[0,T],\,n\ge1}
	\theta\bigl(\tau_n(s),\rho_s^{(n)},\sigma_s^{(n)}\bigr)
	<\infty.
	\]
	
	Taking the limit superior as $n\to\infty$ in \eqref{eq4.asdbksdj} and using the reverse Fatou lemma, we find that
	\begin{align*}
		\mcW_p^p(\rho_t,\si_t) &\le \mcW_p^p(\rho_0,\si_0) + \int_0^t \limsup_{n\to\infty}\ta(\tau_n(s),\rho_s^{(n)},\si_s^{(n)})\,ds\\
		&\le \mcW_p^p(\rho_0,\si_0) + \int_0^t {\ta(s,\rho_s,\si_s)}\,ds.\qedhere
	\end{align*}
\end{proof}

We now state a similar result to Lemma \ref{lem4.askjba}, but with $\mcA$ being a constant generator in $\GPZP$. However, note that a constant curve of generators in $\GPZP$ is not necessarily $\Expp$-continuous, that is, it might not be in $\CEXP([0,T];\GPZP)$. Hence, the following lemma cannot be stated as a special case of Lemma \ref{lem4.askjba} above. We shall skip the proof of this lemma as it is similar to the proof of Lemma \ref{lem4.askjba}. This result will be useful in the coming section.
\begin{lemma}\label{special-stability}
	Suppose $\mcA\in\GPZP$, $\bs\mcB\in\CEXP([0,T];\GPZP)$ and $\rho_0,\si_0\in\PPP$. Let $\rho_t = \rho_0 e^{t\mcA}$ and $\cb{\si_t}_{t} = \mbs(\si_0;\bs\mcB)$.
	Let $\ta:[0,T]\times\PPP^2\to\mbr$ be a continuous function.
	If $\om_p(\mu,\nu;\mcA_s,\mcB_s)\le \ta(s,\mu,\nu)$ on $[0,T]\times\PPP^2$, then it holds
	\begin{align*}
		\mcW_p^p(\rho_t,\si_t)\le \mcW_p^p(\rho_0,\si_0) + \int_0^t \ta(s,\rho_s,\si_s)\,ds,\quad t\in[0,T].
	\end{align*}
\end{lemma}

%=============================================================================
\subsection{Well-posedness of mean-field equations (proof of Theorems \ref{thm4.345:chap4-main-result-linearized} and \ref{thm4.3:chap4-main-result})}\label{sec4.5}
Now, we return to the discussion of the mean-field problem \eqref{eq4.6:mean-field-equation}, where our goal is to prove Theorem \ref{thm4.345:chap4-main-result-linearized}, followed by Theorem \ref{thm4.3:chap4-main-result}. As mentioned, our strategy is to first consider the linearized problem, and then apply a fixed point argument.

%=============================================================================
\subsubsection{Well-posedness of the linearized problem (proof of Theorem \ref{thm4.345:chap4-main-result-linearized})}
Consider the linearized problem of the mean-field equation \eqref{eq4.7:linearized-mf-eq}, which reads
\begin{align*}
	\left\{\begin{aligned}
		\partial_t \rho_t &= \rho_t\mcA(\mu_t),\quad t\in(0,T),\\
		\rho_0 &\in \PPP,
	\end{aligned}\right.
\end{align*}
where $\mu\in \CTPC$ is given. Utilizing the results from Section \ref{sec4.4}, we can show that this linearized problem is well-posed, under the assumption that the mean-field generator $\mcA$ satisfies Hypothesis {\refA}.
\begin{proof}[Proof of Theorem \ref{thm4.345:chap4-main-result-linearized}]
	(i) Let $\mu\in \CTPC$ and $\rho_0\in\PPP$.
	First, observe that the curve of generators $t\mapsto\mcA(\mu_t)$ is in $\CEXP([0,T];\GPZP)$. Indeed, by Hypothesis {\refA}, there exist constants $\al,\be\ge0$ such that, for all $\mu',\nu'\in\PPP$ and $t,s\in[0,T]$,
	\begin{align*}
		\om_p(\mu',\nu';\mcA(\mu_t),\mcA(\mu_s)) \le \al\mcW_p^p(\mu_t,\mu_s)+\be\mcW_p^p(\mu',\nu').
	\end{align*}
	Since $(t,s)\mapsto\mcW_p(\mu_t,\mu_s)$ is continuous on $[0,T]^2$, which is a compact domain in $\mbr^2$, it is uniformly continuous. That is, $\mcW_p^p(\mu_t,\mu_s)$ is small whenever $|t-s|$ is small (independent of the actual values of $t,s$). This means that $t\mapsto\mcA(\mu_t)$ is in $\CEXP([0,T];\GPZP)$. The existence and uniqueness of a $\mcW_p$-stable solution $\mbs(\rho_0,\bs\mcA(\mu))$ then follow from Corollary \ref{cor4.19temp}.
	
	(ii) Let $\rho=\mbs(\rho_0,\bs\mcA(\mu))$ and $\si=\mbs(\si_0,\bs\mcA(\nu))$, where $\si_0\in\PPP$ and $\nu\in \CTPC$. We have
	\begin{align*}
		\om_p(\mu',\nu';\mcA(\mu_s),\mcA(\nu_s)) \le \al\mcW_p^p(\mu_s,\nu_s)+\be\mcW_p^p(\mu',\nu'),
	\end{align*}
	where the right-hand side is a continuous function of $(s,\mu',\nu')$. By Lemma \ref{lem4.askjba},
	\begin{align*}
		\mcW_p^p(\rho_t,\si_t) &\le \mcW_p^p(\rho_0,\si_0) + \int_0^t \sqb{\al\mcW_p^p(\mu_s,\nu_s)+\be\mcW_p^p(\rho_s,\si_s)}\,ds\\
		&\le \mcW_p^p(\rho_0,\si_0) + \int_0^t \sqb{\al \mcD_p(\mu,\nu)^p+\be\mcW_p^p(\rho_s,\si_s)}ds.
	\end{align*}
	It then follows from Gr\"{o}nwall's inequality that
	\[\mcW_p^p(\rho_t,\si_t) \le \mcW_p^p(\rho_0,\si_0) e^{\be t} + \al\zeta_\be(t) \mcD_p(\mu,\nu)^p.\qedhere\]
\end{proof}

%=============================================================================
\subsubsection{Well-posedness of the nonlinear problem (proof of Theorem \ref{thm4.3:chap4-main-result})}
In the preceding discussion, we have found a unique $\mcW_p$-stable solution to the linearized problem \eqref{eq4.7:linearized-mf-eq} for any given $\mu \in \CTPC$ and any initial data. This gives a well-defined map from the given mean-field curve $\mu$ to the unique $\mcW_p$-stable solution $\rho \in \CTPC$ of \eqref{eq4.7:linearized-mf-eq}. We denote this map by
\begin{align}\label{eq4.contractphi}
	\Phi: \CTPC \to \CTPC.
\end{align}
The well-posedness of the mean-field problem \eqref{eq4.6:mean-field-equation} reduces to showing that $\Phi$ has a unique fixed point. That is, there exists a unique $\br \in \CTPC$ such that $\Phi(\br)=\br$. 
To this end, we shall make use of the classical Banach fixed-point theorem.

\begin{proof}[Proof of Theorem \ref{thm4.3:chap4-main-result}]
	Fix any $\br_0\in\PPP$.
	Let $h>0$ be small enough that $\al\zeta_\be(h)<1$. 
	Define the map $\Phi: C([0,h];\PPP) \to C([0,h];\PPP)$ by letting $\Phi(\mu)=\mbs(\br_0;\mcA(\mu))$ be the unique $\mcW_p$-stable solution $\rho\in C([0,h];\PPP)$ of the initial value problem \eqref{eq4.7:linearized-mf-eq}.
	By Theorem \ref{thm4.345:chap4-main-result-linearized}, if $\mu,\nu\in C([0,h];\PPP)$, then
	\begin{align*}
		\mcD_p(\rho,\si) \le (\al\zeta_\be(h))^{1/p} \, \mcD_p(\mu,\nu).
	\end{align*}
	where $\rho=\Phi(\mu)$ and $\si=\Phi(\nu)$. 
	
	Since $\al\zeta_\be(h)<1$, the mapping $\Phi: C([0,h];\PPP) \to C([0,h];\PPP)$ is a contraction under the metric $\mcD_p$. By the Banach fixed-point theorem, $\Phi$ admits a unique fixed point $\br^{(1)}\in C([0,h];\PPP)$. That is, the problem
	\begin{align*}
		\left\{\begin{aligned}
			\partial_t \br^{(1)}_t &= \br^{(1)}_t\mcA(\br^{(1)}_t),\quad t\in(0,h),\\
			\br^{(1)}_0 &= \br_0,
		\end{aligned}\right.
	\end{align*}
	has a unique local solution $\br^{(1)}\in C([0,h];\PPP)$ for small $h>0$ given above.
	
	We establish the local existence and uniqueness of solutions within a time window $h > 0$ that is independent of the initial data. Global existence and uniqueness are then established via the standard ``gluing'' method.
\end{proof}

%% file: Section5.tex
\section{Propagation of chaos in the abstract mean-field model}\label{chap5}
In this section, we study propagation of chaos for abstract mean-field \(N\)-particle systems. \emph{Propagation of chaos} refers to the phenomenon whereby, in a weakly interacting particle system, the statistical dependence between any fixed collection of particles vanishes as \(N\to\infty\). Intuitively, although the system is coupled, any finitely many randomly selected particles become asymptotically independent. The limiting process is therefore described by independent and identically distributed copies of a single-particle process, called the \emph{mean-field limit}. In the present framework, this limit is obtained from the mean-field evolution equation introduced in Section \ref{chap4}.

The main goal of this section is twofold. First, for a given mean-field generator as in Definition \ref{def4.1:mean-field-operator}, we construct an associated \(N\)-particle system directly at the level of its generator. Second, under Hypothesis {\refA} and additional structural assumptions, we establish an exponential estimate in the Wasserstein-\(p\) distance between the law of the \(N\)-particle system and the \(N\)-tensorized mean-field law, leading to quantitative propagation of chaos as \(N\to\infty\).

The section is organized as follows. We first introduce the notions of chaos and propagation of chaos. In Section \ref{sec5.2}, we formulate the abstract mean-field particle system and state the main assumptions and main result. Section \ref{sec5.3} collects preparatory estimates relating optimal transport, $\mcW_p$-stable solutions, and tensor products. In Section \ref{sec5.4}, these estimates are combined with results from Sections \ref{chap3} and \ref{chap4} to prove the main result. The final subsection, Section \ref{sec:hilbert-refinement}, establishes the variance refinement under averaged Hilbert-space dependence.

This section will be notationally intensive, so we introduce here some notation used throughout. Let \((\Pi,d)\) be a locally compact Polish space, and let \(N\ge1\) be the number of particles. We shall frequently work with the \(N\)-fold product space
\[
\Pi^N = \underbrace{\Pi\times\cdots\times\Pi}_{N\text{ times}}.
\]
Elements of \(\Pi^N\) will be denoted by boldface letters, for instance
\[
\Bx = \Bx_N = (x_1,x_2,\dots,x_N)\in\Pi^N,
\qquad
x_k\in\Pi,\quad 1\le k\le N.
\]
Similarly, we denote by \(C_b(\Pi^N)\), \(\mcP(\Pi^N)\), and \(\mcG(\Pi^N)\) the space of bounded continuous functions, the space of probability measures, and the space of Feller generators on \(\Pi^N\), respectively.
Objects in these spaces of \(N\)-particle quantities will also be denoted by boldface symbols; for example,
\[
\bs\Phi=\bs\Phi_N\in C_0(\Pi^N),
\qquad
\bs\mu=\bs\mu_N\in\mcP(\Pi^N),
\qquad
\bs\mcA=\bs\mcA_N\in\mcG(\Pi^N).
\]
In short, we adopt the following notational convention throughout: \emph{regular type is reserved for single-particle objects, while boldface type is reserved for \(N\)-particle objects.} 
This convention will be extended later to other notions such as the Wasserstein metric and the transport derivative.
When there is no risk of confusion, we suppress the subscript \(N\).

%=============================================================================
\subsection{Notion of chaos and propagation of chaos}
Before proceeding, let us briefly explain the terminology chaos and propagation of chaos from a probabilistic perspective. 
Suppose $\scb{\BX_t}_{t\ge0}$ is an $N$-particle system on the state space $\Pi$, where
$$\BX_t=(X_t^1,X_t^2,\cdots,X_t^N)\in\Pi^N.$$
The term \emph{chaos} refers to the concept of \emph{independence and identical distribution} in probability theory. For a fixed time $t\ge0$, $\BX_t$ is considered \emph{chaotic} if its coordinates $\scb{X_t^i}_{1\le i\le N}$ are ``close'' to being i.i.d. The term \emph{propagation of chaos} refers to a key property of a particle system: if $\BX_0$ is chaotic, then $\BX_t$ remains chaotic for any $t\ge0$. In other words, the chaotic nature of the system at the initial time $t=0$ is propagated forward in time through the dynamics of the system. Propagation of chaos often emerges in mean-field $N$-particle systems because, in mean-field systems, the particles are only weakly coupled, resulting in relatively weak dependencies between them.

The first mathematical notion of chaos was introduced by Mark Kac in his seminal article \cite{kac1956foundations}. In what follows, $\Pi$ is a Polish space, and by a symmetric probability measure $\Bmu$ on $\Pi^N$, we mean that $\Bmu$ satisfies the following condition: for any permutation $\sigma$ of the indices $\{1, 2, \ldots, N\}$ and any measurable set $\bs E \subseteq \Pi^N$,
\[\Bmu(\bs E)=\Bmu(\si(\bs E)),\]
where $\sigma(\bs E) := \cb{(x_{\sigma(1)}, x_{\sigma(2)}, \cdots, x_{\sigma(N)}) : (x_1, x_2, \ldots, x_N) \in \bs E}$.
We refer to \cite[Chapter 3]{Chaintron_2022_1} for the coming definitions.
\begin{definition}[Kac's chaos]
    Let $\br\in\mcP(\Pi)$ and $\cb{\Brho^N}_{N\in\mbn}$ be a sequence where each $\Brho^N$ is a symmetric probability measure on $\Pi^N$. The sequence $\cb{\Brho^N}_{N\in\mbn}$ is said to be \emph{(Kac's) $\br$-chaotic} if for any $k\in\mbn$ and any function $\bs\Phi_k\in C_b(\Pi^k)$,
    \begin{align*}
        \lim_{N\to\infty} \inn{\Brho^N,\bs\Phi_k\otimes 1^{\otimes(N-k)}} = \inn{\br^{\otimes k},\bs\Phi_k}.
    \end{align*}
\end{definition}
This means that for all $k\in\mbn$, $\Brho^{k,N}$, the marginal of $\Brho^N$ on its first $k$ coordinates, converges to $\br^{\otimes k}$ weakly in $\mcP(\Pi^k)$. If we view $\Brho^N$ as the law of some $N$-particle system, this property suggests that any group of $k$ particles becomes statistically independent and identically distributed with the common law $\br$ as $N$ grows to $\infty$, hence the terminology of chaos. Equivalently, Kac's chaos can be defined by using only tensorized test functions of the form $\bs\Phi_k = \phi_1\otimes\cdots\otimes\phi_k$, where $\phi_i\in C_b(\Pi)$. Furthermore, there is the following characterization of Kac's chaos from a probabilistic point of view by Sznitman \cite[Proposition 2.2]{sznitman1991topics}.
\begin{proposition}\label{prop:equiv-of-Kac-chaos}
    Let $\br\in\mcP(\Pi)$ and $\cb{\Brho^N}_{N\in\mbn}$ be a sequence of symmetric probability measures on $\Pi^N$. The following are equivalent:
    \begin{enumerate}[(a)]
        \item $\cb{\Brho^N}_{N\in\mbn}$ is (Kac's) $\br$-chaotic.
        \item There exists $k\ge2$ such that $\Brho^{k,N}$ converges weakly to $\br^{\otimes k}$ as $N\to\infty$.
        \item For each $N$, let $\bs X^N=(X^1,\cdots,X^N)$ be a system of $\Pi$-valued random variables with $\bs X^N \sim \Brho^N$. The random empirical measure $\mu_{\bs X^N}$ converges in law to the deterministic measure $\br$ as $N\to\infty$.
    \end{enumerate}
\end{proposition}
It is important to note that Kac's notion of chaos is purely \emph{qualitative}. On the other hand, \emph{quantitative} approaches have also been developed and explored by various authors, often leading to stronger and more precise results, in terms of rates of convergence. The initial step in this approach is to introduce a metric or divergence that quantifies the degree of chaos, and one of them is the Wasserstein metric. 

Notions of quantitative chaos based on the Wasserstein metric have been introduced by authors such as \cite{Chaintron_2022_1, hauray2014kac}. We now introduce a corresponding notion in the present setting. Recall that the state space \((\Pi,d)\) is equipped with a metric \(d\). For \(p\in[1,\infty)\) and \(N\in\mathbb{N}\), we define the normalized \(p\)-th product metric \(\bs d_p=\bs d_{p,N}\) on the \(N\)-fold product space \(\Pi^N\) by
\begin{align}\label{def:N-metric}
	\bs d_p(\Bx,\By)
	:=
	\left(
	\frac{1}{N}\sum_{k=1}^N d(x_k,y_k)^p
	\right)^{1/p},
	\qquad
	\Bx=(x_1,\dots,x_N),\;
	\By=(y_1,\dots,y_N)\in\Pi^N.
\end{align}
Recall the $p$-th cost $c_p = \f 1p d^p$. 
For notational convenience, we also introduce the shorthand \(\Bc_p=\Bc_{p,N}:=\frac1p \bs d_{p,N}^{\,p}\) for the normalized \(p\)-th cost on \(\Pi^N\), that is, 
\[
\Bc_p(\Bx,\By):=\frac{1}{pN}\sum_{k=1}^N d(x_k,y_k)^p = \f 1N\sum_{k=1}^N c_p(x_k,y_k).
\]

Clearly, the \(N\)-fold product space \((\Pi^N,\bs d_p)\) is again a locally compact Polish space. We may therefore define the associated Wasserstein-$p$ space of probability measures. In particular, \(\mcp_p(\Pi^N)\) consists of all probability measures \(\bs\mu \in \mcp(\Pi^N)\) such that
\[
\int_{\Pi^N}\bs c_p(\bs z,\By)\,d\bs\mu(\By) <\infty,\qquad\text{for some, equivalently for all, }\bs z\in\Pi^N.
\]
For \(\Bmu,\Bnu\in\mcp_p(\Pi^N)\), the Wasserstein-\(p\) distance associated with the normalized product metric \(\bs d_p\) is denoted by \(\bs{\mcW}_p=\bs{\mcW}_{p,N}\) and is given by
\[
\bs{\mcW}_p^p(\Bmu,\Bnu)
=
\inf_{\bs\ga\in\Ga(\Bmu,\Bnu)}
\int_{\Pi^{2N}}
\Bc_p(\Bx,\By)
\,d\bs\ga(\Bx,\By).
\]
For clarity, we adopt the same convention: the boldface symbol \(\bs{\mcW}_p\) denotes the Wasserstein-\(p\) distance on \(\mcp_p(\Pi^N)\), while \(\mcW_p\) denotes the Wasserstein-\(p\) distance on \(\mcp_p(\Pi)\). We often suppress the subscript \(N\) in \(\bs d_{p,N}\), \(\bs c_{p,N}\), \(\bs{\mcW}_{p,N}\), and related quantities whenever \(N\) is understood to be fixed.

We now introduce the notion of chaos in the Wasserstein-\(p\) metric, which will be used in the main theorem. This corresponds to the second notion of chaos in \cite[Definition 3.20]{Chaintron_2022_1}. Several related notions are discussed there; the interested reader may consult \cite[Section 3]{Chaintron_2022_1} for details.

\begin{definition}[Chaos in the Wasserstein-$p$ metric]\label{def:chaos-in-wasserstein-p-metric}
Let $\br\in\mcP_p(\Pi)$ and let $\cb{\Brho^N}_N$ be a sequence of symmetric probability measures on $\Pi^N$. We say that $\cb{\Brho^N}_N$ is \emph{infinite dimensional Wasserstein-$p$ $\br$-chaotic} if
\begin{align*}
    \bs{\mcW}_{p,N}\bigl(\Brho^N,\br^{\otimes N}\bigr)\longrightarrow 0,
    \qquad\mbox{as }N\to\infty.
\end{align*}
\end{definition}

On the other hand, the notion of propagation of chaos is a dynamical version of chaos. Fix a final time $T\ge 0$ and let $\cb{\bs X_t^N}_{t\ge0}$ be a process on $\Pi^N$, with initial distribution that is chaotic. The property of propagation of chaos is said to hold when the initial chaos is propagated at later times. Note that this property can hold either at the \emph{pointwise} level or at the \emph{pathwise} level, but we focus only on pointwise propagation of chaos here. The pathwise case will be discussed briefly in a remark later.
\begin{definition}[Pointwise propagation of chaos]\label{def:Kac-poc}
    For each $N\in\mbn$, let $\cb{\bs X_t^N}_{t\in[0,T]}$ be a permutation-invariant Markov process on $\Pi^N$ with $\cb{\Brho^N_t}_{t\ge0}$ being its distribution, and let $\cb{\br_t}_{t\ge0}\in C([0,T];\mcP(\Pi))$. Fix any $T>0$.
    \begin{enumerate}[(i)]
        \item \emph{(Pointwise) propagation of Kac's chaos} holds if the following is true: if the initial distribution $\Brho_0^N$ is (Kac's) $\br_0$-chaotic, then for each $t\in[0,T]$, the distribution $\Brho_t^N\in\mcP(\Pi^N)$ is (Kac's) $\br_t$-chaotic.
        \item \emph{(Pointwise) propagation of infinite dimensional Wasserstein-$p$ chaos} holds if the following is true: if the initial distribution satisfies $\bs\mcW_{p,N}(\Brho_0^N,\br_0^{\otimes N})\to 0$ as $N\to\infty$, then for each $t\in[0,T]$,
        \begin{align}\label{eq:convergence-in-Wp}
            \bs{\mcW}_{p,N}(\Brho_t^N,\br_t^{\otimes N})\to 0,\qquad\mbox{as }N\to\infty.
        \end{align}
    \end{enumerate}
\end{definition}
If propagation of chaos holds, the measure $\br_t$ shall be called the \emph{mean-field limit}. 
In this work, we focus on propagation of infinite dimensional Wasserstein-$p$ chaos. A typical estimate in this case is given by
\begin{align*}
    \sup_{t\in[0,T]}\bs\mcW_{p,N}^p(\Brho_t^N,\br_t^{\otimes N}) \le C_T\bs\mcW_{p,N}^p(\Brho_0^N,\br_0^{\otimes N})+\ep(N,T),
\end{align*}
where $C_T\ge1$ is independent of $N$ and $\ep(N,T)\to0$ as $N\to\infty$, leading to \eqref{eq:convergence-in-Wp} whenever the initial discrepancy tends to zero. Note that by Proposition \ref{prop:equiv-of-c-convergence}, convergence in the Wasserstein metric implies weak convergence as well, which leads to Kac's chaos by Proposition \ref{prop:equiv-of-Kac-chaos}(b).

%=============================================================================
\subsection{Abstract \texorpdfstring{$N$}{N}-particle systems and propagation of chaos}\label{sec5.2}

Let us introduce some further notation before proceeding to the main result. Fix \(p\in[1,\infty)\) and an integer \(N\ge2\). Recall the \(N\)-fold product metric space \((\Pi^N,\bs d_p)\) from \eqref{def:N-metric} and the corresponding Wasserstein-\(p\) space \((\mcp_p(\Pi^N),\bs{\mcW}_p)\). We also denote by \(C_{p}(\Pi^N)\) the space of continuous functions with \(p\)-th moment bounds and by \(\mcG_p^0(\Pi^N)\subset\mcG(\Pi^N)\) the class of Wasserstein-\(p\) continuous Feller generators, as introduced in Section \ref{sec:was-p} and Definition \ref{def:gc0-generators}.

Given \(\bs\mcA,\bs\mcB\in\mcG_p^0(\Pi^N)\), their \(p\)-transport derivative is defined as in Definition \ref{def:om-definition}, with the transport cost \(\bs\mcW_p^p\) in place of \(\mcc_c\). To distinguish it from the one-particle notation, we again make use of the boldface symbol
\begin{align}\label{def:N-om}
	\bs\om_p(\bs\mu,\bs\nu;\bs\mcA,\bs\mcB):=
	\limsup_{t\searrow0}
	\frac{
		\bs{\mcW}_p^p(\bs\mu e^{t\bs\mcA},\bs\nu e^{t\bs\mcB})
		-
		\bs{\mcW}_p^p(\bs\mu,\bs\nu)
	}{t}.	
\end{align}

Let $\bigsqcup_{m\ge1}\Pi^m$ denote the disjoint union of Cartesian powers of $\Pi$, that is, the set of all points of the form $(x_1,x_2,\cdots,x_m)\in\Pi^m$ for any $m\in\mbn$. Let us define $\mu$ to be the \emph{empirical measure} map:
\begin{align*}
    \mu: \bigsqcup_{m\ge1}\Pi^m \to \mcP(\Pi),\qquad \mu(x_1,x_2,\cdots,x_m) = \frac{1}{m}\sum_{k=1}^m \de_{x_k}.
\end{align*}
In particular, if $\Bx=\Bx_N \in\Pi^N$, we have $\mu(\Bx)=\frac{1}{N}\sum_{k=1}^N \delta_{x_k}$, which represents the empirical measure associated to the state $\Bx$ of an $N$-particle system. For $\Bx=\Bx_N\in\Pi^N$ and $1\le k\le N$, we denote the $k$-th truncated variable
\begin{align*}
    \Bx_k' = \Bx_{N,k}' = (x_1,x_2,\cdots,x_{k-1},\cancel{x_k},x_{k+1},\cdots,x_N)\in\Pi^{N-1}.
\end{align*}
In this case, we have
\begin{align*}
    \mu(\Bx_k') = \frac{1}{N-1} \sum_{j=1,j\neq k}^N \de_{x_j},
\end{align*}
which represents the empirical measure of the $N-1$ particles, excluding the $k$-th particle.
Furthermore, given $\Bx=\Bx_N\in\Pi^N$ and $\bs{\Phi}\in C(\Pi^N)$, we denote by $\bs{\Phi}(\cdot;\Bx_k')\in C(\Pi)$ the \emph{$x_k$-slicing function}
\begin{align*}
    y\mapsto \bs{\Phi}(x_1,x_2,\cdots,x_{k-1},y,x_{k+1},\cdots,x_N).
\end{align*}

%=============================================================================
\subsubsection{From mean-field generators to \texorpdfstring{$N$}{N}-particle systems}
Throughout the rest of this section, $\MFG$ is a \emph{mean-field generator} on $\PPP$ (see Definition \ref{def4.1:mean-field-operator}). The argument $\mu$ of $\mcA_\mu=\mcA(\mu)$, which is a probability measure, will be called a \emph{mean-field measure}. Recall also the corresponding transition kernel of $\mcA(\mu)$, which is denoted by $\cb{\ka_t(\mu)}_{t\ge0}$. We remind the reader of the intuition here: $\mcA(\mu)$ is the infinitesimal description of a single particle, given that the mean-field particles are distributed as $\mu$. The solution of the associated (initial value problem of) mean-field equation \eqref{eq4.6:mean-field-equation}:
\begin{align*}
    \left\{\begin{aligned}
        \partial_t \br_t &= \br_t \mcA(\br_t),\quad t\in(0,T),\\
        \br_0 &\in \mcP_p(\Pi),
    \end{aligned}\right.
\end{align*}
gives a candidate for the mean-field limit of some $N$-particle systems as $N\to\infty$. Our main task here is to introduce an $N$-particle system, for each $N\ge2$, associated to the mean-field generator $\mcA$ such that the mean-field limit holds.

There are various classes of $N$-particle systems (Feller processes on $\Pi^N$) associated with a mean-field generator $\cb{\mcA(\mu)}_\mu$. In this work, we focus on a prototypical model of $N$-particle systems, which emerges in various applications, such as McKean-Vlasov diffusion models, mean-field jump processes and piecewise deterministic Markov processes \cite{Chaintron_2022_2}. We will now describe such particle systems by introducing the corresponding Feller generators $\bs{\hat\mcA}\in\mcG(\Pi^N)$ on $\Pi^N$. We point out that such an approach is not new, as it has been discussed, for instance, in \cite[Section 2.2]{Chaintron_2022_1}.
The generator $\bs{\hat\mcA}=\bs{\hat\mcA}_N$ of the $N$-particle system is the superposition (sum) of Feller generators:
\begin{align*}
    \bs{\hat\mcA} = \sum_{k=1}^N \bs{\hat\mcA}^{(k)},\quad \bs{\hat\mcA}^{(k)}\in\mcG(\Pi^N),
\end{align*}
where the generators $\bs{\hat\mcA}^{(k)}$ will be described below.
% is obtained by performing coordinate permutation on $\bs{\hat\mcA}^{(1)}$.

Let us first give a description of $\bs{\hat\mcA}^{(1)}$, from the perspective of stochastic processes. The associated process $\cb{\bs Y_t}_{t\ge0}$, starting at $(x_1,x_2,\cdots,x_N)$, takes the form
\begin{align*}
    \bs Y_t = (Y_t^1,Y_t^2,\cdots,Y_t^N),\quad (Y_t^2,\cdots,Y_t^N) = (x_2,\cdots,x_N),
\end{align*}
where the coordinates of $\bs Y_t$ with $k\ge2$ remain constant, while the first coordinate $\cb{Y_t^1}_{t\ge0}$ is the process associated with the generator
\begin{align*}
    \mcA(\mu(\Bx_1')), \quad\mu(\Bx_1')=\frac{1}{N-1}\sum_{k\ge2}\de_{x_k}.
\end{align*}
Specifically, the process $\cb{\bs Y_t}_{t\ge0}$ remains in the $x_1$-slicing, i.e. $\bs Y_t\in\Pi\times\cb{(x_2,\cdots,x_N)}$. The first coordinate process $Y_t^1$ is still dependent on the initial state $(x_2,\cdots,x_N)$, in particular through the empirical measure $\mu(\Bx_1')$ of $\cb{x_2,\cdots,x_N}$ via the mean-field generator $\mcA$.

At the level of generators, $\bs{\hat\mcA}^{(1)}$ is identified as
\begin{align}\label{eq:bold-hat-mcA-1}
    \bs{\hat\mcA}^{(1)} \bs{\Phi}(\Bx) = \bs{\hat\mcA}^{(1)}\bs{\Phi}(x_1;\Bx_1') := \mcA\sqb{\bs{\Phi}(\cdot;\Bx_1');\mu(\Bx_1')}(x_1),
\end{align}
where the domain $D(\bs{\hat\mcA}^{(1)})$ is identified as the set of all functions $\bs\Phi\in C_0(\Pi^N)$ such that each $x_1$-slice $\bs\Phi(\cdot;\Bx_1')$ belongs to $D(\mcA(\mu(\Bx_1')))$ for every $\Bx_1'\in\Pi^{N-1}$, and $\bs{\hat\mcA}^{(1)}\bs\Phi\in C_0(\Pi^N)$. To further explain the notation above, the value of $\bs{\hat\mcA}^{(1)}\bs{\Phi}(x_1,x_2,\cdots,x_N)$ is given by
\begin{align*}
    \mcA_\mu (\phi) (x_1),
    \qquad\text{where}\quad
    \phi = \bs{\Phi}(\cdot;\Bx_1')\in D(\mcA(\mu)),\quad \mu=\mu(\Bx_1') = \frac{1}{N-1} \sum_{k=2}^N \de_{x_k}.
\end{align*}
Alternatively, we may also define $\bs{\hat\mcA}^{(1)}$ using tensor products: $\bs{\hat\mcA}^{(1)}$ is the unique generator in $\mcG(\Pi^N)$ such that it holds for all (appropriate) $\phi\in C_0(\Pi),\bs\Psi'\in C_0(\Pi^{N-1})$ that
\begin{align*}
    \bs{\hat\mcA}^{(1)}[\phi\otimes\bs\Psi'](\Bx) = \mcA(\phi;\mu(\Bx_1'))(x_1)\cdot\bs\Psi'(\Bx_1').
\end{align*}
Lastly, we may also give a description through its transition kernel $\scb{\ka_t^{(1)}}_{t\ge0}$ and Feller semigroup $\scb{\bs T_t^{(1)}}_{t\ge0}$. The transition kernel is given by
\begin{align}\label{eq5.2:temp}
    \bs\ka_t^{(1)}(\Bx) = \ka_t(x_1;\mu(\Bx_1'))\otimes\de_{x_2}\otimes\cdots\otimes\de_{x_N},
\end{align}
where $\cb{\ka_t(\cdot;\mu)}_{t\ge0}$ is the transition kernel of $\mcA(\mu)$ for a mean-field measure $\mu\in\mcP_p(\Pi)$. It also follows that
\begin{align*}
    e^{t\bs{\hat\mcA}^{(1)}} \bs\Phi(\Bx) = \inn{\bs\ka_t^{(1)}(\Bx),\bs\Phi} = \int_\Pi \bs\Phi(y;\Bx_1')\,\ka_t(x_1,dy;\mu(\Bx_1')) = e^{t\mcA(\mu(\Bx_1'))} \bs\Phi(\cdot;\Bx_1')(x_1).
\end{align*}
For $k\ge2$, the description of $\bs{\hat\mcA}^{(k)}$ is analogous to the case of $\bs{\hat\mcA}^{(1)}$, except that the first particle is replaced by the $k$-th particle. The simplest way to express this notationally is as follows. Let $\si_k$, for $k\ge2$, be the action of interchanging the coordinates $x_1$ and $x_k$ in the variable $(x_1,x_2,\cdots,x_N)$. For example,
\begin{align*}
    \si_2\bs\Phi(x_1,x_2,x_3,\cdots,x_N) &= \bs\Phi(x_2,x_1,x_3,\cdots,x_N),\\
    \si_3\bs\Phi(x_1,x_2,x_3,\cdots,x_N) &= \bs\Phi(x_3,x_2,x_1,\cdots,x_N).
\end{align*}
Then we have
\begin{align}\label{eq:bold-hat-mcA-k}
    \bs{\hat\mcA}^{(k)}\bs{\Phi} = (\si_k^{-1} \bs{\hat\mcA}^{(1)} \si_k)\bs{\Phi}.
\end{align}

Finally, the main object of study is the superposition of the generators $\bs{\hat\mcA}^{(k)}$. Given its importance, we define it explicitly.
\begin{definition}\label{def:N-particle-generator}
    Let $\MFG$ be a mean-field generator on $\PPP$. For $N\ge2$ and $1\le k\le N$, let $\bs{\hat\mcA}^{(k)}\in\mcG(\Pi^N)$ be given by \eqref{eq:bold-hat-mcA-1} and \eqref{eq:bold-hat-mcA-k}. The \emph{$N$-particle generator} associated to the mean-field generator $\mcA$ is the sum of these generators:
    \begin{align*}
        \bs{\hat\mcA} = \sum_{k=1}^N \bs{\hat\mcA}^{(k)}.
    \end{align*}
\end{definition}
The question of whether the operators $\BHA^{(k)}$ and $\BHA$ generate Feller semigroups will be addressed later.

In the theory of Markov processes, operators of the form described above are often employed to construct a Markov process as a superposition of different individual processes. In our context, the process $\cb{\bs X_t}_{t\ge0}$ associated with $\bs{\hat\mcA}$ is a superposition of the processes governed by the operators $\bs{\hat\mcA}^{(k)}$. This means that the state of \emph{each} particle in the system evolves according to the empirical measure $\mu$ of the other particles, with the dynamics determined by $\mcA(\mu)$. Consequently, this framework naturally aligns with the concept of mean-field particle systems.

\begin{remark}
    Note that different authors may define the $N$-particle generator in various ways. For instance, in \cite{Chaintron_2022_1}, the associated $N$-particle generator is defined similarly to ours, but with the empirical measure depending on all particles in the system, i.e., using $\mu=\mu(\Bx)= \frac{1}{N}\sum_{k=1}^N \de_{x_k}$ in place of $\mu=\mu(\Bx_1')=\frac{1}{N-1}\sum_{k=2}^N \de_{x_k}$ in \eqref{eq:bold-hat-mcA-1}. While this distinction is minor and is mostly a matter of preference, we use our definition here to better capture the idea of ``depending on the distribution of other particles in the system''.
\end{remark}

%=============================================================================
\subsubsection{Assumptions for quantitative propagation of chaos}
The primary result of this section is an exponential estimate in the Wasserstein-\(p\) metric \(\bs{\mcW}_p\) between the Markov flow generated by the \(N\)-particle generator \(\BHA\) and the tensor product \(\br_t^{\otimes N}\), where \(\{\br_t\}_{t\ge 0}\) solves the mean-field equation \eqref{eq4.6:mean-field-equation}. The main assumption is Hypothesis {\refA} on the mean-field generator \(\MFG\), which is a Lipschitz condition on its \(p\)-transport derivative.

A subtle point concerns the generation problem for \(\BHA\): one must ensure that each \(\BHA^{(k)}\) and their sum \(\BHA\) generate Feller semigroups. Whether a sum of Feller generators is again a Feller generator is a question in perturbation theory for \(C_0\)-semigroups; under suitable conditions, the answer is affirmative. Since we do not wish to enter into these technical conditions, we collect the required generation properties in the following hypothesis.

\begin{hypothesisAp}[Generation of the $N$-particle generator]\label{hypoA'}

	Let $\mcA:\mcp_p(\Pi)\to\mcg_p^0(\Pi)$ be a mean-field generator. Assume for each $1\le k\le N$ that $\BHA^{(k)}$ and $\BHA=\BHA_N=\sum_{k=1}^N\BHA^{(k)}$ (see Definition \ref{def:N-particle-generator}) generate Feller semigroups on $\Pi^N$. Moreover, $\BHA^{(k)}, \BHA\in\mcG_p^0(\Pi^N)$, and there is a core $\bs\mcD\subset D(\BHA)\subset C_0(\Pi^N)$ such that $\bs\mcD\subset D(\BHA^{(k)})$ for all $k$.	
\end{hypothesisAp}

%\begin{hypothesis}[Generation assumption]\label{hypo:generation-problem}
%    
%\end{hypothesis}

%=============================================================================
\subsubsection{Exponential estimate for quantitative propagation of chaos}

Before stating the main result of this section, we introduce one further quantity that will appear in the main estimate.

\begin{definition}\label{def:aleph-N}
    Let $p \in [1,\infty)$ and $N\ge2$. Given $\br\in\mcP_p(\Pi)$, we define
    \begin{align*}
        \aleph_N(\br) :=\int_{\Pi^N} \mcW_p^p(\mu(\By_1'),\br)\,d\br^{\otimes N}(\By), \qquad\mu(\By_1')=\frac{1}{N-1}\sum_{k=2}^N \de_{y_k}.
    \end{align*}
\end{definition}

\begin{remark}
    To provide some insight into this quantity in the context of probability theory, let $\cb{X_k}_{k=1}^\infty$ be i.i.d. $\Pi$-valued random variables with the common law $\br\in\mcP_p(\Pi)$. Then the integral above represents the expected $p$-th power of the Wasserstein-$p$ distance between $\br$ and its i.i.d. empirical measure $\frac{1}{N-1}\sum_{k=2}^N\de_{X_k}$:
    \begin{align*}
        \aleph_N(\br) = \mbe\msqb{\mcW_p^p\mrb{\frac{1}{N-1}\sum_{k=2}^N \de_{X_k},\br}}.
    \end{align*}
\end{remark}

Intuitively, by the law of large numbers, the empirical measure \(\frac{1}{N-1}\sum_{k=2}^N\delta_{X_k}\) converges weakly to \(\br\) almost surely. Thus, under suitable assumptions on \(\br\) and on the metric space \((\Pi,d)\), one expects that \(\aleph_N(\br)\to0\) as \(N\to\infty\). Quantitative estimates for this error are discussed in Remark \ref{rem:aleph} and, in the Euclidean setting, in Theorem \ref{thm:fournier-convergence-rate}.

We now state the main result of this section, which is an exponential estimate in the Wasserstein-\(p\) metric between the \(N\)-particle flow \(\bs\rho_t=\bs\rho_0 e^{t\BHA}\), for some \(\bs\rho_0\in\mcp_p(\Pi^N)\), and the tensorized mean-field law \(\bs\br_t=\br_t^{\otimes N}\), where \(\br_t\) solves the mean-field equation \eqref{eq4.6:mean-field-equation}. 

\begin{theorem}[Exponential estimate in Wasserstein-$p$ metric]\label{thm:chap5-main-thm}
    Fix $p\ge 1$ and let $\mcA:\mcp_p(\Pi)\to \mcg_p^0(\Pi)$ be a mean-field generator satisfying Hypothesis {\refA} with constants $\al,\be\ge 0$, and Hypothesis \refAp.
    Fix $N\ge2$ and let $\bs{\hat\mcA}$ be the $N$-particle generator associated to $\mcA$ given in Definition \ref{def:N-particle-generator}, $\{\br_t\}_{t\ge 0}\in C([0,\infty);\mcP_p(\Pi))$ be the $\mcW_p$-stable solution of the mean-field equation \eqref{eq4.6:mean-field-equation} associated to $\mcA$, and $\bs\rho_0\in\mcP_p(\Pi^N)$. We denote
    \begin{align*}
        \bs{\rho}_t = \bs{\rho}_0 e^{t\bs{\hat{\mcA}}}, \quad \bs{\br}_t = \br_t^{\otimes N}.
    \end{align*}
    Then for any $T\ge0$, it holds
    \begin{align*}
        \sup_{t\in[0,T]}\bs{\mcW}_p^p\mrb{\bs{\rho}_t,\bs{\br}_t} \le \bs\mcW_p^p\mrb{\bs{\rho}_0,\bs{\br}_0} e^{K T} + C \zeta_{K}(T) \sup_{t\in[0,T]} \aleph_{N}(\br_t),
    \end{align*}
    where $\aleph_{N}$ is given in Definition \ref{def:aleph-N}, $C=\al 2^{p-1},K=\be+\al 2^{p-1}$, and $\zeta_K$ is defined in \eqref{eq3.25:zeta-be}.
\end{theorem}

%=============================================================================
For $p=2$, an additional structural condition involving averaged Hilbert-valued features allows us to replace the empirical transport error $\aleph_N$ by a variance bound of order $N^{-1}$. This condition captures the cancellation mechanism underlying classical estimates for McKean--Vlasov models with averaged coefficients; see \cite{sznitman1991topics}. We establish the abstract refinement in Section~\ref{sec:hilbert-refinement} and illustrate it for a class of L\'evy-type mean-field generators in Section~\ref{sec6:averaged-amplitudes}.

\subsubsection{Pointwise propagation of chaos from Theorem \ref{thm:chap5-main-thm}}
Let us now present the propagation of chaos result for abstract mean-field systems as a corollary of Theorem \ref{thm:chap5-main-thm}.
\begin{corollary}[Pointwise propagation of chaos]
    Assume the settings in Theorem \ref{thm:chap5-main-thm}. Assume that
    \begin{align*}
        \lim_{N\to\infty} \sup_{t\in[0,T]} \aleph_N(\br_t)=0.
    \end{align*}
    Then the sequence $\scb{\Brho_t^N}_{N}$ exhibits pointwise propagation of infinite dimensional Wasserstein-$p$ chaos as $N\to\infty$. That is, if
    $\bs\mcW_p(\Brho_0^N,\br_0^{\oN})\xrightarrow[]{N\to\infty} 0$,
    then for fixed $T\ge0$, it holds
    $\bs \mcW_p(\Brho_t^N,\br_t^{\oN})\xrightarrow[]{N\to\infty} 0$
    for each $t\in[0,T]$.
\end{corollary}

\begin{remark}\label{rem:aleph}
	The convergence assumption on $\aleph_N(\br_t)$ can be quantified in a number of concrete settings. In particular, when $\aleph_N(\br_t)$ is given by the Wasserstein distance between an empirical measure and its underlying law, quantitative rates of convergence have been established in terms of the geometric properties of the underlying metric space and the measure itself; see, for example, \cite{BoissardLeGouic2014,WeedBach2019}. In such results, the rate is determined not only by the ambient dimension, but more generally by the metric complexity of the support of the measure, as reflected through covering or entropy properties. 
	For instance, Weed and Bach \cite{WeedBach2019} obtain, under suitable assumptions, the bound
	\begin{align*}
		\aleph_N(\bar\rho)=O\bigl(N^{-p/s}\bigr)\quad \mbox{ for all }s>d,
	\end{align*}
	where the exponent $d$ is determined by the metric complexity of the probability measure $\bar\rho$, as quantified through suitable covering numbers. Thus, the convergence rate of the empirical measure is governed by the intrinsic metric structure of $\bar\rho$, rather than necessarily by the dimension of the ambient space. In particular, a measure concentrated on a lower-dimensional subset may exhibit a substantially faster rate than that suggested by the ambient dimension.
	Related quantitative estimates for empirical measures on $\mathbb R^d$ were obtained by Fournier and Guillin \cite{fournier2015rate}. We shall return to these estimates in the setting of $\mathbb R^d$ in Section~\ref{chap6} and therefore leave their detailed discussion until then.
\end{remark}

For a probabilistic interpretation of the result above, suppose that the mean-field evolution $\scb{\br_t}_{t\ge0}$ is realized as the time-marginal laws of a process $\cb{\bX_t}_{t\ge0}$, and let $\cb{\bX_t^k}_{t\ge0}$, $1\le k\le N$, be independent copies of this process. Let $\cb{\bs X_t}_{t\ge0}$ be the Feller process associated to the $N$-particle generator $\bs{\hat\mcA}$, and write $\bs\bX_t=(\bX_t^1,\cdots,\bX_t^N)$. Let $\mu_t^N,\bar\mu_t^N$ be their empirical measures:
\begin{align*}
    \mu_t^N = \frac{1}{N}\sum_{k=1}^N \de_{X_t^k},\qquad \bar\mu_t^N=\frac{1}{N}\sum_{k=1}^N \de_{\bX_t^k},
\end{align*}
and assume that $\BX_0=\bs\bX_0$.
For each fixed $t\in[0,T]$, there exists a coupling of the random vectors $\bs X_t$ and $\bs\bX_t$ such that
\begin{align*}
    \mbe[\mcW_p^p(\mu_t^N,\bar\mu_t^N)]\le C \zeta_{K}(T) \sup_{t\in[0,T]} \aleph_{N}(\br_t),\quad t\in[0,T].
\end{align*}
This follows from the following observation: for any $\Bx,\By\in \Pi^N$, 
\begin{align}\label{eq5.5:temp}
	\mcW_p^p(\mu(\Bx),\mu(\By)) &\le \ \Bc_{p}(\Bx,\By).
\end{align}
Indeed, the coupling
\(
	\gamma=\frac1N\sum_{k=1}^N \delta_{(x_k,y_k)}
\)
of the empirical measures \(\mu(\Bx)\) and \(\mu(\By)\) immediately gives the bound. 
Consequently, the Wasserstein stability estimate at the level of the \(N\)-particle laws directly yields a corresponding estimate for their empirical measures.

\begin{remark}\label{rem:pointvspath}
    The exponential estimate from the main theorem above ensures \emph{pointwise} propagation of chaos, provided that $\sup_{t\in[0,T]} \aleph_N(\br_t)\to0$ as $N\to\infty$, but it is insufficient to guarantee \emph{pathwise} propagation of chaos. Specifically, to state a pathwise result, one must first define a coupling (Markov) process $\scb{(\BX_t,\bs\bX_t)}_{t\ge0}$ whose marginal processes $\scb{\BX_t}_{t\ge0}$ and $\scb{\bs\bX_t}_{t\ge0}$ have laws $\scb{\Brho_t}_{t\ge0}$ and $\scb{\Bbr_t}_{t\ge0}$, respectively. Let $\mu_t^N$ and $\bar\mu_t^N$ be the empirical measures of $\BX_t$ and $\bs\bX_t$. Then, the quantitative pathwise propagation of chaos requires an exponential estimate of the form:
    \begin{align*}
        \mbe\msqb{\sup_{t\in[0,T]}\mcW_p^p(\mu_t^N,\bar\mu_t^N)} \le \rb{\mbe\msqb{\Bc_{p}(\BX_0,\bs\bX_0)}+\ep_N}e^{K T},
    \end{align*}
    where $\ep_N\to0$ as $N\to\infty$. Our result, which is pointwise in nature, is weaker than this. Specifically, it leads to the following
    \begin{align*}
        \sup_{t\in[0,T]}\mbe\msqb{\mcW_p^p(\mu_t^N,\bar\mu_t^N)} \le \rb{\mbe\msqb{\Bc_{p}(\BX_0,\bs\bX_0)}+\ep_N}e^{K T}.
    \end{align*}
    Note that in this case, the supremum is taken after the expectation. The above inequality holds for some coupling between the laws of $\BX_t$ and $\bs\bX_t$ for each fixed $t\in[0,T]$. However, this coupling does not necessarily have to be a process.
\end{remark}

%=============================================================================
\subsubsection{Strategy of proof}
Let us now explain the strategy for obtaining the exponential bound from Theorem \ref{thm:chap5-main-thm}. The main step of the proof is to establish an integral inequality for $\bs\mcW_p^p(\bs\rho_t,\bs\br_t)$ of the form
\begin{align*}
    \bs \mcW_p^p(\bs\rho_t,\bs\br_t) \le \bs\mcW_p^p(\bs\rho_0,\bs\br_0) + \int_0^t\mrb{\ep_N+K\cdot\bs\mcW_p^p(\bs\rho_s,\bs\br_s)}\,ds,
\end{align*}
where $\ep_N,K\ge0$. Then the exponential bound follows from Gr\"{o}nwall's inequality.

To establish the integral bound above, we will use the $p$-transport derivative $\bs\om_p=\bs\om_{p,N}$ (see the notation introduced in \eqref{def:N-om}) between the flows
\[\bs\rho_t=\bs\rho_0 e^{t\bs{\hat\mcA}} \quad\text{and}\quad \bs\br_t=\br_t^{\otimes N}.\]
To proceed, we shall first identify the evolution problem that the tensorized measure $\bs\br_t$ solves. It takes the form
\begin{align}\label{eq5:expected-tensor-evolution}
    \left\{\begin{aligned}
        \partial_t \bs\br_t &= \bs\br_t \bs\mcM_t,\quad t\in(0,T),\\
        \bs\br_0 &= \br_0^{\otimes N},
    \end{aligned}\right.
\end{align}
where $\bs\mcM_t=\bs\mcM(\br_t)\in\mcG_p^0(\Pi^N)$ is the tensorized generator associated with $\mcA(\br_t)$, as defined in the next subsection. In light of Lemma \ref{special-stability}, it holds
\begin{align*}
    \bs \mcW_p^p(\bs\rho_t,\bs\br_t) \le \bs \mcW_p^p(\bs\rho_0,\bs\br_0) + \int_0^t \ta(s,\bs\rho_s,\bs\br_s) \,ds,
\end{align*}
where $\ta(s,\bs\mu,\bs\nu)$ is an upper bound of $\bs \om_{p}(\Bmu,\Bnu;\bs{\hat\mcA},\bs\mcM_s)$ that is continuous. The proof is then completed by obtaining an appropriate bound for $\bs \om_{p}(\cdot,\cdot;\bs{\hat\mcA},\bs\mcM_s)$ of the form
\begin{align*}
    \bs \om_{p}(\bs\rho_s,\bs\br_s;\bs{\hat\mcA},\bs\mcM_s) \le \ep_N+K\cdot \bs\mcW_p^p(\bs\rho_s,\bs\br_s).
\end{align*}

%We define $\bs{\mcM}=\bs{\mcM}_N:\mcP_{c}(\Pi)\to\mcG_{\bs{c}}(\Pi^N)$ by
%\begin{align}\label{eq5.3:mean-field-superposition}
%    \bs{\mcM}_N(\mu) = \Si(\mcM^{(1)}(\mu)).
%\end{align}
%The strategy of the proof of this theorem is to recognize $\bs{\mu}_t$ and $\bs{\br}_t$ as the $c$-stable solutions of
%\begin{align*}
%	\partial_t \bs{\mu}_t = \bs{\mu}_t \bs{\hat{\mcA}}_N,\quad \partial_t \bs{\br}_t = \bs{\br}_t \bs{\mcM}_N(\br_t),
%\end{align*}
%and use Theorem \ref{newest-stability}, which yields an integral inequality of the form
%\begin{align*}
%	\mcC_{c_N}\rb{\bs{\mu}_t,\bs{\br}_t} \le \mcC_{c_N}\rb{\bs{\mu}_0,\bs{\br}_0} + \int_0^t \Om(\bs{\mu}_\tau,\bs{\br}_\tau;\bs{\hat{\mcA}}_N,\bs{\mcl{M}}_N(\br_\tau))\,d\tau.
%\end{align*}
%The result will then follows by bounding the quantity in the integral and applying Gr\"{o}nwall's inequality.
%\C'D{Why these two generators are in $\mcG_c(\Pi^N)$.}

%=============================================================================
\subsection{Preliminaries: tensorization, optimal transport, and \texorpdfstring{$\mcW_p$}{Wp}-stable solutions}\label{sec5.3}
As mentioned in the proof strategy of the main result, we shall first explore the relationship between optimal transport, $\mcW_p$-stable solutions, and the action of tensorization.

\subsubsection{Properties of tensorized measures}
Let us first record an elementary bound for the Wasserstein distance between tensorized measures. The proof is immediate by taking the tensor product of the optimal couplings, so we omit it. In fact, equality holds under the normalized product metric, but only the inequality below will be needed.

\begin{lemma}\label{lem:cost-of-tensorized-measures}
    Let $\cb{\mu_k}_{k=1}^N, \cb{\nu_k}_{k=1}^N\subset\mcP_p(\Pi)$, and let $\bs\mu=\mu_1\otimes\cdots\otimes\mu_N$, $\bs\nu=\nu_1\otimes\cdots\otimes\nu_N$. Then $\bs\mu,\bs\nu\in\mcp_p(\Pi^N)$ and it holds
    \begin{align*}
        \bs \mcW_p^p(\bs\mu,\bs\nu) \le \frac{1}{N}\sum_{k=1}^N \mcW_p^p(\mu_k,\nu_k).
    \end{align*}
\end{lemma}

The next result records several stability and continuity properties that are inherited under tensorization. They are immediate consequences of Lemma \ref{lem:cost-of-tensorized-measures}.

\begin{corollary}\label{cor5.21:temp}
	Let \(p\in[1,\infty), T>0\) and \(N\ge1\). 
	\begin{enumerate}[(i)]
		\item If \(\mu=\{\mu_t\}_{0\le t\le T}\in C([0,T];\mcp_p(\Pi))\), then \(\mu^{\otimes N}=\{\mu_t^{\otimes N}\}_{0\le t\le T}\in C([0,T];\mcp_p(\Pi^N))\).
		\item For each $n\ge 1$, let $\mu^{(n)},\mu\in C([0,T];\mcp_p(\Pi))$. If $\mu^{(n)}\to \mu $ in $C([0,T];\mcp_p(\Pi))$, then $(\mu^{(n)})^{\otimes N}\to \mu^{\otimes N}$ in $C([0,T];\mcp_p(\Pi^N))$. That is, if \(\sup_{t\in[0,T]}\mcW_p(\mu_t^{(n)},\mu_t)\to0\), then 
		\[
		\sup_{t\in[0,T]} \bs{\mcW}_p\bigl((\mu_t^{(n)})^{\otimes N},\mu_t^{\otimes N}\bigr)\to0.
		\]
	\end{enumerate}
\end{corollary}

\begin{proof}
    (i) Since $\mu\in C([0,T];\mcp_p(\Pi))$, we have $\mu_t\in \mcp_p(\Pi)$ for each $t\in[0,T]$.     
    By Lemma \ref{lem:cost-of-tensorized-measures}, $\mu_t^{\otimes N}\in \mcp_p(\Pi^N)$ for all such $t$. By the same lemma, it holds
    \begin{align*}
        \bs\mcW_p^p(\mu_{t+h}^{\otimes N},\mu_t^{\otimes N}) \le \frac{1}{N}\sum_{k=1}^N \mcW_p^p(\mu_{t+h},\mu_t) = \mcW_p^p(\mu_{t+h},\mu_t) \xrightarrow{h\to0} 0.
    \end{align*}
    Thus, $t\mapsto \mu_t^{\otimes N}$ is continuous in $\bs\mcW_p$. This shows $\mu^{\otimes N}\in C([0,T];\mcp_p(\Pi^N))$. 
    
    (ii) Similarly, using the same inequality, the claim follows from the following calculation:
    \[\sup_{t\in[0,T]}\bs\mcW_p^p\mrb{(\mu_t^{(n)})^{\otimes N},\mu_t^{\otimes N}} \le \frac{1}{N}\sum_{k=1}^N \sup_{t\in[0,T]}\mcW_p^p(\mu_t^{(n)},\mu_t)=\sup_{t\in[0,T]}\mcW_p^p(\mu_t^{(n)},\mu_t)\to 0. \qedhere \]  
\end{proof}

\subsubsection{The \texorpdfstring{$N$}{N}-tensorized mean-field generator}
Let us now take up the task outlined previously: constructing the generator $\bs\mcM(\br_t)\in\mcG_p^0(\Pi^N)$ in the evolution equation \eqref{eq5:expected-tensor-evolution} that governs the tensorized measure $\Bbr_t=\br_t^{\oN}$. The generator $\bs{\mcl{M}}$ is given in the following definition.

\begin{definition}[$N$-tensorized mean-field generator]\label{def:independent-generator-mcM}
	Let $\MFG$ be a mean-field generator on $\PPP$ and $N\ge 1$. Its \emph{$N$-tensorized mean-field generator} is given by $\bs\mcM:\mcP_p(\Pi)\to\mcG(\Pi^N)$, where
	\begin{align*}
		\bs\mcM(\mu) &= \bs\mcM^{(1)}(\mu) + \bs\mcM^{(2)}(\mu) + \cdots + \bs\mcM^{(N)}(\mu),\\
		\bs\mcM^{(1)}(\mu) &= \mcA(\mu) \otimes I \otimes I \otimes \cdots \otimes I,\\
		\bs\mcM^{(2)}(\mu) &= I \otimes \mcA(\mu) \otimes I \otimes \cdots \otimes I,\\
		&\vdots\\
		\bs\mcM^{(N)}(\mu) &= I \otimes I \otimes I \otimes \cdots \otimes \mcA(\mu).
	\end{align*}
	That is, for each mean-field measure \(\mu\in\mcp_p(\Pi)\), \(\bs{\mcM}(\mu)\) is the generator of the \(N\)-fold tensor-product semigroup \(\{(e^{t\mcA(\mu)})^{\otimes N}\}_{t\ge0}\).
\end{definition}

Our next task is to show that the tensorized curve \(\bs{\br}_t=\br_t^{\otimes N}\) solves \eqref{eq5:expected-tensor-evolution} in the sense of $\bs\mcW_p$-stable solutions. To this end, we first verify that the \(N\)-tensorized mean-field generator \(\bs{\mcM}\) takes values in \(\mcg_p^0(\Pi^N)\), and we then establish a \(p\)-exponential stability condition for \(\bs{\mcM}\) from Hypothesis \refA on the mean-field generator. These steps are carried out in the following lemmas.

\begin{lemma}\label{lem5.24:temp}
    Suppose $\mcA\in\mcG_p^0(\Pi)$, and let $\bs\mcA$ be the Feller generator of the tensorized semigroup $\scb{(e^{t\mcA})^{\oN}}_{t\ge0}$. Then $\bs\mcA\in\mcG_p^0(\Pi^N)$.
\end{lemma}

\begin{proof}
    By Proposition \ref{prop:characterization-Gp0}, to show that $\bs\mcA\in\mcG_p^0(\Pi^N)$, we must verify that for all $\Bmu\in\mcP_p(\Pi^N)$ and $t\ge0$, it holds that $\Bmu e^{t\bs\mcA}\in\mcP_p(\Pi^N)$, and the map
    \[t\mapsto\inn{\Bmu,e^{t\bs\mcA}\bs c_{p}(\bs z,\cdot)}\]
    is continuous for some fixed $\bs z\in\Pi^N$, recalling $\bs c_{p}(\Bx,\By) = \frac{1}{N}\sum_{k=1}^N c_p(x_k,y_k)$, $c_p=\f 1pd^p$.

    Let $\Bmu\in\mcP_p(\Pi^N)$ and fix $t\ge0$. For $1\le k\le N$, let $\mu_k\in\mcP_p(\Pi)$ denote the $k$-th marginal of $\Bmu$. Since $e^{t\bs\mcA}$ is the tensorized semigroup, the $k$-th marginal of $\Bmu e^{t\bs\mcA}$ is given by $\mu_k e^{t\mcA}$. Because $\mcA\in\mcG_p^0(\Pi)$, Proposition \ref{prop:characterization-Gp0} ensures that $\mu_k e^{t\mcA}\in\mcP_p(\Pi)$ for each $k$. Since all its marginals have finite $p$-th moments, it follows that the joint measure $\Bmu e^{t\bs\mcA}\in\mcP_p(\Pi^N)$.

    To verify the continuity condition, we use the tensorized structure to find
    \begin{align*}
        e^{t\bs\mcA}\bs c_{p}(\bs z,\cdot)(\Bx) = \frac{1}{N}\sum_{k=1}^N e^{t\mcA}c_p(z_k,\cdot)(x_k).
    \end{align*}
    Integrating this with respect to $\Bmu$ yields
    \begin{align}\label{eq5.6:temp}
        \inn{\Bmu,e^{t\bs\mcA}\bs c_{p}(\bs z,\cdot)} = \frac{1}{N}\sum_{k=1}^N \inn{\mu_k,e^{t\mcA}c_p(z_k,\cdot)}.
    \end{align}
    Since $\mcA\in\mcG_p^0(\Pi)$ and $\mu_k\in\mcP_p(\Pi)$, the map $t\mapsto\inn{\mu_k,e^{t\mcA}c_p(z_k,\cdot)}$ is continuous for each $k$ by Proposition \ref{prop:characterization-Gp0}. Therefore, the finite sum in \eqref{eq5.6:temp} is continuous in $t$, which completes the proof.
\end{proof}

\begin{corollary}\label{cor:mcM-is-gc0}
    The range of the $N$-tensorized mean-field generator $\bs\mcM$ is in $\mcG_p^0(\Pi^N)$. That is, for all $\rho\in\mcP_p(\Pi)$, we have $\bs\mcM(\rho)\in\mcG_p^0(\Pi^N)$.
\end{corollary}
\begin{proof}
    Fix $\rho\in\mcP_p(\Pi)$, and write $\mcM=\mcA(\rho)$, $\bs\mcM^{(k)}=\bs\mcM^{(k)}(\rho)$ and $\bs\mcM=\bs\mcM(\rho)$. We note that $\bs\mcM$ is a Feller generator. In fact, $\bs\mcM$ is the generator of the tensor product semigroup, i.e., $e^{t\bs\mcM}=(e^{t\mcM})^{\oN}$, see Definition \ref{def:independent-generator-mcM}. Then by Lemma \ref{lem5.24:temp}, since $\mcM\in\mcG_p^0(\Pi)$, we have $\bs\mcM\in\mcG_p^0(\Pi^N)$.
\end{proof}

\begin{lemma}\label{lem5.26:temp}
    Let $\MFG$ be a mean-field generator satisfying Hypothesis {\refA} with constants $\al,\be\ge 0$, and $\bs\mcM$ be the associated $N$-tensorized mean-field generator. Then for all $\Bmu,\Bnu\in\mcP_p(\Pi^N)$ and $\mu',\nu'\in\mcP_p(\Pi)$, it holds
    \[\bs \om_{p}(\Bmu,\Bnu;\bs\mcM(\mu'),\bs\mcM(\nu'))\le \al\mcW_p^p(\mu',\nu')+\be \bs\mcW_p^p(\Bmu,\Bnu).\]
\end{lemma}
\begin{proof}
	By Corollary \ref{cor:mcM-is-gc0}, we have $\bs\mcM(\rho)\in\mcG_p^0(\Pi^N)$ for all $\rho\in\PPP$. Hence, by Corollary \ref{cor3.20:chap3-main-cor}, it suffices to establish:
	\begin{align*}
		\bs\mcW_p^p(\de_{\Bx}e^{t\bs\mcM(\mu')},\de_{\By}e^{t\bs\mcM(\nu')}) \le e^{\be t}\Bc_p(\Bx,\By) + \al \mcW_p^p(\mu',\nu')\zeta_{\be}(t).
	\end{align*}
	Since for each $\mu',\nu'\in\mcP_p(\Pi)$, $\cb{e^{t\bs\mcM(\mu')}}_{t\ge0}$, $\cb{e^{t\bs\mcM(\nu')}}_{t\ge0}$ are tensorized Feller semigroups, their laws are in the tensorized form:
	\begin{align*}
		\de_{\Bx} e^{t\bs\mcM(\mu')} &= \de_{x_1}e^{t\mcA(\mu')}  \otimes \cdots \otimes \de_{x_N}e^{t\mcA(\mu')},\qquad 
		\de_{\By} e^{t\bs\mcM(\nu')} = \de_{y_1}e^{t\mcA(\nu')}  \otimes \cdots \otimes \de_{y_N}e^{t\mcA(\nu')}.
	\end{align*}
	Applying Lemma \ref{lem:cost-of-tensorized-measures} above, we arrive at
	\begin{align*}
		\bs\mcW_p^p(\de_{\Bx} e^{t\bs\mcM(\mu')},\de_{\By} e^{t\bs\mcM(\nu')}) &\le \frac{1}{N}\sum_{k=1}^N \mcW_p^p(\de_{x_k} e^{t\mcA(\mu')},\de_{y_k} e^{t\mcA(\nu')})\\
		&\le \frac{1}{N}\sum_{k=1}^N \rb{e^{\be t}c_p(x_k,y_k)+\al\mcW_p^p(\mu',\nu')\zeta_{\be}(t)}\\
		&= e^{\be t}\Bc_p(\Bx,\By) + \al \mcW_p^p(\mu',\nu')\zeta_{\be}(t).
	\end{align*}
	We have used the fact that $\mcA$ satisfies Hypothesis {\refA} in the second step.
\end{proof}

%===================================================
\subsubsection{Evolution problem of the tensorized mean-field measure}
Finally, here is the main result of this subsection, which is the evolution equation that governs the tensorized measure $\Bbr_t=\br_t^{\oN}$, as outlined earlier during our discussion on the strategy for proving propagation of chaos.
\begin{proposition}
    Let $\MFG$ be a mean-field generator that satisfies Hypothesis {\refA}. Let $\cb{\br_t}_{t\ge0}\in C([0,T];\mcP_p(\Pi))$ be a $\mcW_p$-stable solution of the mean-field equation \eqref{eq4.6:mean-field-equation}. Then $\cb{\Bbr_t}_{t\ge0}\in C([0,T];\mcP_{p}(\Pi^N))$ defined by $\bs{\br}_t = \br_t^{\oN}$ is a $\bs\mcW_p$-stable solution to the following initial value problem: 
    \begin{align}\label{eq:tensorized-mf-prob}
        \left\{\begin{aligned}
            \partial_t \bs{\br}_t &= \bs{\br}_t \bs{\mcM}(\br_t),\quad t\in(0,T),\\
            \bs{\br}_0 &= \br_0^{\otimes N}\in\mcP_{p}(\Pi^N).
        \end{aligned}\right.
    \end{align}
\end{proposition}
\begin{proof}
	Let $\BMM_t= \BMM(\br_t)$. The proof has two main steps: (1) show that $\cb{\BMM_t}_{t\ge0}$ is $\Exp$-continuous, and hence Corollary \ref{cor4.19temp} guarantees the evolution problem \eqref{eq:tensorized-mf-prob} admits a unique $\bs\mcW_p$-stable solution for every initial condition; (2) show that $\Bbr_t=\br_t^{\oN}$ is a $\bs\mcW_p$-stable solution of \eqref{eq:tensorized-mf-prob}.
	
	(1) $\cb{\BMM_t}_{t\ge0}$ is $\Exp$-continuous. This is because by Lemma \ref{lem5.26:temp},
	\begin{align*}
		\bs\om_{p}(\Bmu,\Bnu;\bs\mcM_t,\bs\mcM_s)\le \al\mcW_p^p(\br_t,\br_s)+\be \bs\mcW_p^p(\Bmu,\Bnu).
	\end{align*}
	Since $(t,s)\mapsto\mcW_p^p(\br_t,\br_s)$ is continuous on $[0,T]^2$, which is a compact domain in $\mbr^2$, it is uniformly continuous. That is, $\mcW_p^p(\br_t,\br_s)$ is small whenever $|t-s|$ is small (independent of the actual values of $t,s$). Hence, $\cb{\BMM_t}_{t\ge0}$ is $\Exp$-continuous.
	
	(2) $\cb{\Bbr_t}_{t\ge0}$ given by $\bs{\br}_t = \br_t^{\oN}$ is a $\bs\mcW_p$-stable solution of \eqref{eq:tensorized-mf-prob}. Let $\scb{\scb{\BMM_t^{(n)}}_{t\ge0}}_n$ be a PCA sequence of $\scb{\BMM_t}_{t\ge0}$. From the definition of the $N$-tensorized mean-field generator, $\BMM_t^{(n)}$ also takes the form
	\begin{align*}
		\bs\mcM_t^{(n)} &= \bs\mcM_t^{(n,1)} + \bs\mcM_t^{(n,2)} + \cdots + \bs\mcM_t^{(n,N)},\\
		\bs\mcM_t^{(n,1)} &= \mcA_t^{(n)} \otimes I \otimes I \otimes \cdots \otimes I,\\
		\bs\mcM_t^{(n,2)} &= I \otimes \mcA_t^{(n)} \otimes I \otimes \cdots \otimes I,\\
		&\vdots\\
		\bs\mcM_t^{(n,N)} &= I \otimes I \otimes I \otimes \cdots \otimes \mcA_t^{(n)},
	\end{align*}
	where $\scb{\scb{\mcA_t^{(n)}}_{t\ge0}}_n$ is a PCA sequence of $\scb{\mcA(\br_t)}_{t\ge0}$. For each $n$, let $\scb{\bs{\br}_t^{(n)}}_{t\ge0}$ be the solution of the evolution problem with generators $\scb{\BMM_t^{(n)}}_{t\ge0}$. On each interval on which $t\mapsto\BMM_t^{(n)}$ is constant, $\BMM_t^{(n)}$ is the infinitesimal generator of a Feller semigroup. Hence, we see that the PCA solutions $\Bbr_t^{(n)}$ must be given by
	\begin{align*}
		\Bbr_t^{(n)} = (\br_t^{(n)})^{\oN},
	\end{align*}
	where $\scb{\br_t^{(n)}}_{t\ge0}$ is the solution of the evolution problem with generator $\scb{\mcA_t^{(n)}}_{t\ge0}$.
 
    Since $\cb{\br_t}_{t\ge0}$ is a $\mcW_p$-stable solution of \eqref{eq4.6:mean-field-equation}, it follows that $\br_t^{(n)}\to\br_t$ in $\CTPC$. By Corollary \ref{cor5.21:temp}, we have $\Bbr_t^{(n)}\to\Bbr_t$ in $C([0,T];\mcP_{p}(\Pi^N))$. In total, for every PCA sequence $\scb{\BMM_t^{(n)}}$ of $\BMM_t$, the solutions $\scb{\Bbr_t^{(n)}}_n$ of the corresponding evolution equation converge to $\Bbr_t$. Therefore, $\Bbr_t$ is a $\bs\mcW_p$-stable solution.
\end{proof}

%=============================================================================
\subsection{Proof of Theorem \ref{thm:chap5-main-thm}}\label{sec5.4}
We shall now return to the proof of the main theorem. Recall that our goal is to control the Wasserstein distance between the flows $\Brho_t=\Brho_0 e^{t\bs{\hat\mcA}}$ and $\Bbr_t=\br_t^{\oN}$, where we have shown that $\Bbr_t$ is the $\bs\mcW_p$-stable solution of \eqref{eq:tensorized-mf-prob}. Thus, Lemma \ref{special-stability} is applicable, which yields an integral inequality. The remaining part of the proof is to ``close'' the inequality here, that is, to establish the bound of the form
\begin{align*}
	\bs\om_{p}(\Bmu,\Bnu;\BHA,\BMM(\br_t)) \le \ep_N(\br_t) + K \bs\mcW_p^p(\Bmu,\Bnu).
\end{align*}
To do so, we observe that $\BHA,\BMM(\rho)$ have a similar superposition structure:
\begin{align*}
	\BHA=\sum_{k=1}^N \BHA^{(k)},\quad \BMM(\rho)=\sum_{k=1}^N \BMM^{(k)}(\rho).
\end{align*}
A natural approach is to first establish a bound on $\bs\om_{p}(\cdot,\cdot;\BHA^{(k)},\BMM^{(k)}(\rho))$ for each $1\le k\le N$, then obtain a bound for their superposition by exploiting the subadditive nature of the $p$-transport derivative $\bs\om_{p}$ (Theorem \ref{thm:om-subadditivity}). The precise steps are detailed in the next two results.

\begin{lemma}\label{lem5.28:temp}
    Assume the settings from Theorem \ref{thm:chap5-main-thm}.
    For $1\le k\le N$, let $\bs{\hat\mcA}^{(k)}$ be given by \eqref{eq:bold-hat-mcA-1} and \eqref{eq:bold-hat-mcA-k}, and $\bs{\mcM}^{(k)}$ be from Definition \ref{def:independent-generator-mcM}.
    Then for each $1\le k\le N$, it holds for all $\Bx,\By\in\Pi^N$ and $\rho\in\mcP_p(\Pi)$ that
    \begin{align*}
        \bs\om_{p}(\Bx,\By;\BHA^{(k)},\BMM^{(k)}(\rho)) \le \frac{1}{N}\sqb{\al \mcW_p^p(\mu(\Bx_k'),\rho)+\be c_p(x_k,y_k)},
    \end{align*}
    where $\mu(\Bx_k')=\frac{1}{N-1}\sum_{j\neq k}\de_{x_j}$.
\end{lemma}
\begin{proof}	
    Fix any $\rho\in\mcP_p(\Pi)$.
    By the symmetric structure of $\bs{\hat\mcA}^{(k)}$ and $\bs\mcM^{(k)}(\rho)$, it suffices to consider $k=1$. Let us recall the notation
    \begin{align*}
        \Bx=(x_1,x_2,\cdots,x_N)=(x_1;\Bx_1'),\quad \By=(y_1,y_2,\cdots,y_N)=(y_1;\By_1'),
    \end{align*}
    where $\Bx_1'=(x_2,\cdots,x_N), \By_1'=(y_2,\cdots,y_N)\in\Pi^{N-1}$. For notational simplicity, we shall use the shorthand notation $\mu=\mu(\Bx_1')$. Notice that by the structure of $\bs{\hat\mcA}^{(1)}$ and $\bs\mcM^{(1)}(\rho)$ (see \eqref{eq5.2:temp} and Definition \ref{def:independent-generator-mcM}), we have
    \begin{align*}
        \de_{\Bx}e^{t\bs{\hat\mcA}^{(1)}} &= \de_{x_1}e^{t\mcA(\mu)}\otimes\de_{x_2}\otimes\cdots\otimes\de_{x_N},\qquad 
        \de_{\By}e^{t\bs\mcM^{(1)}(\rho)} = \de_{y_1}e^{t\mcA(\rho)}\otimes\de_{y_2}\otimes\cdots\otimes\de_{y_N}.
    \end{align*}
    Then by Lemma \ref{lem:cost-of-tensorized-measures},
    \begin{align*}
        \bs\mcW_p^p\mrb{\de_{\Bx}e^{t\bs{\hat\mcA}^{(1)}},\de_{\By}e^{t\bs\mcM^{(1)}(\rho)}} &\le \frac{1}{N} \mcW_p^p\mrb{\de_{x_1}e^{t\mcA(\mu)},\de_{y_1}e^{t\mcA(\rho)}}+ \frac{1}{N} \sum_{k\ge 2} c_p(x_k,y_k).
    \end{align*}
    Using Hypothesis {\refA} for the mean-field generator $\mcA$ and Theorem \ref{thm:chap3-main-thm}, we can bound the first term on the right-hand side above to get
    \begin{align*}
        \bs\mcW_p^p\mrb{\de_{\Bx}e^{t\bs{\hat\mcA}^{(1)}},\de_{\By}e^{t\bs\mcM^{(1)}(\rho)}} &\le \frac{1}{N} \msqb{e^{\be t} c_p(x_1,y_1) + \al\mcW_p^p(\mu,\rho)\int_0^t e^{\be(t-s)}\,ds} + \frac{1}{N} \sum_{k\ge 2} c_p(x_k,y_k).
    \end{align*}
    Subtracting $\Bc_{p}(\Bx,\By)$ and dividing by $t$, then taking the limsup as $t\searrow0$, we find
    \begin{align*}
        &\bs\om_{p}\mrb{\Bx,\By;\bs{\hat\mcA}^{(1)},\bs\mcM^{(1)}(\rho)}\\
        &\le \limsup_{t\searrow0}\frac{1}{Nt} \msqb{(e^{\be t}-1) c_p(x_1,y_1) + \al\mcW_p^p(\mu,\rho)\int_0^t e^{\be(t-s)}\,ds}\\
        &= \frac{\be}{N} c_p(x_1,y_1) + \frac{\al}{N}\mcW_p^p(\mu,\rho).
    \end{align*}
    Therefore, we obtain the claimed inequality (with $k=1$).
\end{proof}

\begin{corollary}\label{cor5.29:temp}
    Assume the settings of Theorem \ref{thm:chap5-main-thm}. Let $\bs{\hat\mcA}=\sum_{k=1}^N\BHA^{(k)}$ be the associated $N$-particle generator from Definition \ref{def:N-particle-generator} and $\bs{\mcM}=\sum_{k=1}^N\bs{\mcM}^{(k)}$ be the associated $N$-tensorized mean-field generator from Definition \ref{def:independent-generator-mcM}.
    Then it holds for all $\bs{\mu},\bs{\nu}\in\mcP_p(\Pi^N)$ and $\rho\in\mcP_p(\Pi)$ that
    \begin{align*}
        \bs\om_{p}(\bs{\mu},\bs{\nu};\bs{\hat{\mcA}},\bs{\mcM}(\rho))
        \le \frac{\al}{N} \sum_{k=1}^N \int_{\Pi^N} \mcW_p^p(\mu(\Bx_k'),\rho)\,d\bs{\mu}(\Bx) + \be \bs\mcW_p^p\mrb{\bs{\mu},\bs{\nu}}.
    \end{align*}
\end{corollary}
\begin{proof}
    Fix any $\rho\in\mcP_p(\Pi)$.
    For each $1\le k\le N$, let
    \begin{align*}
        \bs F_k(\Bx) = \mcW_p^p(\mu(\Bx_k'),\rho),\qquad \bs\ta_k(\Bx,\By) = \frac{1}{N}\sqb{\al \bs F_k(\Bx)+\be c_p(x_k,y_k)},
    \end{align*}
    and $\bs F(\Bx) = \frac{1}{N} \sum_{k=1}^N \bs F_k(\Bx)$.
    By Lemma \ref{lem5.28:temp}, we have
    \begin{align*}
        \bs\om_{p}(\Bx,\By;\BHA^{(k)},\BMM^{(k)}(\rho)) \le \bs\ta_k(\Bx,\By).
    \end{align*}
    Since $\mcW_p$ and $c_p$ are continuous, $\bs\ta_k:\Pi^{2N}\to\mbr_+$ is a continuous function.
    We claim that $\bs\ta_k$ satisfies the condition given in Theorem \ref{thm:om-subadditivity}, that is, for some $\bs\Phi,\bs\Psi\in C_{p}(\Pi^N)$, it holds
    \begin{align*}
        \bs\ta_k(\Bx,\By)&=\frac \al N\bs F_k(\Bx)+ \frac{\be}{N}c_p(x_k,y_k)\le \bs \Phi(\Bx)+\bs \Psi(\By),\qquad (\Bx,\By) \in \Pi^{2N}.
    \end{align*}
    This is because $\bs F_k\in C_{p}(\Pi^N)$ (see Lemma \ref{lem5.30:temp} below) and $c_p(x,y)\le 2^{p-1}[c_p(x,z)+c_p(z,y)]$.
    
    Applying the subadditivity theorem for $\bs\om_{p}$ (Theorem \ref{thm:om-subadditivity})\footnote{The core for $\BHA$ required for Theorem \ref{thm:om-subadditivity} comes from Hypothesis {\refAp}. For $\bs\mcM(\rho)$, the core can be taken to be $\mcD:=D(\mcA(\rho))^{\oN}$. It is clear that $\mcD\subset D(\bs\mcM^{(k)}(\rho))$ for each $k$, and one can verify that $\mcD$ is indeed a core. }, we find
    \begin{align*}
        \bs\om_{p}\mrb{\Bx,\By;\bs{\hat\mcA},\bs\mcM(\rho)} \le \sum_{k=1}^N\bs\ta_k(\Bx,\By) = \al \bs F(\Bx)+\be \Bc_{p}(\Bx,\By).
    \end{align*}
    Finally, by Theorem \ref{thm:chap3-main-thm}, we obtain
    \begin{align*}
        \bs\om_{p}(\bs{\mu},\bs{\nu};\bs{\hat{\mcA}},\bs{\mcM}(\rho))
        \le \int_{\Pi^N} \al \bs F(\Bx)\,d\bs{\mu}(\Bx) + \be \bs\mcW_p^p\mrb{\bs{\mu},\bs{\nu}},
    \end{align*}
    which is the desired result.
\end{proof}

\begin{lemma}\label{lem5.30:temp}
    Given $\rho\in\mcP_p(\Pi)$, define $\bs F_1(\Bx)= \mcW_p^p(\mu(\Bx_1'),\rho).$
    Then $\bs F_1\in C_{p}(\Pi^N)$. 
\end{lemma}
\begin{proof}
	To show \(\bs F_1\in C_{p}(\Pi^N)\), it suffices to find some \(\bs z\in \Pi^N\) and constants \(M_0,M_1\ge 0\) such that
	\begin{align}\label{eq:temp5.1}
		\bs F_1(\Bx) \le M_0 + M_1\Bc_{p}(\bs z,\Bx),
		\qquad \Bx \in \Pi^N.
	\end{align}
	Fix some \(\bs z\in \Pi^N\), and write \(\bs z_1'=(z_2,\dots,z_N)\). By the triangle inequality for \(\mcW_p\), we have
	\[
	\bs F_1(\Bx)^{1/p}
	=
	\mcW_p\bigl(\mu(\Bx_1'),\rho\bigr)
	\le
	\mcW_p\bigl(\mu(\Bx_1'),\mu(\bs z_1')\bigr)
	+
	\mcW_p\bigl(\mu(\bs z_1'),\rho\bigr).
	\]
	For the first term, using \eqref{eq5.5:temp}, we get
	\[
	\mcW_p\bigl(\mu(\Bx_1'),\mu(\bs z_1')\bigr)
	\le
	\Bc_{p,N-1}(\bs z_1',\Bx_1')^{1/p}.
	\]
	Inserting this bound into the preceding inequality, then raising both sides to the \(p\)-th power and noting that
	\(\Bc_{p,N-1}(\bs z_1',\Bx_1')\le C_N\,\Bc_p(\bs z,\Bx)\)
	for some constant \(C_N>0\) depending only on \(N\), we obtain \eqref{eq:temp5.1} for suitable \(M_0,M_1\ge0\). Note that $M_1$ depends on $N$ while $M_0$ depends on $\bs z$, but it does not matter here. 
\end{proof}

We may now complete the proof of Theorem \ref{thm:chap5-main-thm}.
\begin{proof}[Proof of Theorem \ref{thm:chap5-main-thm}]
    The main step of the proof is to establish an integral inequality for $\bs\mcW_p(\bs\rho_t,\bs\br_t)$, and then apply Gr\"{o}nwall's inequality. First, recall that $\bs\br_t$ is a $\bs\mcW_p$-stable solution of 
    \begin{align*}
        \partial_t \bs{\br}_t = \bs{\br}_t \bs{\mcM}(\br_t).
    \end{align*}
    We shall use the stability estimate for optimal cost between $\Brho_t = \Brho_0 e^{t\BHA}$ and a $\bs\mcW_p$-stable solution $\bs\br_t$ (Lemma \ref{special-stability}), where we have the integral inequality of the form
    \begin{align*}
        \bs\mcW_p^p(\bs\rho_t,\bs\br_t) \le \bs\mcW_p^p(\bs\rho_0,\bs\br_0) + \int_0^t \ta(s,\Brho_s,\Bbr_s)\,ds,
    \end{align*}
    where $\ta$ will be specified later.
    Our goal now is to use Corollary \ref{cor5.29:temp} to find a suitable $\ta$. By the corollary, we have
    \begin{align*}
        \bs\om_{p}(\bs{\mu},\bs{\nu};\bs{\hat{\mcA}},\bs{\mcM}(\br_s))
        \le \frac{\al}{N} \sum_{k=1}^N \int_{\Pi^N} \mcW_p^p(\mu(\Bx_k'),\br_s)\,d\bs{\mu}(\Bx) + \be \bs\mcW_p^p\mrb{\bs{\mu},\bs{\nu}}.
    \end{align*}
    For convenience, let ($\ast$) denote the first term on the right-hand side above, without the factor $\al$. We want to obtain a more convenient bound for ($\ast$). Let $\bs\ga_s$ be an optimal coupling of $\bs\mu$ and $\bs\br_s$. Then
    \begin{align*}
        (\ast) = \frac{1}{N} \sum_{k=1}^N \int_{\Pi^{2N}} \mcW_p^p(\mu(\Bx_k'),\br_s)\,d\bs{\ga}_s(\Bx,\By).
    \end{align*}
    Using the standard triangle inequality for the Wasserstein metric, we have
    \begin{align*}
        \mcW_p(\mu(\Bx_k'),\br_s) \le \mcW_p(\mu(\Bx_k'),\mu(\By_k')) + \mcW_p(\mu(\By_k'),\br_s).
    \end{align*}
    Taking the $L^p$ norm of this inequality with respect to the discrete average over $k \in \{1,\dots,N\}$ and the coupling measure $\bs\ga_s$, Minkowski's inequality yields
    \begin{align}\label{eq5.8:temp}
        (\ast)^{1/p} &= \mrb{\frac{1}{N} \sum_{k=1}^N \int_{\Pi^{2N}} \mcW_p^p(\mu(\Bx_k'),\br_s)\,d\bs{\ga}_s(\Bx,\By)}^{1/p}\nonumber\\
        &\le \mrb{\frac{1}{N} \sum_{k=1}^N \int_{\Pi^{2N}} \mcW_p^p(\mu(\Bx_k'),\mu(\By_k'))\,d\bs{\ga}_s(\Bx,\By)}^{1/p}\nonumber\\
        &\quad + \mrb{\frac{1}{N} \sum_{k=1}^N \int_{\Pi^{2N}} \mcW_p^p(\mu(\By_k'),\br_s)\,d\bs{\ga}_s(\Bx,\By)}^{1/p}.
    \end{align}
    For the first term on the right-hand side of \eqref{eq5.8:temp}, we recall from \eqref{eq5.5:temp} that $\mcW_p^p(\mu(\Bx_k'),\mu(\By_k')) \le \frac{1}{N-1} \sum_{j\neq k} c_p(x_j,y_j)$. Summing this up yields
    \begin{align}\label{eq:tool5.3}
        \frac{1}{N} \sum_{k=1}^N \mcW_p^p(\mu(\Bx_k'),\mu(\By_k')) \le \frac{1}{N(N-1)} \sum_{k=1}^N \sum_{j\neq k} c_p(x_j,y_j) = \frac{1}{N} \sum_{k=1}^N c_p(x_k,y_k) = \Bc_{p}(\Bx,\By).
    \end{align}
    Integrating this bound against the optimal coupling $\bs\ga_s$ shows that the first term is bounded by $\bs\mcW_p(\bs\mu,\bs\br_s)$.
    For the second term, since the marginal of $\bs\ga_s$ on $\By$ is $\bs\br_s = \br_s^{\otimes N}$, we recognize the definition of $\aleph_{N}(\br_s)$:
    \begin{align*}
        \mrb{\frac{1}{N} \sum_{k=1}^N \int_{\Pi^N} \mcW_p^p(\mu(\By_k'),\br_s)\,d\bs{\br}_s(\By)}^{1/p} = \aleph_{N}(\br_s)^{1/p} \le \sup_{\tau\in[0,T]}\aleph_{N}(\br_\tau)^{1/p}.
    \end{align*}
    Putting these together, we obtain
    \begin{align*}
        (\ast)^{1/p} \le \bs\mcW_p(\Bmu,\Bbr_s) + \sup_{\tau\in[0,T]}\aleph_{N}(\br_\tau)^{1/p}.
    \end{align*}
    Raising this to the $p$-th power and applying the elementary algebraic inequality $(a+b)^p \le 2^{p-1}(a^p + b^p)$, we establish
    \begin{align*}
        (\ast) \le B\bs\mcW_p^p(\Bmu,\Bbr_s) + B \sup_{\tau\in[0,T]}\aleph_{N}(\br_\tau),
    \end{align*}
    where $B := 2^{p-1}$.
    Substituting this bound back gives
    \begin{align*}
        \bs\om_{p}(\bs{\mu},\bs{\nu};\bs{\hat{\mcA}},\bs{\mcM}(\br_s)) &\le \be \bs\mcW_p^p\mrb{\bs{\mu},\bs{\nu}} + \al B\bs\mcW_p^p(\bs\mu,\bs\br_s) + \al B \sup_{\tau\in[0,T]}\aleph_{N}(\br_\tau)
        =: \ta(s,\Bmu,\Bnu).
    \end{align*}
    Since $\scb{\bs\br_s}_{s\in[0,T]}$ is a bounded curve in $\mcP_p(\Pi^N)$ and $\ta$ is continuous, we can apply Lemma \ref{special-stability}. Evaluating $\ta$ at $(s,\bs\rho_s,\bs\br_s)$, we have
    \begin{align*}
        \ta(s,\bs\rho_s,\bs\br_s) = (\be + \al B) \bs\mcW_p^p\mrb{\bs\rho_s,\bs\br_s} + \al B \sup_{\tau\in[0,T]}\aleph_{N}(\br_\tau).
    \end{align*}
    Then by Lemma \ref{special-stability}, we obtain the integral inequality
    \begin{align*}
        \bs\mcW_p^p(\bs\rho_t,\bs\br_t) \le \bs\mcW_p^p(\bs\rho_0,\bs\br_0) + \int_0^t \sqb{(\be+\al B) \bs\mcW_p^p(\bs\rho_s,\bs\br_s) + \al B\sup_{\tau\in[0,T]} \aleph_{N}(\br_\tau)}\,ds.
    \end{align*}
    Finally, by Gr\"{o}nwall's inequality and setting $K := \be+\al B$ and $C := \al B$, we obtain
    \begin{align*}
        \bs\mcW_p^p(\bs\rho_t,\bs\br_t) \le \bs\mcW_p^p(\bs\rho_0,\bs\br_0) e^{K t} + C\zeta_{K}(t) \sup_{s\in[0,T]} \aleph_{N}(\br_s).
    \end{align*}
    Taking the supremum over $t\in[0,T]$ yields the desired result.
\end{proof}

%=============================================================================
% Averaged Hilbert-valued features and the variance refinement
%=============================================================================
\subsection{An $O(N^{-1})$ estimate under averaged Hilbert-space dependence}\label{sec:hilbert-refinement}

In this final subsection, we revisit Theorem~\ref{thm:chap5-main-thm} under an additional Hilbert-space structure, obtaining an $O(N^{-1})$ sampling contribution to the quantitative propagation-of-chaos estimate. In contrast to the empirical transport error $\aleph_N$ appearing in that theorem, the error here is controlled by fluctuations of averaged Hilbert-valued features. This cancellation mechanism is classical for McKean--Vlasov models with averaged coefficients; see \cite{sznitman1991topics}. We formulate it in the abstract setting and illustrate it with a class of L\'evy-type mean-field generators in Section~\ref{sec6:averaged-amplitudes}.

Let us now introduce the main assumption.

\begin{hypothesisH}[Averaged Hilbert-valued features]\label{hyp:hilbert-features}
Let $\mcA:\mcP_2(\Pi)\to\mcG_2^0(\Pi)$ be a mean-field generator. There exist a separable Hilbert space $\mathcal H$, a map $h:\Pi^2\to\mathcal H$, and constants $C,\be,L_1,L_2\ge0$ such that
\begin{align}\label{eq:hilbert-feature-lipschitz}
 \|h(x,z)-h(y,w)\|_{\mathcal H}^2\le L_1c_2(x,y)+L_2c_2(z,w)
\end{align}
for all $x,y,z,w\in\Pi$. Define the Bochner integral
\begin{align*}
 H(x,\mu)=\int_\Pi h(x,z)\,\mu(dz).
\end{align*}
For all $x,y\in\Pi$ and $\mu,\nu\in\mcP_2(\Pi)$, assume
\begin{align}\label{eq:hilbert-generator-bound}
 \om_2(x,y;\mcA(\mu),\mcA(\nu))
 \le \be c_2(x,y)+C\|H(x,\mu)-H(x,\nu)\|_{\mathcal H}^2.
\end{align}
\end{hypothesisH}

\begin{remark}
The Lipschitz condition ensures that the Bochner integrals defining $H(y,\mu)$ are well-defined. Moreover, Hypothesis \refH\ implies Hypothesis \refA. Indeed, let $\gamma$ be an optimal coupling of $\mu,\nu$. By Jensen's inequality and \eqref{eq:hilbert-feature-lipschitz},
\begin{align}\label{eq:hilbert-implies-wasserstein}
 \|H(y,\mu)-H(y,\nu)\|_{\mathcal H}^2
 &=\left\|\int_{\Pi^2}[h(y,z)-h(y,w)]\,\gamma(dz,dw)\right\|_{\mathcal H}^2\nonumber\\
 &\le\int_{\Pi^2}\|h(y,z)-h(y,w)\|_{\mathcal H}^2\,\gamma(dz,dw)\nonumber\\
 &\le L_2\int_{\Pi^2}c_2(z,w)\,\gamma(dz,dw)
 =L_2\mcW_2^2(\mu,\nu).
\end{align}
Thus Hypothesis \refA\ holds with $\al=CL_2$ and the same $\be$. Theorem~\ref{thm4.3:chap4-main-result} guarantees a unique $\mcW_2$-stable mean-field solution, and the preceding tensorization results apply. The generator need not depend linearly on the law; it is the bound on its transport derivative that involves averaged features.
\end{remark}

For $\rho\in\mcP_2(\Pi)$, define the integrated Hilbert variance
\begin{align}\label{eq:hilbert-variance}
 V(\rho)=\int_{\Pi^2}\|h(y,z)-H(y,\rho)\|_{\mathcal H}^2\,\rho(dz)\rho(dy).
\end{align}
The Hilbert variance identity and \eqref{eq:hilbert-feature-lipschitz} imply
\begin{align}\label{eq:hilbert-variance-moment}
 V(\rho)
 =\frac12\int_{\Pi^3}\|h(y,z)-h(y,w)\|_{\mathcal H}^2\,\rho(dy)\rho(dz)\rho(dw)
 \le \frac{L_2}{2}\int_{\Pi^2}c_2(z,w)\,\rho(dz)\rho(dw).
\end{align}
The middle expression has a continuous integrand of quadratic growth, so $V$ is finite and Wasserstein-$2$ continuous. In particular, $V(\br_t)$ is bounded on finite time intervals.

\begin{theorem}[Variance estimate for quantitative propagation of chaos]\label{thm:hilbert-poc}
Assume Hypotheses \refH and \refAp\ with $p=2$ and $N\ge2$. Let $\{\br_t\}_{t\ge 0}$ be the $\mcW_2$-stable solution of the mean-field equation, and set
\(
 \Brho_t=\Brho_0e^{t\BHA},\Bbr_t=\br_t^{\otimes N}.
\)
Then, with $K=\be+C(6L_1+3L_2)$, for every $t\ge0$,
\begin{align}\label{eq:hilbert-poc-estimate}
 \bs\mcW_2^2(\Brho_t,\Bbr_t)
 \le e^{Kt}\bs\mcW_2^2(\Brho_0,\Bbr_0)
 +\frac{3C}{N-1}\int_0^t e^{K(t-s)}V(\br_s)\,ds.
\end{align}
Consequently, if $\Brho_0=\br_0^{\otimes N}$, the normalized quadratic transport cost is $O(N^{-1})$ uniformly on each finite time interval.
\end{theorem}

We first record the counterpart of Corollary \ref{cor5.29:temp}, retaining the averaged-feature discrepancy in the bound.

\begin{lemma}\label{lem:hilbert-particle-bound}
Under the assumptions of Theorem \ref{thm:hilbert-poc}, for all $\rho\in\mcP_2(\Pi)$ and $\Bmu,\Bnu\in\mcP_2(\Pi^N)$,
\begin{align}\label{eq:hilbert-particle-bound}
 \bs\om_2(\Bmu,\Bnu;\BHA,\BMM(\rho))
 &\le\be\bs\mcW_2^2(\Bmu,\Bnu)
 +\frac{C}{N}\sum_{i=1}^N\int_{\Pi^N}
 \|H(x_i,\mu(\Bx_i'))-H(x_i,\rho)\|_{\mathcal H}^2\,d\Bmu(\Bx).
\end{align}
\end{lemma}

\begin{proof}
	We follow the arguments of Lemma \ref{lem5.28:temp} and Corollary \ref{cor5.29:temp}, replacing the Wasserstein discrepancy by the Hilbert discrepancy.
	Fix $\rho\in\mcP_2(\Pi)$ and set
	\[
	 \bs F_i(\Bx)=C\|H(x_i,\mu(\Bx_i'))-H(x_i,\rho)\|_{\mathcal H}^2,
	\]
	and $\bs F(\Bx)=\frac1N\sum_{i=1}^N\bs F_i(\Bx).$
	The coordinate tensorization argument gives
	\[
	 \bs\om_2(\Bx,\By;\BHA^{(i)},\BMM^{(i)}(\rho))
	 \le\frac1N\bigl[\bs F_i(\Bx)+\be c_2(x_i,y_i)\bigr].
	\]
	The functions $\bs F_i$ are continuous and have quadratic growth. Indeed, for any fixed $z_0\in\Pi$, \eqref{eq:hilbert-implies-wasserstein} gives
	\begin{align}\label{bdd:quadratic-5}
	 \|H(x_i,\mu(\Bx_i'))-H(x_i,\rho)\|_{\mathcal H}^2
	 \le2L_2\left[\frac1{N-1}\sum_{j\ne i}c_2(x_j,z_0)
	 +\int_\Pi c_2(z,z_0)\,\rho(dz)\right].
	\end{align}
	Thus the coordinate majorants satisfy the growth condition of Theorem \ref{thm:om-subadditivity}. Applying that theorem with the same common cores as in Corollary \ref{cor5.29:temp}, we obtain
	\[
	 \bs\om_2(\Bx,\By;\BHA,\BMM(\rho))\le\bs F(\Bx)+\be\Bc_2(\Bx,\By).
	\]
	Finally, Theorem \ref{thm:chap3-main-thm} yields \eqref{eq:hilbert-particle-bound}.
\end{proof}

Let us now finish the proof of Theorem \ref{thm:hilbert-poc}.

\begin{proof}[Proof of Theorem \ref{thm:hilbert-poc}]
As in the proof of Theorem \ref{thm:chap5-main-thm}, it suffices to establish
\begin{align}\label{ineq:majorant}
 \bs\om_2(\Bmu,\Bnu;\BHA,\BMM(\br_s))\le\ta(s,\Bmu,\Bnu),
\end{align}
with a continuous majorant $\ta$, then apply Lemma \ref{special-stability} and Gr\"{o}nwall's inequality.

The feature discrepancy satisfies, for all $x,y\in\Pi$ and $\mu,\nu,\rho\in\mcP_2(\Pi)$,
\begin{align}\label{eq:hilbert-feature-transfer}
 \|H(x,\mu)-H(x,\rho)\|_{\mathcal H}^2
 \le{}&6L_1c_2(x,y)+3L_2\mcW_2^2(\mu,\nu)
 +3\|H(y,\nu)-H(y,\rho)\|_{\mathcal H}^2.
\end{align}
This follows by splitting through $H(y,\nu)$ and $H(y,\rho)$, applying Jensen's inequality and \eqref{eq:hilbert-feature-lipschitz}, and using $\|a+b+c\|^2\le3(\|a\|^2+\|b\|^2+\|c\|^2)$.

Fix $s\in[0,T]$ and let $\bs\ga_s$ be an optimal coupling of $\Bmu$ and $\Bbr_s=\br_s^{\otimes N}$. Apply \eqref{eq:hilbert-feature-transfer} to the feature term of \eqref{eq:hilbert-particle-bound}, with $(x,y,\mu,\nu,\rho)=(x_i,y_i,\mu(\Bx_i'),\mu(\By_i'),\br_s)$. Integrating against $\bs\ga_s$ and using the empirical-measure estimate \eqref{eq:tool5.3}, exactly as in the proof of Theorem \ref{thm:chap5-main-thm}, bounds that term, without the factor $C$, by
\begin{align*}
 (6L_1+3L_2)\bs\mcW_2^2(\Bmu,\Bbr_s)
 +\frac3N\sum_{i=1}^N\int_{\Pi^N}
 \|H(y_i,\mu(\By_i'))-H(y_i,\br_s)\|_{\mathcal H}^2\,d\br_s^{\otimes N}(\By).
\end{align*}
For fixed $y_i$, the difference in the last integrand equals
\[
 \frac1{N-1}\sum_{j\ne i}\bigl[h(y_i,y_j)-H(y_i,\br_s)\bigr].
\]
The summands are centered under integration in $y_j$, so the cross terms in the squared Hilbert norm integrate to zero. Consequently,
\begin{align}\label{eq:hilbert-iid-identity}
 &\frac1N\sum_{i=1}^N\int_{\Pi^N}
 \|H(y_i,\mu(\By_i'))-H(y_i,\br_s)\|_{\mathcal H}^2\,d\br_s^{\otimes N}(\By)\nonumber\\
 &\qquad=\frac1{N-1}\int_{\Pi^2}
 \|h(y,z)-H(y,\br_s)\|_{\mathcal H}^2\,\br_s(dz)\br_s(dy)
 =\frac{V(\br_s)}{N-1}.
\end{align}
This variance identity replaces the estimate involving $\aleph_N$ in the proof of Theorem \ref{thm:chap5-main-thm}. Substituting these bounds into \eqref{eq:hilbert-particle-bound} gives \eqref{ineq:majorant}, for arbitrary $\Bmu,\Bnu$, with
\begin{align*}
 \ta(s,\Bmu,\Bnu):={}&\be\bs\mcW_2^2(\Bmu,\Bnu)
 +C(6L_1+3L_2)\bs\mcW_2^2(\Bmu,\Bbr_s)
 +\frac{3C}{N-1}V(\br_s).
\end{align*}
The continuity of $\br_s$ and $V$, together with continuity under tensorization, ensures that $\ta$ is jointly continuous. Evaluating at $(s,\Brho_s,\Bbr_s)$, and setting $K=\be+C(6L_1+3L_2)$, Lemma \ref{special-stability} gives
\begin{align*}
 \bs\mcW_2^2(\Brho_t,\Bbr_t)
 \le{}&\bs\mcW_2^2(\Brho_0,\Bbr_0)+\int_0^t\left[K\bs\mcW_2^2(\Brho_s,\Bbr_s)
                 +\frac{3C}{N-1}V(\br_s)\right]ds.
\end{align*}
Gr\"{o}nwall's inequality yields \eqref{eq:hilbert-poc-estimate}. The finite-time $O(N^{-1})$ conclusion follows from boundedness of $V(\br_s)$ on $[0,T]$.
\end{proof}

%% file: Section6.tex
\section{L\'{e}vy-type mean-field systems and propagation of chaos}\label{chap6}

%=============================================================================
\input{Subsection6.1}

%% file: Subsection6.1.tex
\subsection{Overview}
In this final section, we investigate \emph{L\'evy-type mean-field systems and their propagation of chaos} as a concrete application of the abstract framework developed above. By a L\'evy-type mean-field generator we mean a mean-field generator
\(\mA:\mcp_2(\mbr^d)\to\mcg_2^0(\mbr^d)\) of the form
\begin{align}
	\mA(\phi;\mu)(x)
	&= b(x,\mu)\cdot\nabla\phi(x)
	+\frac12 \tr[a(x,\mu)D^2\phi](x) \nonumber \\
	&\qquad
	+\int_{\mbr^{d}\setminus\{0\}}
	\bigl[\phi(x+y)-\phi(x)-y\cdot\nabla\phi(x)\bigr]
	\,\Theta(x,dy,\mu), \label{def:levy-meanfield}
\end{align}
where \((b,a,\Theta)\) is an \((x,\mu)\)-dependent \emph{L\'evy triplet}.

The class of L\'evy-type mean-field systems contains the classical McKean--Vlasov diffusion as a special case, corresponding to the absence of the jump component, \(\Theta\equiv0\). This diffusion case has been extensively studied in the literature (see, e.g., \cite{McKean1969, gartner1988mckean, andreis2018mckean, lacker2018, Chaintron_2022_2}). Generator-based formulations of L\'evy-type mean-field dynamics were developed by Kolokoltsov, who established well-posedness results for nonlinear measure-valued evolution equations together with particle-limit results, including a dynamic law of large numbers and quantitative convergence estimates for classes of interacting L\'evy-type systems \cite{Kolokoltsov2007}. He subsequently constructed nonlinear Markov evolutions under Wasserstein-type Lipschitz conditions on state- and law-dependent L\'evy triplets \cite{Kolokoltsov2009}. Such jump-type mean-field models arise naturally in kinetic and gas dynamics; they include, for instance, Nanbu-type particle approximations of collisional Boltzmann dynamics and piecewise deterministic Markov processes on \(\mbr^d\). Further illustrative examples are discussed in Section \ref{sec6:examples}.

Cavallazzi \cite{Cavallazzi2023} studies McKean--Vlasov equations driven by pure-jump L\'evy processes under Lipschitz assumptions. His well-posedness theorem allows a state- and law-dependent multiplicative noise coefficient, with the associated particle system written as
\begin{align}\label{eq:Cav-model}
	dX_t^i
	&= b\bigl(X_t^i, \mu_{\BX_t}\bigr)\,dt
	+\sigma\bigl(X_{t-}^i, \mu_{\BX_t}\bigr)\,dZ_t^i,
	\qquad t\in[0,T],
	\quad 1\le i\le N.
\end{align}
Here, \(b:\mbr^d\times\mcp(\mbr^d)\to\mbr^d\) is the drift coefficient, \(\sigma:\mbr^d\times\mcp(\mbr^d)\to\mcl(\mbr^d)\) is the matrix-valued jump coefficient, and the \(Z^i\) are independent copies of a fixed pure-jump L\'evy process with L\'evy measure \(\Omega\). This multiplicative-noise structure corresponds to a L\'evy-type generator with \(a\equiv0\) and jump kernel
\begin{align}\label{def:pushforward-levy}
	\Theta(x,\mu)
	= \sigma(x,\mu)_\sharp \Omega .
\end{align}
His quantitative propagation-of-chaos theorem assumes \(\sigma=\mathrm{Id}\), so the noise is additive. It gives pathwise estimates under moment assumptions of order \(p\in(1,2]\), allowing noises without finite second moments. Our results instead concern time-marginal Wasserstein estimates under finite-second-moment assumptions and permit jump kernels beyond the fixed-measure linear-pushforward form.

For general L\'evy-type dynamics, a suitable synchronous coupling need not be supplied by the model's representation. Our approach formulates the comparison directly through generators and transport costs. Under the generation and regularity assumptions stated below, the abstract transport-stability criteria yield well-posedness of the mean-field evolution and quantitative propagation of chaos, without requiring a prescribed coupling of sample paths. When the triplets depend on the population law through averaged Hilbert-valued features, an additional quantitative condition allows us to verify Hypothesis \refH\ and obtain an $O(N^{-1})$ sampling contribution to the quadratic transport estimate.

The stability condition is expressed in the geometry of the underlying transport problem. Classical coefficient-based conditions typically compare diffusion coefficients through the \emph{Frobenius (Hilbert--Schmidt) norm}. Our condition instead compares the covariance matrices through the \emph{Bures--Wasserstein distance}, which depends on the diffusion generators rather than on the chosen matrix factors. Lemma~\ref{lem:noise-amplitude-bounds} shows that this distance is bounded by the Frobenius discrepancy of any prescribed pair of diffusion amplitudes, up to the normalization factor. Thus a Frobenius bound implies the corresponding Bures--Wasserstein bound, while the latter may be strictly less restrictive. Likewise, jump kernels are compared directly through the \emph{L\'evy--Wasserstein distance}, without requiring a fixed-noise representation. These transport-type distances are intrinsic to the quadratic cost governing our stability estimates. Transport distances between L\'evy measures have been studied previously by Kolokoltsov \cite{Kolokoltsov2009} and by Guillen--Mou--Swiech \cite{GuillenMouSwiech2019}, following the work of Figalli--Gigli \cite{FigalliGigli2010}. Here, the explicit transport-derivative identity from Kang--Lim \cite{KangLim2025OptimalMarkovianCouplings} connects these intrinsic comparisons with the abstract stability criterion developed in Section \ref{chap3}. The condition also retains the signed drift contribution, allowing dissipativity to enter directly into the estimate.

\subsubsection*{Organization.}
The rest of this section is organized as follows. In the next subsection, we introduce L\'evy generators with finite second moment and provide an explicit formula for the \(2\)-transport derivative between two such generators. We then introduce L\'evy-type generators and give a sufficient criterion for the \(2\)-exponential stability condition in terms of their L\'evy triplets. In Section~\ref{sec6.4}, we impose the relevant Lipschitz hypothesis and apply the abstract results to establish quantitative propagation of chaos for L\'evy-type mean-field dynamics. Finally, in Section~\ref{sec6:averaged-amplitudes}, we consider triplets depending on averaged Hilbert-valued features, illustrate the construction with averaged drift and noise amplitudes, and derive the $O(N^{-1})$ refinement from Section~\ref{sec:hilbert-refinement}.

%Note that similar notions may be considered with $p\ge1$ in place of $2$.

%=============================================================================

\subsection{L\'evy generators of finite second moment and their 2-transport derivatives}

Throughout this section, the state space is \(\Pi=\mbr^d\), the \(d\)-dimensional Euclidean space equipped with the standard Euclidean distance \(d(x,y)=|x-y|\) and the squared cost
\[
c_2(x,y)=\frac{1}{2}d(x,y)^2=\frac{1}{2}|x-y|^2.
\]
We work in the framework of Sections \ref{chap4} and \ref{chap5} with \(p=2\). Accordingly, all notation—the space \(\mcp_2(\mbr^d)\) of probability measures with finite second moment, the Wasserstein distance \(\mcW_2\), the transport derivatives and their dual \(\om_2,\om_2^*\), the generator class \(\mcg_2^0(\mbr^d)\), the function space \(C_{2}(\mbr^d)\), and the exponential stability class \(\mathrm{Exp}_2\)—is understood in the quadratic setting. Throughout this section, we therefore fix \(p=2\).

\subsubsection{L\'evy generators on \texorpdfstring{$\mbr^d$}{R\^d}}
Throughout, we consider a \emph{L\'evy generator on \(\mbr^d\) of the global (untruncated) form}:
\begin{align}
	(\mA \phi)(x)
	&:= b \cdot \nabla \phi(x)
	+ \frac12 \tr[a D^2 \phi](x)
	+ \int_{\mbr^d \setminus \{0\}}
	\left[ \phi(x + y) - \phi(x) - y \cdot \nabla \phi(x) \right]
	\,d\Theta(y), \label{LK-rep2}
\end{align}
where
\begin{itemize}
	\item \(b\in\mbr^d\) is the drift vector;
	\item \(a\in\mcl{S}_d^{\ge0}(\mbr)\) is the diffusion matrix;
	\item \(\Theta\in\La_2(\mbr^d)\) is the L\'evy measure with finite second moment.
\end{itemize}
Here \(\mcl{S}_d^{\ge0}(\mbr)\) denotes the set of symmetric positive semidefinite
\(d\times d\) real matrices, and \(\La_2(\mbr^d)\) denotes the set of \emph{L\'evy
measures with finite second moment}, that is, measures \(\Theta\) on
\(\mbr^d\setminus\{0\}\) satisfying
\[
\int_{\mbr^d\setminus\{0\}} |y|^2\,d\Theta(y) < \infty,\qquad \Ta(\{0\})=0 .
\]
Throughout, the Greek letter \(\Lambda\) is reserved for any notation pertaining to ``\emph{L\'evy}''-related structures or quantities.
The triplet $(b,a,\Ta)$ is called the \emph{L\'{e}vy triplet} associated to $\mcA$. We denote by
\begin{align}\label{def:levy-triplet-space}
 \mathfrak L_d:=\mbr^d\times\mcl{S}_d^{\ge0}(\mbr)\times\La_2(\mbr^d)
\end{align}
the space of L\'evy triplets with finite second moment, with the drift understood under the global compensation convention in \eqref{LK-rep2}. The three operators of \eqref{LK-rep2} will be called \emph{drift}, \emph{diffusion}, and \emph{jump operators}, respectively.
We also denote by $\mcg_2^\La (\mbr^d)\subset \mcg(\mbr^d)$ the \emph{family of all L\'evy generators with finite second moment} defined above. 

We remark here that in the general setting, a L\'evy measure is a positive
measure on \(\mbr^d\) satisfying
\[
\int_{\mbr^d\setminus\{0\}}
\min\{|y|^2,1\}\,d\Theta(y)<\infty,
\qquad
\Theta(\{0\})=0.
\]
The associated L\'evy generator in the general ``truncated'' form has a drift
correction term, that is, the jump term $\mA^J\phi$ in \eqref{LK-rep2} is written as
\[
\int_{\mbr^d\setminus\{0\}}
\bigl[
\phi(x+y)-\phi(x)
-y\cdot\nabla\phi(x)\,
\mathbf 1_{\{|y|\le1\}}(y)
\bigr]
\,d\Theta(y),
\]
where the cutoff \(\mathbf 1_{\{|y|\le1\}}\) removes the linear compensation term for large jumps.
In the present case, since the L\'evy measure is assumed to have finite second moment, by absorbing
this term into the drift vector, we may rewrite the generator in the global form
\eqref{LK-rep2}.

Although this global form is less common in the literature because of its more restrictive integrability requirements, it proves particularly convenient for the study of couplings of processes and generators, as it provides a cleaner representation and simplifies several conditions, notably in the context of coupling and moment estimates (see, e.g., \cite{Kolokoltsov2009}). \emph{Throughout this section, we adopt this convention.}

L\'evy generators arise as generators of \emph{L\'{e}vy processes} on $\mbr^d$, that is, Markov processes on $\mbr^d$ that have stationary and independent increments. 
A well-known fact is: if $\cb{X_t}_{t\ge0}$ is a Feller process on $\mbr^d$ that has a finite second moment and is translation invariant, meaning that, for all $\Phi\in C_0(\mbr^d)$, $x,h\in\mbr^d$ and $t\ge0$,
\begin{align*}
	\mbe^x[\Phi(X_t+h)] = \mbe^{x+h}[\Phi(X_t)],\qquad \mbe^{0}[|X_1|^2]<\infty,
\end{align*}
then its generator is given by \eqref{LK-rep2} for some triplet $(b,a,\Ta)$ (see \cite[Chapter 2]{bottcher2013} or \cite[Theorem 6.8]{schilling2016introduction}). The converse also holds.  Moreover, in this case \(C_0^2(\mbr^d)\) is a core for \(\mA\), and the associated semigroup is continuous in the Wasserstein-\(2\) topology, so that \(\mA\in\mcg_2^0(\mbr^d)\).

\subsubsection{2-Transport derivatives between L\'evy generators}
Let us now address the \(2\)-transport derivative and its dual,
\(\om_2,\om_2^*\), between two L\'evy generators
\(\mA,\mB\in\mcg_2^\La(\mbr^d)\), both written in the global form
\eqref{LK-rep2}. By the main result of
\cite{KangLim2025OptimalMarkovianCouplings}, \(\om_2,\om_2^*\) admit an explicit
expression as a Wasserstein-type distance between the corresponding L\'evy
triplets. More precisely, the drift contribution is expressed through a Euclidean inner product, the diffusion matrices through the \emph{Bures--Wasserstein distance}, and
the L\'evy measures through the \emph{L\'evy--Wasserstein distance}. We now define these metrics.

\begin{definition}[Bures--Wasserstein and L\'evy--Wasserstein metrics]
	\label{def:levy-trans}
	\begin{enumerate}[(i)]
		\item Given two nonnegative semidefinite matrices \(a,\tilde a\in\mcl{S}_d^{\ge0}(\mbr)\), the
		\emph{Bures--Wasserstein distance} between them is defined by
		\begin{align}\label{def:BW-dist}
			\BWD(a, \tilde a)^2
			:= \frac{1}{2}
			\Big[
			\tr(a) + \tr(\tilde a)
			- 2\,\tr\bigl(
			(\tilde a^{1/2} a\, \tilde a^{1/2})^{1/2}
			\bigr)
			\Big].
		\end{align}
		
		\item Given two L\'evy measures \(\Theta, \tilde\Theta \in \Lambda_2(\mbr^d)\), the
		\emph{L\'evy--Wasserstein distance} between them is defined by
		\begin{align}\label{def:levy-cost}
			\WLA(\Theta, \tilde\Theta)^2
			:= \inf_{\Omega \in \Gamma_\Lambda(\Theta, \tilde\Theta)}
			\int_{\mbr^d \times \mbr^d}
			\frac{1}{2} |x - y|^2
			\,\Omega(dx, dy),
		\end{align}
		where the infimum is taken over all \emph{coupling L\'evy measures}
		\(\Omega \in \Gamma_\Lambda(\Theta, \tilde\Theta)\), that is, all L\'evy measures
		\(\Omega \in \Lambda_2(\mbr^{2d})\) satisfying the following marginal condition:
		for every Borel set \(E\subset\mbr^d\setminus\{0\}\),
		\[
		\Omega(E \times \mbr^d) = \Theta(E),
		\qquad
		\Omega(\mbr^d \times E) = \tilde\Theta(E).
		\]
		Equivalently, for every test function \(\phi \in C_b(\mbr^d)\) satisfying
		\(|\phi(x)| \le C|x|^2\) for some \(C \ge 0\),
		\[
		\int_{\mbr^d \times \mbr^d}
		\phi(x)\, \Omega(dx, dy)
		= \int_{\mbr^d} \phi(x)\, \Theta(dx),\quad \int_{\mbr^d \times \mbr^d}
		\phi(y)\, \Omega(dx, dy)
		= \int_{\mbr^d} \phi(y)\, \tilde\Theta(dy).
		\]
	\end{enumerate}
\end{definition}

The Bures--Wasserstein metric arises naturally as the Wasserstein-2 distance between Gaussian measures on \(\mbr^d\). More precisely, let \(\mu \sim \mathrm{Normal}(0, a)\) and \(\tilde \mu \sim \mathrm{Normal}(0, \tilde a)\) be two centered Gaussian measures with covariance matrices \(a, \tilde a \in \mcl{S}_d^{\ge0}(\mbr)\). A classical result of Givens and Shortt (see \cite[Proposition 7]{givens1984class}) states that
\[
\mcW_2(\mu, \tilde\mu) = \BWD(a, \tilde a).
\]

The L\'evy--Wasserstein metric \eqref{def:levy-cost} is the $p=2$ case of the transport metrics for L\'evy measures introduced by Guillen--Mou--Swiech \cite[Section 3]{GuillenMouSwiech2019} for $p\in[1,2]$, up to the normalization factor $1/2$ in the cost. Their construction builds on the boundary-reservoir formulation of Figalli--Gigli \cite{FigalliGigli2010}. It is structurally analogous to the classical Wasserstein-$2$ distance, but permits transport to and from the origin. The marginal constraints are imposed away from the origin, so coupling measures may charge the coordinate axes. The origin thus acts as a reservoir, allowing comparison of L\'evy measures with unequal or infinite total mass.

Both quantities define genuine metrics on their respective spaces. The Bures--Wasserstein distance is a metric on \(\mcl{S}_d^{\ge0}(\mbr)\), inherited from the fact that it is the Wasserstein-2 distance between centered Gaussian measures. The L\'evy--Wasserstein distance is a metric on \(\La_2(\mbr^d)\); its metric properties and the existence of optimal couplings are established in \cite[Section 3 and Appendix A]{GuillenMouSwiech2019}.

The following bounds compare these intrinsic transport distances with the Hilbert norms of prescribed noise amplitudes.
\begin{lemma}\label{lem:noise-amplitude-bounds}
\begin{enumerate}[(i)]
 \item\label{item:bures-amplitude} Let $\sigma,\tilde\sigma\in\mbr^{d\times m}$. Then
 \begin{align}\label{eq:bures-amplitude-bound}
  \mcW_{\mcS}(\sigma\sigma^\top,\tilde\sigma\tilde\sigma^\top)^2
  \le\frac12\|\sigma-\tilde\sigma\|_{\mcl F}^2.
 \end{align}
 \item\label{item:levy-amplitude} Let $(E,\Omega)$ be a standard Borel space equipped with a $\sigma$-finite measure, and let $\eta,\tilde\eta\in L^2(E,\Omega;\mbr^d)$. The pushforwards $\eta_\sharp\Omega$ and $\tilde\eta_\sharp\Omega$, restricted to $\mbr^d\setminus\{0\}$, are L\'evy measures with finite second moment, and
 \begin{align}\label{eq:levy-amplitude-bound}
  \mcW_{\La}(\eta_\sharp\Omega,\tilde\eta_\sharp\Omega)^2
  \le\frac12\int_E|\eta(u)-\tilde\eta(u)|^2\,\Omega(du).
 \end{align}
\end{enumerate}
\end{lemma}
\begin{proof}
(i) Let $\gamma_m$ be the standard Gaussian measure on $\mbr^m$. The pushforward of $\gamma_m$ under $z\mapsto(\sigma z,\tilde\sigma z)$ couples the centered Gaussian measures with covariances $\sigma\sigma^\top$ and $\tilde\sigma\tilde\sigma^\top$. By the Gaussian interpretation of the Bures--Wasserstein distance,
\[
 \mcW_{\mcS}(\sigma\sigma^\top,\tilde\sigma\tilde\sigma^\top)^2
 \le\frac12\int_{\mbr^m}|(\sigma-\tilde\sigma)z|^2\,\gamma_m(dz)
 =\frac12\|\sigma-\tilde\sigma\|_{\mcl F}^2.
\]
(ii) Discarding mass mapped to zero, the pushforward satisfies
\[
 \int_{\mbr^d\setminus\{0\}}|v|^2\,(\eta_\sharp\Omega)(dv)
 =\int_E|\eta(u)|^2\,\Omega(du)<\infty.
\]
Hence $\eta_\sharp\Omega\in\La_2(\mbr^d)$, and similarly for $\tilde\eta$. Define $\Upsilon=(\eta,\tilde\eta)_\sharp\Omega$ on $\mbr^{2d}\setminus\{(0,0)\}$.
Its second moment is $\int_E(|\eta(u)|^2+|\tilde\eta(u)|^2)\,\Omega(du)<\infty$. For every Borel set $B\subset\mbr^d\setminus\{0\}$, the pushforward identity gives
\[
 \Upsilon(B\times\mbr^d)=(\eta_\sharp\Omega)(B),\qquad
 \Upsilon(\mbr^d\times B)=(\tilde\eta_\sharp\Omega)(B).
\]
Thus $\Upsilon\in\Gamma_\Lambda(\eta_\sharp\Omega,\tilde\eta_\sharp\Omega)$ is an admissible L\'evy coupling. Evaluating its cost yields
\[
 \mcW_{\La}(\eta_\sharp\Omega,\tilde\eta_\sharp\Omega)^2
 \le\frac12\int_{\mbr^{2d}}|v-w|^2\,\Upsilon(dv,dw)
 =\frac12\int_E|\eta(u)-\tilde\eta(u)|^2\,\Omega(du),
\]
which proves \eqref{eq:levy-amplitude-bound}.
\end{proof}

The following proposition, which is the main result of \cite{KangLim2025OptimalMarkovianCouplings}, expresses the \(2\)-transport derivative between two L\'evy generators directly in terms of the distances
between their L\'evy triplets. 
\begin{proposition}\label{prop:levy-bdd}
	Let $\mA,\mB\in\mcg_2^\La(\mbr^d)$ be two L\'evy generators with L\'evy triplets $(b,a,\Ta),(\tilde b,\tilde a,\tilde \Theta)$. 
	The 2-transport derivative between $\mA,\mB$ is given by
	\begin{align*}
		\om_2^*(x,y;\mA,\mB)=\om_2(x,y;\mA,\mB)= (b-\tilde b)^\top(x-y)+ \mcW_{\mcS}(a,\tilde a)^2 +\mcW_{\La}(\Ta,\tilde \Ta)^2.
	\end{align*}
\end{proposition}

Let us briefly comment on this identity. In fact, the version with equality replaced by \(\le\) is straightforward to prove by a suitable coupling of the two L\'evy processes; see also Lemma \ref{lem:weak-dual}. The main contribution of \cite{KangLim2025OptimalMarkovianCouplings} is the reverse inequality, showing that the dual transport derivative \(\om_2^*\) dominates the right-hand side. This reverse bound is the nontrivial part of the proposition.

%=============================================================================
\subsection{L\'{e}vy-type generators on \texorpdfstring{$\mbr^d$}{R\^d} and their \texorpdfstring{$\mathrm{Exp}_2$}{Exp2}-class}\label{sec6.3}
A \emph{L\'evy-type generator} is a generator of the form \eqref{LK-rep2}, except that the triplet may depend on the position $x\in\mbr^d$. Specifically, we consider the L\'evy-type generator in the \emph{global} form:
\begin{align}
	(\mA\phi)(x)
	&= b(x) \cdot \nabla \phi(x) + \frac{1}{2} \tr\big(a(x) D^2 \phi(x)\big) \nonumber\\
	&\qquad  + \int_{\mbr^d \setminus \{0\}} \left[\phi(x + y) - \phi(x) - y \cdot \nabla \phi(x) \right] \Ta(x, dy). \label{eq:levy-operator}
\end{align}
Here the coefficients form a spatially dependent triplet field \((b,a,\Theta):\mbr^d\to\mathfrak{L}_d\), where $b$ is the drift vector field, $a$ is a measurable field of positive semidefinite matrices, and $\Theta(x)$ is a L\'evy measure with finite second moment for each $x$.

Alternatively, we may regard a L\'evy-type generator (with finite second moment) as a generator on $\mbr^d$ obtained by ``gluing'' together a family of L\'evy generators $\{\mA(x)\}_{x\in\mbr^d}\subset \mcg_2^\La(\mbr^d)$ that vary with position $x\in\mbr^d$, or simply a map 
\begin{align*}
	\mA: \mbr^d \to \GLVR,
\end{align*}
The operator \(\mA\) acting on functions $\phi\in C_c^2(\mbr^d)$ is then defined pointwise by
\begin{align*}
	(\mA \phi)(x) := (\mA(x)\phi)(x).
\end{align*}
In particular, in the expression above, $\mA(x)$ is the L\'evy generator in the global form with L\'evy triplet $(b(x),a(x),\Ta(x))$.

The class of L\'{e}vy-type operators is extensively studied in the theory of Markov processes on $\mbr^d$. The well-known theorem due to Courr\`{e}ge states that if $(\mcA,D(\mcA))$ is the generator of a Feller process on $\mbr^d$ such that $C_c^\infty(\mbr^d)\subset D(\mcA)$, then $\mcA$ must take the form of \eqref{eq:levy-operator} (in the truncated form) for some L\'{e}vy triplet $\scb{(b(x),a(x),\Ta(x))}_{x\in\mbr^d}$ (see \cite[Theorem 2.21]{bottcher2013}). However, the necessary condition for the triplets to guarantee that $\mcA$ is a Feller generator, to the best of our knowledge, is not known. We point the reader to \cite[Chapter 3]{bottcher2013} for a detailed discussion of sufficient conditions, which are difficult to state concisely here.

Throughout this section, we always assume that the L\'{e}vy triplet \(\{(b(x),a(x),\Theta(x))\}_{x\in\mbr^d}\) is such that (i) the associated generator \(\mcA\) belongs to \(\GCZR\) (see also Proposition \ref{prop:characterization-Gp0} for its equivalent characterization), and (ii) \(C_c^2(\mbr^d)\) is contained in \(D(\mcA)\) and is a core for \(\mcA\). It is expected that these properties can be established under suitable Lipschitz continuity of the triplet with respect to the Euclidean, Bures--Wasserstein, and L\'evy--Wasserstein metrics. We do not pursue this question here and leave it to future work.

\subsubsection{2-Exponential stability between L\'evy-type generators}
Let us now specialize the 2-exponential stability condition of Definition \ref{def:c-exp-stab} to L\'evy-type generators. The following proposition gives a sufficient condition on the L\'evy triplets for this stability condition, using Corollary \ref{cor3.20:chap3-main-cor}.

\begin{proposition}\label{prop:levy-type-bdd}
	Let \(\mA=\{\mA(x)\}_{x\in\mbr^d}\) and
	\(\tilde\mA=\{\tilde\mA(y)\}_{y\in\mbr^d}\) be two L\'evy-type generators with L\'evy triplets
	\(\{(b(x),a(x),\Theta(x))\}_{x\in\mbr^d}\) and
	\(\{(\tilde b(y),\tilde a(y),\tilde\Theta(y))\}_{y\in\mbr^d}\), respectively.
	Let \(\al,\be\in\mbr_+\). Consider the following conditions:
	\begin{enumerate}[(i)]
		\item for all \(x,y\in\mbr^d\),
		\[
		(x-y)\cdot\bigl(b(x)-\tilde b(y)\bigr)
		+\mcW_{\mcS}\bigl(a(x),\tilde a(y)\bigr)^2
		+\mcW_{\La}\bigl(\Theta(x),\tilde\Theta(y)\bigr)^2
		\le \al+\frac{\be}{2}|x-y|^2;
		\]
		\item \(\mA,\tilde\mA\) satisfy the \(\Expt(\al,\be)\)-condition, namely, for all
		\(x,y\in\mbr^d\),
		\[
		\om_2(x,y;\mA,\tilde\mA)
		\le \al+\frac{\be}{2}|x-y|^2.
		\]
	\end{enumerate}
    Then (i) implies (ii).
\end{proposition}

The proof is a direct consequence of the dual characterization of the transport derivative. By duality, it suffices to test the exponential stability condition against pairs \((\phi,\psi)\in C_c^2(\mbr^d)^2\) such that \(\phi\oplus\psi-c_2\) attains a global maximum at a given point. For such pairs, the action of the L\'evy-type generators reduces to the evaluation of the pointwise generators \(\mA(x)\) and \(\tilde\mA(y)\). Proposition \ref{prop:levy-bdd} bounds their joint action by the explicit triplet expression in (i). In this way, the dual formulation plays an essential role: it transforms a semigroup-level comparison into a static pointwise condition on the coefficients.

\begin{proof}
	Assume (i) and fix \((x,y)\in\mbr^{2d}\). By Proposition
	\ref{prop:levy-bdd}, (i) is equivalent to
	\begin{align}\label{temp:pg78}
		\om_2^*(x,y;\mA(x),\tilde\mA(y))
		\le \al+\frac{\be}{2}|x-y|^2.
	\end{align}
	Let \(\phi,\psi\in C_c^2(\mbr^d)\) be such that \(\phi\oplus\psi-c_2\) attains a global maximum at \((x,y)\). Then
	\[
	\mA\phi(x)+\tilde\mA\psi(y)
	=\mA(x)\phi(x)+\tilde\mA(y)\psi(y)
	\le \om_2^*(x,y;\mA(x),\tilde\mA(y))
	\le \al+\frac{\be}{2}|x-y|^2.
	\]
	Since this holds for all such pairs \((\phi,\psi)\), and \(C_c^2(\mbr^d)\) is a core
	for both \(\mA\) and \(\tilde\mA\), Corollary \ref{cor:general-duality-theta}
	yields (ii).
	
\end{proof}

\begin{remark}
Only the sufficient implication is needed in the applications below. The converse would follow from the fixed-contact L\'evy duality identity with the supremum restricted to $C_c^2(\mbr^d)$. Establishing this restricted-test-function formulation requires an additional approximation argument, which we do not pursue here.
\end{remark}

%=============================================================================

\subsection{L\'evy-type mean-field dynamics and propagation of chaos}\label{sec6.4}
By a \emph{L\'evy-type mean-field generator of finite second moment}, we mean a mean-field generator
\(\mathcal{A} : \mcp_2(\mbr^d) \to \GCZR\) on \(\mbr^d\) such that, for every
\(\mu \in \mcp_2(\mbr^d)\), the operator \(\mathcal{A}(\mu)\) is a L\'evy-type generator.
Thus, for each \(\mu\), the generator \(\mathcal{A}(\mu)\) admits the global
L\'evy--Khintchine decomposition
\begin{align}
	[\mathcal{A}(\mu)\phi](x)
	&= b(x,\mu)\cdot\nabla\phi(x)
	+ \frac12 \tr\bigl[a(x,\mu)D^2\phi\bigr](x) \nonumber \\
	&\qquad
	+ \int_{\mbr^{d}\setminus\{0\}}
	\bigl[\phi(x+y)-\phi(x)-y\cdot\nabla\phi(x)\bigr]
	\,\Theta(x,dy,\mu), \label{eq:levy-mean-field-generator}
\end{align}
where the coefficients form a measure-dependent L\'evy triplet field
\[
 (b,a,\Theta):\mbr^d\times\mcp_2(\mbr^d)\longrightarrow\mathfrak{L}_d.
\]

\subsubsection{Examples of L\'evy-type mean-field dynamics}\label{sec6:examples}
Most mean-field models in the literature are formulated on \(\mbr^d\) or on the lattice \(\mbz^d\). Under mild regularity assumptions, their generators fall into the L\'evy-type framework (or its variation) considered in this section. On \(\mbr^d\), this is a consequence of the Courr\`ege representation theorem for Feller generators satisfying the positive maximum principle. On \(\mbz^d\), the generators are pure-jump operators whose jump measures are supported on the lattice. Models on lattice or finite-type state spaces can be treated within the abstract framework on their natural state spaces; applying the Euclidean results below additionally requires a suitable extension of their generators satisfying the stated assumptions. We now present several classes of L\'evy-type mean-field systems that arise naturally in the literature and fit into the abstract framework developed in this section.

For convenience, finite-intensity jump terms in the examples below are written in uncompensated form. Whenever $\int|u|\,\Theta(x,du,\mu)<\infty$, their conversion to \eqref{eq:levy-mean-field-generator} is understood to include the drift correction $b=\tilde b+\int u\,\Theta(x,du,\mu)$, where $\tilde b$ is the drift in the uncompensated representation.

\begin{example}[McKean--Vlasov dynamics driven by Brownian motions or L\'evy jumps]\label{ex:mv-dynamics}
	The first class of examples is the classical McKean--Vlasov dynamics described by the system of SDEs
	\[
	dX_t^i
	= b\bigl(X_t^i, \mu_{\BX_t}\bigr)\,dt
	+ \sigma\bigl(X_{t-}^i, \mu_{\BX_t}\bigr)\,dZ_t^i,
	\qquad t\in[0,T],
	\quad 1\le i\le N,
	\]
	where \(\{Z_t^i\}_{t\ge0}\) are independent Brownian motions \(\{B_t^i\}_{t\ge0}\) or pure-jump L\'evy processes. As discussed at the beginning of this section, the Brownian case corresponds to the L\'evy-type generator with diffusion matrix \(a=\sigma\sigma^\top\) and jump kernel \(\Theta\equiv0\), while the pure-jump case corresponds to \(a\equiv0\) and \(\Theta(x,\mu)=\sigma(x,\mu)_\sharp\Omega\), where \(\Omega\) is the L\'evy measure of \(Z^i\). We refer the reader to the discussion following equation \eqref{eq:Cav-model}.
	
	Beyond the first-order dynamics, a mean-field model can also be set in a \emph{second-order Langevin-type} framework \cite{carrillo2009double, cattiaux2018stochastic, Chaintron_2022_1, degond2015phase, dorsogna2006self, monmarche2017long, schuh2024global}: 
	\[
	\begin{cases}
		dX_t^i = V_t^i\,dt,\\[2mm]
		dV_t^i = b\bigl(X_t^i,V_t^i,\mu_{\BY_t}\bigr)\,dt
		+ \sigma\bigl(X_{t-}^i,V_{t-}^i,\mu_{\BY_t}\bigr)\,dB_t^i,
	\end{cases}
	\qquad 1\le i\le N,
	\]
	where \(\BY_t=(X_t,V_t)\in\mbr^d\times\mbr^d\) denotes the full state. This system remains within the McKean--Vlasov regime, but its generator is now degenerate and hypoelliptic. Written in the L\'evy-type form \eqref{eq:levy-mean-field-generator} on the enlarged state space \(\mbr^{2d}\), the generator acts on \(\phi=\phi(x,v)\) as
	\[
	[\mA(\mu)\phi](x,v)
	= v\cdot\nabla_x\phi(x,v)
	+ b(x,v,\mu)\cdot\nabla_v\phi(x,v)
	+ \frac12 \tr\bigl[\sigma\sigma^\top D_v^2\phi\bigr](x,v).
	\]
	Here the drift vector and diffusion matrix on \(\mbr^{2d}\) are respectively given by
	\[
	\tilde b(x,v,\mu)
	=
	\begin{pmatrix}
		v\\
		b(x,v,\mu)
	\end{pmatrix}
	\in\mbr^{2d},
	\qquad
	\tilde a(x,v,\mu)
	=
	\begin{pmatrix}
		0 & 0\\
		0 & (\sigma\sigma^\top)(x,v,\mu)
	\end{pmatrix}
	\in\mcl{S}_{2d}^{\ge0}(\mbr).
	\]
	Thus, despite the degeneracy of the diffusion in the \(x\)-direction, the model still fits into the L\'evy-type mean-field framework \eqref{eq:levy-mean-field-generator}. We also note that the second-order dynamics can be driven by L\'evy jumps instead of Brownian motions, simply by replacing \(B_t^i\) with a pure-jump L\'evy process. This leads to an additional jump term of the same form as in \eqref{eq:levy-mean-field-generator}, although such second-order jump-driven models have received less attention in the literature.
\end{example}

\begin{example}[Mean-field piecewise deterministic Markov processes]
	A \emph{mean-field piecewise deterministic Markov process (PDMP)} on \(\mbr^d\) (see \cite[Section 2.2.3]{Chaintron_2022_1} and \cite{delmoral2017stochastic, diez2020propagation, fournier2016toy, Villani2002}) combines deterministic motion with mean-field-dependent jumps. Between jumps, each particle evolves according to a deterministic flow \(\dot x = \tilde b(x)\), where the drift \(\tilde b\) is independent of the empirical measure. At random jump times, the particle state is updated according to a jump kernel that may depend on the current position and on the empirical distribution of the whole population. Thus, the interaction between particles occurs only through the jump mechanism, while the deterministic part remains non-interacting.
	
	In the uncompensated representation, this corresponds to taking \(a\equiv0\) and a law-independent deterministic drift \(\tilde b(x)\). The jump part is given by
	\[
	\int_{\mbr^d}
	\bigl[\phi(x+y)-\phi(x)\bigr]
	\,\Theta(x,dy,\mu),
	\]
	where \(\Theta(x,\cdot,\mu)\) is a finite positive measure with finite second moment for each \((x,\mu)\), encoding the jump intensity and jump distribution. Conversion to the global form \eqref{eq:levy-mean-field-generator} follows the compensation convention stated above; the resulting compensated drift may depend on \(\mu\).
\end{example}

\begin{example}[Mean-field collision models]\label{ex:collision}
	A \emph{mean-field collision model} \cite{fournier2016nanbu,graham1997stochastic,meleard1996asymptotic} is a second-order model on the position--velocity state space \((x,v)\in\mbr^d\times\mbr^d\), describing particles whose positions evolve by free transport while their velocities are updated at random collision times through mean-field interaction. It is a mean-field approximation of classical \emph{binary collision dynamics}, commonly referred to as a \emph{Nanbu-type approximation}. In the binary Kac model, collisions involve two particles whose velocities are updated simultaneously; in the mean-field version, a tagged particle collides with a virtual partner sampled from the current one-particle law \(\mu\), so that only the velocity of the tagged particle is modified.
	
	The mean-field generator takes the form
	\[
	[\mA(\mu)\phi](x,v)
	=
	v\cdot\nabla_x\phi(x,v)
	+
	\int_{\mbr^{2d}}
	\int_{\mbr^d}
	\bigl[\phi(x,v')-\phi(x,v)\bigr]
	\,k\bigl(x,v;x_*,v_*;dv'\bigr)
	\,\mu(dx_*,dv_*),
	\]
	where the collision kernel \(k\bigl(x,v;x_*,v_*;dv'\bigr)\) encodes both the jump intensity and the post-collision velocity distribution, given a virtual partner of position \(x_*\) and velocity \(v_*\). Thus, the mean-field interaction enters through the sampling of the partner from the current one-particle law \(\mu\). In concrete physical models, the collision kernel \(k\) is often assumed to have a specific structural form encoding microscopic reversibility, conservation of momentum and energy, and detailed balance. We do not spell out these assumptions here; for the relevant derivations and physical assumptions, we refer the reader to \cite{Villani2002}.
	
	This model is a special case of the mean-field PDMP described in the previous example: the drift is \(b(x,v)=(v,0)\), there is no diffusion, and the jumps act only on the velocity variable. In the spatially homogeneous or well-mixed case, the collision kernel is independent of \(x\) and \(x_*\), and the model reduces to a pure mean-field jump process on the velocity space.
\end{example}

\begin{example}[Mean-field integer-valued processes]\label{ex:integer}
	We identify \(\mbz\) with its image in \(\mbr\), so that a mean-field Markov process taking values in \(\mbz\) can be viewed as a L\'evy-type mean-field system whose jump kernel is supported on \(\mbz\).
	
	The simplest example is the \emph{mean-field birth-death process} \cite{Thai2015}. Each particle jumps one step to the right or to the left at rates depending on its current position and on the mean-field distribution \(\mu\). The corresponding generator acts on \(\phi:\mbz\to\mbr\) as
	\[
	[\mA(\mu)\phi](n)
	= \la_b(n,\mu)\bigl[\phi(n+1)-\phi(n)\bigr]
	+ \la_d(n,\mu)\bigl[\phi(n-1)-\phi(n)\bigr],
	\]
	where \(\la_b(n,\mu)\ge0\) and \(\la_d(n,\mu)\ge0\) are the birth and death rates, respectively. This is a pure-jump model with \(a\equiv0\) and jump kernel
	\[
	\Theta(n,dy;\mu)
	= \la_b(n,\mu)\delta_{+1}(dy)
	+ \la_d(n,\mu)\delta_{-1}(dy).
	\]
	
	More generally, a \emph{mean-field continuous-time Markov chain} on \(\mbz\) \cite{BorkarSundaresan2012,CohenHuffman2026,DAWSON1991293} has generator
	\[
	[\mA(\mu)\phi](n)
	= \sum_{k\in\mbz\setminus\{0\}}
	\lambda(n,n+k;\mu)
	\bigl[\phi(n+k)-\phi(n)\bigr],
	\]
	where \(\lambda(n,m;\mu)\ge0\) is the transition rate from \(n\) to \(m\). Here the jump kernel is supported on the set of possible jumps and is given by
	\[
	\Theta(n,\{k\};\mu)
	= \lambda(n,n+k;\mu),
	\qquad k\in\mbz\setminus\{0\}.
	\]
	Under a suitable finite-second-moment condition, the jump kernel \(\Theta(n,\cdot;\mu)\) lies in \(\La_2(\mbr)\).
\end{example}

\begin{example}[Mean-field reaction--diffusion models]\label{ex:reaction-diffusion}
	A \emph{mean-field reaction--diffusion model} \cite{IsaacsonMaSpiliopoulos2022, lim2020quan, Oelschlager1989} describes a population of particles on \(\mbr^d\), each carrying an internal type or chemical state \(k\in S=\{1,\dots,m\}\). Between type changes, a particle of type \(k\) moves according to a diffusion process whose drift and covariance may depend on its position, its type, and the empirical distribution of the full type--position population. At random times, the particle changes type as a result of a reaction or interaction with the mean field. The reaction rate and the transition probabilities between types may depend on the current position, the current type, and the empirical measure.
	
	The mean-field generator acts on functions \(\phi=\phi(x,k)\), \((x,k)\in \mbr^d\times S\), as
	\begin{align*}
		[\mA(\mu)\phi](x,k)
		&=
		b(x,k,\mu)\cdot\nabla_x\phi(x,k)
		+
		\frac12 \tr\bigl[a(x,k,\mu)D_x^2\phi\bigr](x,k)\\
		&\qquad +\sum_{\ell\neq k}
		r_{k\to\ell}(x,\mu)
		\bigl[\phi(x,\ell)-\phi(x,k)\bigr],
	\end{align*}
	where \(b(x,k,\mu)\) is the drift vector and \(a(x,k,\mu)\) is the diffusion matrix of a particle of type \(k\), while \(r_{k\to\ell}(x,\mu)\ge0\) is the mean-field-dependent rate of switching from type \(k\) to type \(\ell\).
	
	As before, we may view \(S\) as a subset of \(\mbr\) and identify the full state with \((x,k)\in\mbr^{d+1}\). The reaction--diffusion generator above fits into the L\'evy-type framework under the compensation convention stated at the beginning of this subsection. In the displayed uncompensated representation, the continuous motion acts on the spatial component \(x\), while jumps change the type \(k\mapsto\ell\), leaving \(x\) unchanged.
\end{example}

\subsubsection{Quantitative propagation of chaos for L\'evy-type mean-field dynamics}
We now apply the main results of Sections \ref{chap4} and \ref{chap5} to L\'evy-type mean-field dynamics. In particular, under the following hypothesis, we establish quantitative propagation of chaos for this class of models.

\begin{hypothesisLa}\label{hypoLa}
	Let \(\{\mathcal{A}(\mu)\}_{\mu\in\mcp_2(\mbr^d)}\) be a L\'evy-type mean-field generator with L\'evy triplet \((b,a,\Theta):\mbr^d\times\mcp_2(\mbr^d)\to\mathfrak{L}_d\). Assume the following.
	\begin{enumerate}[(i)]
		\item For every \(\mu'\in\mcp_2(\mbr^d)\), the generator \(\mathcal{A}(\mu')\) belongs to \(\mcg_2^0(\mbr^d)\).
		\item For every \(\mu'\in\mcp_2(\mbr^d)\), the space \(C_c^2(\mbr^d)\) is contained in \(D(\mathcal{A}(\mu'))\) and is a core for \(\mathcal{A}(\mu')\).
		\item There exist constants \(\al,\be\ge0\) such that for all \(x,y\in\mbr^d\) and \(\mu',\nu'\in\mcp_2(\mbr^d)\),
		\begin{align*}
			&(x-y)\cdot\bigl(b(x,\mu')-b(y,\nu')\bigr)
			+ \mcW_{\mcS}\bigl(a(x,\mu'),a(y,\nu')\bigr)^2
			+ \mcW_{\La}\bigl(\Theta(x,\cdot,\mu'),\Theta(y,\cdot,\nu')\bigr)^2\\
			&\qquad \le \al \mcW_2(\mu',\nu')^2
			+ \frac{\be}{2}|x-y|^2.
		\end{align*}
	\end{enumerate}
\end{hypothesisLa}

\begin{remark}
	Condition (iii) may be replaced by the following three separate bounds:
	\begin{align*}
		(x-y)\cdot\bigl(b(x,\mu')-b(y,\nu')\bigr)
		&\le \al_1 \mcW_2(\mu',\nu')^2
		+ \frac{\be_1}{2}|x-y|^2,\\
		\mcW_{\mcS}\bigl(a(x,\mu'),a(y,\nu')\bigr)^2
		&\le \al_2 \mcW_2(\mu',\nu')^2
		+ \frac{\be_2}{2}|x-y|^2,\\
		\mcW_{\La}\bigl(\Theta(x,\cdot,\mu'),\Theta(y,\cdot,\nu')\bigr)^2
		&\le \al_3 \mcW_2(\mu',\nu')^2
		+ \frac{\be_3}{2}|x-y|^2.
	\end{align*}
	The first inequality is a \emph{one-sided Lipschitz condition} in the spatial variable, which is the natural condition for stability and uniqueness of the underlying dynamical system. A sufficient condition is that \(b\) be Lipschitz continuous in both \(x\) and \(\mu'\). The second and third inequalities express Lipschitz continuity of \(a\) and \(\Theta\) from the product space \(\mbr^d\times\mcp_2(\mbr^d)\), equipped with the Euclidean--Wasserstein metric, into the Bures--Wasserstein and L\'evy--Wasserstein metrics, respectively. Summing the three bounds recovers Condition (iii) with \(\al=\al_1+\al_2+\al_3\) and \(\be=\be_1+\be_2+\be_3\).
\end{remark}

To apply Theorems \ref{thm4.3:chap4-main-result} and \ref{thm:chap5-main-thm}, we need the following lemma.

\begin{lemma}\label{lem:chap6-small-corollary}
	Let \(\{\mathcal{A}(\mu)\}_{\mu\in\mcp_2(\mbr^d)}\) be a L\'evy-type mean-field
	generator satisfying Hypothesis \refLa. Then the mean-field generator
	satisfies Hypothesis {\refA} with the same constants \(\al,\be\) appearing in
	Hypothesis \refLa(iii).
\end{lemma}

\begin{proof}
	Apply Proposition \ref{prop:levy-type-bdd} with \(\mA=\mathcal{A}(\mu')\) and \(\tilde\mA=\mathcal{A}(\nu')\). By that proposition, Condition (iii) of Hypothesis \refLa\, implies
	\[
	\om_2(x,y;\mathcal{A}(\mu'),\mathcal{A}(\nu'))
	\le \al \mcW_2(\mu',\nu')^2
	+ \frac{\be}{2}|x-y|^2.
	\]
	This is exactly the bound required in Hypothesis {\refA}, with the same constants \(\al\) and \(\be\).
\end{proof}

Now we combine all the results from Sections \ref{chap4} and  \ref{chap5} to obtain the exponential estimate for $\mcW_2$-distance between $\bs\rho_t=\bs\rho_0 e^{t\bs{\hat\mcA}}$, a Markov flow under the $N$-particle generator $\bs{\hat\mcA}$ associated to a L\'{e}vy-type mean-field generator $\mcA$, and $\bs\br_t= \br_t^{\otimes N}$, the tensor product of the associated mean-field solution $\br_t$.

\begin{theorem}\label{thm:levy-poc-aleph}
	Let $\{\mA(\mu)\}_{\mu\in\mcp_2(\mbr^d)}$ be a L\'{e}vy-type mean-field generator satisfying Hypothesis \refLa\, with constants $\al,\be\ge 0$. Given $N\ge2$, let $\BHA_N$ be the $N$-particle generator associated to $\mcA$ (see Definition \ref{def:N-particle-generator}) and assume that Hypothesis {\refAp} holds.
	\begin{enumerate}[(i)]
		\item For every $\br_0\in\PCR$, the associated mean-field evolution problem
		\begin{align}\label{eq6.16:temp}
			\left\{\begin{aligned}
				\partial_t \br_t &= \br_t \mcA(\br_t),\quad t\in(0,T),\\
				\br_0 &\in \PCR,
			\end{aligned}\right.
		\end{align}
		admits a unique $\mcW_2$-stable solution.
		\item Let $\scb{\br_t}_{t}\in C([0,\infty);\PCR)$ be a $\mcW_2$-stable solution of the mean-field problem \eqref{eq6.16:temp}, and $\bs\rho_0\in\PCRN$. We denote
		\begin{align*}
			\bs{\rho}_t = \bs{\rho}_0 e^{t\bs{\hat{\mcA}}_N}, \quad \bs{\br}_t = \br_t^{\otimes N}.
		\end{align*}
		Then for any $T\ge0$, it holds for some constants $C,K>0$ depending on $\al,\be$ that
		\begin{align*}
			\sup_{t\in[0,T]} \bs{\mcW}_2^2\mrb{\bs{\rho}_t,\bs{\br}_t} \le \bs{\mcW}_2^2\mrb{\bs{\rho}_0,\bs{\br}_0} e^{K T} + C \zeta_{K}(T) \sup_{t\in[0,T]} \aleph_N(\br_t),
		\end{align*}
		where $\aleph_{N}$ is given in Definition \ref{def:aleph-N}.
		% \begin{align*}
			%     \zeta_{\Ta}(T) :=
			%     \begin{cases}
				%         \displaystyle \frac{e^{\Ta T}-1}{\Ta}, &\text{if }\Ta>0,\\
				%         T, &\text{if }\Ta=0.
				%     \end{cases}
			% \end{align*}
	\end{enumerate}
\end{theorem}
\begin{proof}
	This follows from Lemma \ref{lem:chap6-small-corollary} and Theorems \ref{thm4.3:chap4-main-result}, \ref{thm:chap5-main-thm}.
\end{proof}

%=============================================================================
\subsubsection{Estimate for \texorpdfstring{$\aleph_N(\br)$}{aleph\_N} and pointwise propagation of chaos}
By Theorem \ref{thm:levy-poc-aleph}(ii), propagation of chaos follows if
\begin{align*}
	\lim_{N\to\infty}\sup_{t\in[0,T]}\aleph_N(\br_t)=0.
\end{align*}
where $\aleph_N$ is given by (recalling $\bs{x}=(x_1,\cdots, x_N)\in(\mbr^d)^N)$:
\begin{align*}
	\aleph_N(\br) = \int_{(\mbr^d)^N} \mcW_2^2(\mu(\Bx_1'),\br)\,d\br^{\oN}(\Bx),\qquad \mu(\bs{x}_1')=\f{1}{N-1}\sum_{k\ge 2}\de_{x_k}.
\end{align*}
From a probabilistic point of view, we have
\begin{align*}
	\aleph_N(\br) = \mbe\msqb{\mcC_2(\mu(\BX_{N-1}),\br)},
\end{align*}
where $\BX_{N-1}=(X_2,\cdots,X_N)$ is the vector of i.i.d. $\Pi$-valued random variables with the common law $X_k\sim\br$. Fournier and Guillin \cite{fournier2015rate} obtained the following estimate. Their result concerns the Wasserstein-$p$ distance $\mcW_p$ for $p\ge1$; we state only the case $p=2$.

\begin{theorem}[\cite{fournier2015rate}]\label{thm:fournier-convergence-rate}
	Let $\br\in\mcP(\mbr^d)$. Assume that $M_q(\br):=\int_{\mbr^d} |x|^q\,d\br(x)<\infty$ for some $q>2$. Then there exists a constant $L$ depending only on $d$ and $q$ such that for all $N\ge2$, it holds $\aleph_N(\br)\le \ep_{d,q}(N-1)$, where
	\begin{align*}
		\ep_{d,q}(N) =
		\begin{cases}
			LM^{2/q}_q(\br)\rb{N^{-1/2}+N^{-(q-2)/q}}, &\text{if }d<4\text{ and }q\neq4,\\
			LM^{2/q}_q(\br)\rb{N^{-1/2}\log(1+N)+N^{-(q-2)/q}}, &\text{if }d=4\text{ and }q\neq4,\\
			LM^{2/q}_q(\br)\rb{N^{-2/d}+N^{-(q-2)/q}}, &\text{if }d>4\text{ and }q\neq 2d/(d-2).
		\end{cases}
	\end{align*}
\end{theorem}

The estimate above applies for $q>2$, with $q\neq4$ when $d\le4$ and $q\neq 2d/(d-2)$ when $d>4$. We now obtain quantitative propagation of chaos for L\'evy-type mean-field systems in these noncritical cases.
\begin{theorem}\label{thm:levy-poc-fournier}
	Assume the settings of Theorem \ref{thm:levy-poc-aleph}. For any $T\ge0$, suppose that the solution $\scb{\br_t}_{t\ge0}$ of the mean-field equation has a finite $q$-th moment for some fixed $q>2$ satisfying the restrictions stated above:
	\begin{align*}
		C_T := \sup_{t\in[0,T]} \int_{\mbr^d} |x|^q\,d\br_t(x) <\infty.
	\end{align*}
	If $\Brho_0=\Bbr_0$, then it holds for some constants $C,K>0$ depending on the constants in Hypothesis {\refLa} that 
	\begin{align*}
		\bs{\mcW}_2^2(\Brho_t,\Bbr_t) \le C\zeta_{K}(T)\ep_{d,q}(N-1),\quad t\in[0,T],
	\end{align*}
	where $\ep_{d,q}(N)$ is given by Theorem \ref{thm:fournier-convergence-rate}, with $M_q=C_T$. In particular, pointwise propagation of chaos holds for the $N$-particle system generated by $\BHA_N$ as $N\to\infty$.
\end{theorem}

%=============================================================================
\subsection{\texorpdfstring{An $O(N^{-1})$ estimate under averaged Hilbert-space dependence}{An O(1/N) estimate under averaged Hilbert-space dependence}}\label{sec6:averaged-amplitudes}

The preceding estimates depend on the empirical transport error $\aleph_N$. In this final subsection, we impose additional structure on the dependence of the L\'evy triplet on the population law and apply Theorem~\ref{thm:hilbert-poc} to obtain an $O(N^{-1})$ sampling contribution. Dependence on averaged features is familiar in the McKean--Vlasov literature; for example, \cite[Section 2.2.2, Example 1]{Chaintron_2022_1} considers drift and diffusion amplitudes that are nonlinear functions of averaged finite-dimensional features. Here we extend this construction to Hilbert-valued features and L\'evy-type dynamics.

Let $\mathcal H$ be a separable Hilbert space and let $h:\mbr^{2d}\to\mathcal H$ satisfy
\begin{align}\label{eq:levy-hilbert-feature}
 \|h(x,z)-h(y,w)\|_{\mathcal H}^2
 \le L_1c_2(x,y)+L_2c_2(z,w),\qquad x,y,z,w\in\mbr^d,
\end{align}
for constants $L_1,L_2\ge0$. As in Hypothesis \refH, define the map $H:\mbr^d\times\mcp_2(\mbr^d)\to\mcl{H}$ by
\[
 H(x,\mu)=\int_{\mbr^d}h(x,z)\,\mu(dz),
 \qquad (x,\mu)\in\mbr^d\times\mcP_2(\mbr^d).
\]
Consider a map into the L\'evy triplet space \eqref{def:levy-triplet-space},
\[
 K:\mbr^d\times\mathcal H\longrightarrow\mathfrak{L}_d,
 \qquad K(x,q)=(b(x,q),a(x,q),\Theta(x,q)).
\]
For each fixed $\mu$, the averaged feature field $x\mapsto H(x,\mu)$ determines the triplet field $x\mapsto K(x,H(x,\mu))$. The construction is summarized by
\[
 \mu
 \longmapsto
 \underbrace{[x\mapsto H(x,\mu)]}_{\text{averaged feature field}}
 \longmapsto
 \underbrace{[x\mapsto K(x,H(x,\mu))]}_{\text{L\'evy triplet field}}
 \longmapsto\mcA(\mu).
\]
The last step uses the L\'evy--Khintchine formula \eqref{eq:levy-mean-field-generator} with the composed triplet
\begin{align}\label{eq:hilbert-composed-triplet}
 \bigl(b(x,H(x,\mu)),a(x,H(x,\mu)),\Theta(x,H(x,\mu))\bigr).
\end{align}
We require the resulting generators to satisfy conditions (i)--(ii) of Hypothesis \refLa. Separately, to obtain the variance estimate, we impose the following quantitative condition on $K$.
Assume that, for some $\be_0,C_0\ge0$ and all $x,y\in\mbr^d$, $q,r\in\mathcal H$,
\begin{align}\label{eq:hilbert-triplet-condition}
 &(x-y)\cdot[b(x,q)-b(y,r)]
 +\mcW_{\mcS}(a(x,q),a(y,r))^2
 +\mcW_{\La}(\Theta(x,q),\Theta(y,r))^2\nonumber\\
 &\qquad\le\be_0c_2(x,y)+C_0\|q-r\|_{\mathcal H}^2.
\end{align}

\begin{example}[Averaged drift and noise amplitudes]\label{ex:averaged-levy-amplitudes}
Let $E$ be a standard Borel space equipped with a $\sigma$-finite reference measure $\Omega$. We use $E$ to parametrize jump displacements: a measurable map $\eta:E\to\mbr^d$ assigns a displacement to each jump parameter $u\in E$. By Lemma~\ref{lem:noise-amplitude-bounds}(ii), any $\eta\in L^2(E,\Omega;\mbr^d)$ pushes $\Omega$ forward to a L\'evy measure with finite second moment, after discarding mass mapped to zero. Thus $\Omega$ need not itself be a L\'evy measure.

We now allow the displacement map to depend on the particle position and the population law. Consider a vector field $\tilde b:\mbr^{2d}\to\mbr^d$, a matrix field $\tilde\sigma:\mbr^{2d}\to\mbr^{d\times m}$, and a jointly measurable map $\tilde\eta:\mbr^{2d}\times E\to\mbr^d$ with $\tilde\eta(x,z,\cdot)\in L^2(E,\Omega;\mbr^d)$ for each $x,z$. Define the averaged coefficients by
\begin{align*}
 b(x,\mu)&=\int_{\mbr^d}\tilde b(x,z)\,\mu(dz),\quad
 \sigma(x,\mu)=\int_{\mbr^d}\tilde\sigma(x,z)\,\mu(dz),\quad
 \eta(x,\mu,\cdot)=\int_{\mbr^d}\tilde\eta(x,z,\cdot)\,\mu(dz).
\end{align*}
For fixed $x,\mu$, the map $u\mapsto\eta(x,\mu,u)$ assigns the jump displacement to the parameter $u$. It is obtained by averaging the pairwise maps $\tilde\eta(x,z,\cdot)$ before taking the pushforward of $\Omega$.

Combining the drift, diffusion amplitude, and jump map, take
\[
 \mathcal H=\mbr^d\oplus\mbr^{d\times m}\oplus L^2(E,\Omega;\mbr^d),
 \qquad \|(v,\sigma,\eta)\|_{\mathcal H}^2
 =|v|^2+\|\sigma\|_{\mcl F}^2+\int_E|\eta(u)|^2\,\Omega(du),
\]
where $\|\cdot\|_{\mcl F}$ denotes the Frobenius norm. Assume that
\[
 h(x,z)=\bigl(\tilde b(x,z),\tilde\sigma(x,z),\tilde\eta(x,z,\cdot)\bigr)
\]
satisfies \eqref{eq:levy-hilbert-feature}, which ensures that the preceding averages are well-defined and that
$H(x,\mu)=(b(x,\mu),\sigma(x,\mu),\eta(x,\mu,\cdot))$. Take the triplet map
\[
 K(x,(v,\sigma,\eta))=(v,\sigma\sigma^\top,\eta_\sharp\Omega),
\]
which is independent of its explicit spatial argument. In particular, the mean-field triplet is given by
\[
 a(x,\mu)=\sigma(x,\mu)\sigma(x,\mu)^\top,
 \qquad \Theta(x,\mu)=\eta(x,\mu,\cdot)_\sharp\Omega,
\]
and the resulting generator is
\begin{align}\label{eq:averaged-amplitude-generator}
 [\mcA(\mu)\phi](x)
 ={}&b(x,\mu)\cdot\nabla\phi(x)
 +\frac12\tr[a(x,\mu)D^2\phi(x)]\nonumber\\
 &+\int_E\bigl[\phi(x+\eta(x,\mu,u))-\phi(x)
       -\eta(x,\mu,u)\cdot\nabla\phi(x)\bigr]\,\Omega(du).
\end{align}

Let us verify \eqref{eq:hilbert-triplet-condition}. For $q=(v,\sigma,\eta)$ and $r=(\tilde v,\tilde\sigma,\tilde\eta)$, Young's inequality gives
\[
 (x-y)\cdot(v-\tilde v)\le c_2(x,y)+\frac12|v-\tilde v|^2.
\]
Together with the diffusion and jump bounds of Lemma~\ref{lem:noise-amplitude-bounds}(i)--(ii), this yields \eqref{eq:hilbert-triplet-condition} with $\be_0=1$ and $C_0=\f 12$.
\end{example}

We now apply Theorem~\ref{thm:hilbert-poc} to L\'evy-type mean-field generators whose triplets depend on the population law through the averaged Hilbert-valued quantity $H(x,\mu)$, obtaining the following quantitative propagation-of-chaos estimate.

\begin{theorem}\label{thm:levy-hilbert-poc}
Assume that the generators defined by \eqref{eq:hilbert-composed-triplet} satisfy conditions (i)--(ii) of Hypothesis \refLa, and that \eqref{eq:levy-hilbert-feature} and \eqref{eq:hilbert-triplet-condition} hold with constants $L_1,L_2,\be_0,C_0\ge0$. For $N\ge2$, assume Hypothesis \refAp\ with $p=2$. Let $\{\br_t\}_{t\ge 0}$ be the $\mcW_2$-stable mean-field solution, and set $\Brho_t=\Brho_0e^{t\BHA}$ and $\Bbr_t=\br_t^{\otimes N}$. Define $V$ by \eqref{eq:hilbert-variance}. Then there exist constants $K,C>0$, depending only on the constants above, such that the following estimate holds for all $t\ge0$:
\begin{align}\label{eq:levy-hilbert-poc}
 \bs\mcW_2^2(\Brho_t,\Bbr_t)
 \le e^{Kt}\bs\mcW_2^2(\Brho_0,\Bbr_0)
 +\frac{C}{N-1}\int_0^t e^{K(t-s)}V(\br_s)\,ds.
\end{align}
In particular, if $\Brho_0=\br_0^{\otimes N}$, quantitative propagation of chaos holds with an $O(N^{-1})$ bound on the normalized quadratic transport cost, uniformly on finite time intervals.
\end{theorem}
\begin{proof}
By Theorem \ref{thm:hilbert-poc}, it suffices to verify Hypothesis \refH. The Lipschitz condition on $h$ is already given by \eqref{eq:levy-hilbert-feature}, so it remains to establish \eqref{eq:hilbert-generator-bound}. By Jensen's inequality and \eqref{eq:levy-hilbert-feature},
\begin{align*}
 \|H(x,\mu)-H(y,\nu)\|_{\mathcal H}^2
 \le2\|H(x,\mu)-H(x,\nu)\|_{\mathcal H}^2+2L_1c_2(x,y).
\end{align*}
Substituting $q=H(x,\mu)$ and $r=H(y,\nu)$ into \eqref{eq:hilbert-triplet-condition} therefore bounds the triplet comparison by
\[
 (\be_0+2C_0L_1)c_2(x,y)+2C_0\|H(x,\mu)-H(x,\nu)\|_{\mathcal H}^2.
\]
The same duality argument as in Proposition \ref{prop:levy-type-bdd}, using Corollary \ref{cor:general-duality-theta}, gives
\[
 \om_2(x,y;\mcA(\mu),\mcA(\nu))
 \le\be c_2(x,y)+2C_0\|H(x,\mu)-H(x,\nu)\|_{\mathcal H}^2,
\]
with $\be=\be_0+2C_0L_1$. The majorant is continuous with quadratic growth by \eqref{eq:hilbert-implies-wasserstein}, as required for the duality argument. Thus Hypothesis \refH\ holds, and Theorem \ref{thm:hilbert-poc} yields the claimed estimate.
\end{proof}